\documentclass[11pt,reqno]{amsart}

\usepackage[top=1in,bottom=1in,left=1.25in,right=1.25in]{geometry}
\usepackage{indentfirst}
\usepackage{amsmath,amssymb,amsthm,mathtools,mathrsfs}
\usepackage{enumitem}
\usepackage{xcolor}
\usepackage{graphicx}
\usepackage{float}
\usepackage[colorlinks=true,
  linkcolor=blue!45!black,
  citecolor=green!35!black,
  urlcolor=blue!55!black]{hyperref}
\usepackage[nameinlink,capitalize,noabbrev]{cleveref}

\allowdisplaybreaks
\numberwithin{equation}{section}

\newtheorem{theorem}{Theorem}[section]
\newtheorem{proposition}[theorem]{Proposition}
\newtheorem{lemma}[theorem]{Lemma}
\newtheorem{corollary}[theorem]{Corollary}
\newtheorem{conjecture}[theorem]{Conjecture}

\theoremstyle{definition}
\newtheorem{definition}[theorem]{Definition}

\theoremstyle{remark}

\crefname{equation}{equation}{equations}
\crefname{section}{Section}{Sections}
\crefname{subsection}{Section}{Sections}

\newcommand{\R}{\mathbb R}
\newcommand{\C}{\mathbb C}

\newcommand{\Z}{\mathbb Z}
\newcommand{\Ztwo}{\mathbb Z/2}
\newcommand{\cI}{\mathcal I}
\newcommand{\Euc}{\mathrm{Euc}}

\newcommand{\vol}{\operatorname{vol}}
\newcommand{\supp}{\operatorname{supp}}
\newcommand{\diver}{\operatorname{div}}
\newcommand{\dist}{\operatorname{dist}}
\newcommand{\diam}{\operatorname{diam}}

\newcommand{\Zero}{\operatorname Z}
\newcommand{\Mono}{\operatorname D}
\newcommand{\Freq}{\operatorname N}
\newcommand{\id}{\mathrm{id}}
\newcommand{\interior}{\operatorname{int}}

\newcommand{\eps}{\varepsilon}

\setlist[enumerate,1]{label=\textup{(\roman*)},leftmargin=2.4em}
\setlist[itemize,1]{leftmargin=2em}

\title[Examples of $\mathbb Z/2$-harmonic 1-forms]
{Examples of $\mathbb Z/2$-Harmonic 1-Forms}

\author[J. Chen]{Jiahuang Chen}
\address{Institute of Mathematics, Academy of Mathematics and Systems Science,
Chinese Academy of Sciences, Beijing 100190, P. R. China}
\email{chenjiahuang@amss.ac.cn}

\author[S. He]{Siqi He}
\address{Morningside Center of Mathematics, Chinese Academy of Sciences,
Beijing 100190, P. R. China}
\email{sqhe@amss.ac.cn}

\begin{document}

\begin{abstract}
In this paper, we develop methods for constructing
$\mathbb Z/2$-harmonic 1-forms in dimension three by varying
the background metric.
On $\mathbb R^3$, we construct an example for which the complement
of the smooth locus in the singular set is a Cantor set.
We realize every finite graph with positive even valence at each
vertex as the monodromy locus of a $\mathbb{Z}/2$ harmonic 1-form on $B^3$.
We also construct desingularization models for every critical
$\mathbb Z/2$-eigensection, agreeing exactly with the associated
homogeneous model outside a compact set.
Finally, we construct a nondegenerate $\mathbb Z/2$-harmonic 1-form
on every closed connected oriented three-manifold for a suitable
smooth metric.
The case of $S^3$ gives counterexamples to the conjectures in
\cite{HeWentworthZhangTrees,HeParkerConnectedSums}.
\end{abstract}

\maketitle

\section{Introduction}\label{sec:introduction}

$\Ztwo$-harmonic forms and spinors arise in Taubes' work on
compactness problems in low-dimensional gauge theory.
In several such theories, sequences of solutions with unbounded
fields can produce $\Ztwo$-harmonic forms or spinors after
normalization
\cite{TaubesPSL2C,HaydysWalpuskiCompactness,WalpuskiZhangCompactness}.
These objects play an important role in gauge theory and have many
interesting applications in geometry; see
\cite{HeSpecialLagrangian,FranceschiniMazzeoMinterMinimal,ParkerGluing}.
Donaldson \cite{Donaldson} developed a deformation theory for
nondegenerate multivalued harmonic functions.
Related deformation results for $\Ztwo$-harmonic forms and spinors
appear in
\cite{TakahashiModuli,ParkerDeformations,HeParkerWalpuskiUniversal}.

Examples have been constructed using wall-crossing
\cite{DoanWalpuskiExistence}, harmonic morphisms
\cite{HaydysMazzeoTakahashiExamples}, twistor methods
\cite{donaldson-twistor}, and gluing
\cite{HeParkerConnectedSums,YanShrinkingBranches}; see also \cite{ZhouQuadricBranching,LMZ2026}.
Calabi surgery provides a method for modifying and gluing such
examples \cite{ChenPerturbationPruning,ChenHeYanCalabiSurgery}.
In \cite{ChenHeSalmPrescribed}, every smooth link in $\R^3$ is
realized as the branching set of a nondegenerate $\Ztwo$-harmonic
1-form for a complete metric that is Euclidean outside a compact set.

Our motivation comes from Taubes' foundational 2014 work on the
zero loci of $\Ztwo$-harmonic forms and spinors
\cite{TaubesZeroLoci}.
We aim to construct examples that reflect the singular behavior
discussed in that paper and explore how complicated these singular
sets can be.
In this paper, we develop methods for constructing $\Ztwo$-harmonic
1-forms in dimension three by varying the background metric.

A $\Ztwo$-harmonic 1-form on a Riemannian three-manifold $X$ is a triple
$(Z,\cI,\alpha)$.
Here $Z\subset X$ is closed, and $\cI\to X\setminus Z$ is a flat real
line bundle with structure group $\{\pm1\}$.
The form $\alpha$ is a nonzero $\cI$-valued harmonic 1-form on
$X\setminus Z$ satisfying the finite-energy and growth conditions in
Definition~\ref{def:z2-harmonic}.

We write $\Zero(\alpha)=Z$ for the complete zero set.
The monodromy locus $\Mono(\cI)\subset Z$ consists of the points across
which $\cI$ does not extend trivially.
Points of $Z\setminus\Mono(\cI)$ are ordinary zeros of a locally defined
smooth harmonic 1-form.

For $p\in Z$, denote by $\Freq_p(0)$ the limiting value of Taubes'
frequency function.  Define
\begin{equation}\label{eq:CT-definition}
  C_{\mathrm T}
  :=\left\{
    p\in Z:
    q\mapsto\Freq_q(0)\text{ is continuous at }p,
    \quad
    \Freq_p(0)\in\frac12+\Z_{\geq0}
  \right\}.
\end{equation}
Taubes established fundamental regularity results for the zero sets of
$\Ztwo$-harmonic forms and spinors \cite{TaubesZeroLoci}.
In particular, $C_{\mathrm T}$ is a relatively open $\mathcal{C}^1$ curve
with $\overline{C_{\mathrm T}}=\Mono(\cI)$.
See also \cite{ZhangRectifiability} for further regularity results. Minimal submanifolds and laminations with complicated singular sets
have been constructed in
\cite{LiuGraphSingularSet,LiuFractalSingularSets,HoffmanWhitePrescribedBlowup}.

Our first example shows that $\Mono(\cI)\setminus C_{\mathrm T}$ can be
a closed Cantor set in $\R^3$.

\begin{theorem}\label{thm:intro-cantor}
There exist a complete smooth metric $g$ on
$\R^3=\R_t\times\C$ and a $\Ztwo$-harmonic 1-form
\[
  \alpha\in\Omega^1(\R^3\setminus Z;\cI)
\]
such that $g=g_{\Euc}$ and $\alpha=dt$ outside a compact set.
Moreover,
\[
  Z=\Mono(\cI),
  \qquad
  \Mono(\cI)\setminus C_{\mathrm T}=K_0,
\]
where $K_0\subset(-1,1)\times\{0\}$ is a closed Cantor set in $\R^3$.
\end{theorem}

At every $p\in K_0$, we have $\Freq_p(0)=1$, and the normalized
tangent cone is unique. In orthonormal coordinates $(t,x,y)$ on $T_p\R^3$, the tangent form is $\alpha_\infty=x\,dx-y\,dy,$ with zero set the $t$-axis. The limiting line bundle is trivial, and $\alpha_\infty$ extends smoothly across this axis.

Even when the monodromy locus is a smooth circle, the leading
branch coefficient can vanish on any prescribed closed subset
of the circle; see Appendix~\ref{sec:prescribed-degeneracy}.

$\Ztwo$-harmonic 1-forms with graph singularities are of particular
interest in the study of deformations of branching loci.
The linear analysis for forms and spinors branching along graphs
has been developed in \cite{HaydysMazzeoTakahashiIndex}.
Let $G\subset\interior B^3$ be an embedded finite graph.
Let $\cI\to B^3\setminus G$ be a flat real line bundle with structure
group $\{\pm1\}$.
If $\cI$ has holonomy $-1$ around every edge meridian, then each
vertex of $G$ has even valence.

\begin{theorem}\label{thm:intro-even-graphs}
For every finite graph $G$ with positive even valence at each vertex,
there exist a smooth metric $g$ on $B^3$ and a $\Ztwo$-harmonic 1-form
\[
  \alpha\in\Omega^1(B^3\setminus Z;\cI)
\]
such that $\Mono(\cI)=G$ for some embedding of $G$ into
$\interior B^3$.
\end{theorem}

The construction may also produce finitely many ordinary zero arcs
near the vertices.

$\Ztwo$ eigensections and the associated homogeneous models have been
studied by Taubes and Wu
\cite{TaubesWuModels,TaubesWuTopologicalEigenfunctions,TaubesWuPolytopes}.
Further work on critical eigensections appears in
\cite{ChenHeCritical,HeHaydysSalm}.  The graph analysis in \cite{HaydysMazzeoTakahashiIndex} and the Calabi
surgery in \cite{ChenHeYanCalabiSurgery} motivate the construction of
desingularization models for these homogeneous forms.

Let $P=\{p_1,\ldots,p_{2n}\}\subset S^2$, and let
$\cI_P\to S^2\setminus P$ have holonomy $-1$ around each puncture.
A nonzero finite-energy eigensection $\phi$ is called
\emph{critical} if both $|\phi|$ and $|d\phi|$ tend to zero at $P$.
Write
\[
  -\Delta_{S^2}\phi=\mu(\mu+1)\phi,
  \qquad
  F_{\mathrm{hom}}(r,\omega)=r^\mu\phi(\omega),
\]
where $\mu>0$.
Then $dF_{\mathrm{hom}}$ is a homogeneous
$\Ztwo$-harmonic 1-form on $\R^3$, with branch rays through $P$;
see \cref{subsec:tangent-cones-eigensections}.

Fix an arbitrary partition
\[
  \Pi=\bigl\{\{p_{a_1},p_{b_1}\},\ldots,
              \{p_{a_n},p_{b_n}\}\bigr\}
\]
of $P$ into pairs.
The following theorem resolves the vertex according to this pairing
and preserves the homogeneous model outside a compact set.

\begin{theorem}\label{thm:intro-desingularization}
For every critical eigensection $\phi$ and pairing $\Pi$ as above,
there exist a complete smooth metric $g$ on $\R^3$ and a
$\Ztwo$-harmonic 1-form
\[
  \alpha=dF\in\Omega^1(\R^3\setminus Z;\cI)
\]
whose monodromy locus is a disjoint union
$\Gamma=\gamma_1\sqcup\cdots\sqcup\gamma_n$
of smooth proper copies of $\R$.
For some $R>0$,
\[
  \gamma_k\cap\{|x|\geq R\}
  =\{rp_{a_k}:r\geq R\}\sqcup\{rp_{b_k}:r\geq R\},
  \qquad 1\leq k\leq n.
\]
After identifying the line bundles at infinity,
\[
  g=g_{\Euc}\quad\text{on }\R^3\setminus B_R,
  \qquad
  F=F_{\mathrm{hom}}
  \quad\text{on }(\R^3\setminus B_R)\setminus Z.
\]
\end{theorem}

\begin{figure}[H]
  \centering
  \includegraphics[width=.80\textwidth]{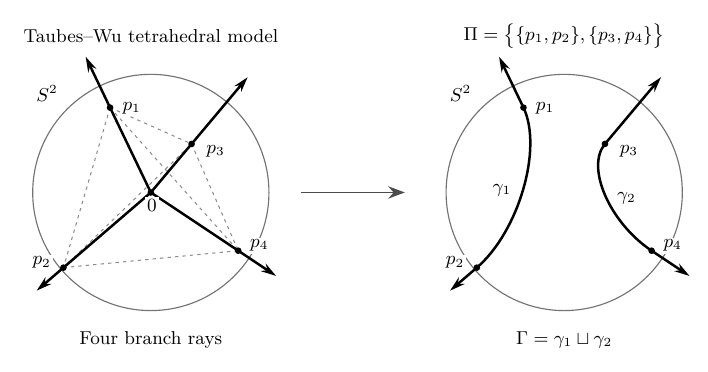}
  \caption{A pairing of the tetrahedral branch rays, giving
  $\Gamma=\gamma_1\sqcup\gamma_2$.}
  \label{fig:tetrahedral-pairing}
\end{figure}

We finally consider the existence of nondegenerate $\Ztwo$-harmonic
1-forms on closed three-manifolds.
One motivation comes from the role of $\Ztwo$-harmonic 1-forms in
compactifications of $\operatorname{SL}_2(\C)$ character varieties
\cite{TaubesPSL2C,HeWentworthZhangTrees}.
Haydys proved that, on an integral homology three-sphere, a smooth
zero set of a $\Ztwo$-harmonic 1-form must have at least two connected
components \cite{HaydysPSL2R}.
The existence question on $S^3$ is the subject of the following conjecture.

\begin{conjecture}[\cite{HeParkerConnectedSums,HeWentworthZhangTrees}]
\label[conjecture]{conj:sphere-nonexistence}
For every smooth Riemannian metric $g$ on $S^3$, there is no nonzero
$\Ztwo$-harmonic 1-form on $(S^3,g)$.
\end{conjecture}

He--Wentworth--Zhang prove nonexistence under additional hypotheses
\cite[Theorem~1.7]{HeWentworthZhangTrees}.
See also \cite{SunSeifertDichotomy} for nonexistence results for
certain metrics.

Allowing the metric to vary, we obtain the following existence
theorem.

\begin{theorem}\label{thm:intro-sphere}
Let $M$ be a closed connected oriented three-manifold.
Let $K_0,K_1\subset M$ be smooth knots contained respectively
in the interiors of two disjoint smoothly embedded closed
three-balls.
Then there exist a smooth metric $g$ on $M$ and a nondegenerate
$\Ztwo$-harmonic 1-form $(Z,\cI,\alpha)$ on $(M,g)$ such that
\[
  \Mono(\cI)=K_0\sqcup K_1.
\]
The set $Z\setminus\Mono(\cI)$ consists of finitely many
ordinary Morse zeros.
\end{theorem}

For $M=S^3$, the theorem gives the following corollary.

\begin{corollary}\label{cor:intro-sphere-split}
Let $L\subset S^3$ be a smooth split link with two components.
Then there exist a smooth metric $g$ on $S^3$ and a nondegenerate
$\Ztwo$-harmonic 1-form $(Z,\cI,\alpha)$ on $(S^3,g)$ such that $\Mono(\cI)=L.$
\end{corollary}

This disproves \cref{conj:sphere-nonexistence}.
The forms constructed on $S^3$ do not satisfy the transverse
distance condition in
\cite[Theorem~1.7(2)]{HeWentworthZhangTrees}, as explained in
\cref{subsec:sphere-s3}.

\subsection*{Acknowledgments}
We thank Rafe Mazzeo, Clifford Taubes, Thomas Walpuski and Dasen Yan
for helpful comments, discussions and encouragement.

\subsection*{Declaration of AI use}
We used GPT-5.5 Pro and GPT-5.6 Sol to assist with parts of the proofs.
The authors verified the mathematical arguments and rewrote the manuscript.
The authors take full responsibility for its content.

\section{Preliminaries on \texorpdfstring{$\Ztwo$}{Z/2}-harmonic 1-forms}
\label{sec:positive-pairs}

In this section, we collect the preliminaries needed for our constructions.
We recall the definition of $\Ztwo$-harmonic 1-forms and review the
basic results on their zero sets and homogeneous models.
We then give criteria for choosing the background metric so that
a prescribed closed 1-form is harmonic.

\subsection{\texorpdfstring{$\Ztwo$}{Z/2}-harmonic 1-forms}

Let $(X,g)$ be a smooth Riemannian three-manifold, let $Z\subset X$
be closed, and let $\cI\to X\setminus Z$ be a real line bundle with
structure group $\{\pm1\}$.
The flat connection on $\cI$ extends the exterior derivative $d$ to
$\cI$-valued forms.
We write $d^*$ for its formal adjoint with respect to $g$.
We denote by $\nabla$ the covariant derivative induced by the flat
connection on $\cI$ and the Levi--Civita connection of $g$.
We use the convention $\Delta_g=-d^*d$ on functions.

For a double branched cover $\pi:\widetilde X\to X$ associated with
$\cI$, we write $\tau$ for its involution.
A $\cI$-valued form $\alpha$ corresponds to a form $\widetilde\alpha$
on $\pi^{-1}(X\setminus Z)$ with
$\tau^*\widetilde\alpha=-\widetilde\alpha$.

\begin{definition}\label[definition]{def:z2-harmonic}
A $\Ztwo$-harmonic $1$-form on a connected Riemannian three-manifold
$(X,g)$ is a triple $(Z,\cI,\alpha)$ satisfying:
\begin{enumerate}
  \item $Z\subsetneq X$ is closed and
  $\cI\to X\setminus Z$ is a real line bundle with structure group
  $\{\pm1\}$;
  \item $\alpha\in\Omega^1(X\setminus Z;\cI)$ satisfies harmonic equation $d\alpha=0,\;d^*\alpha=0;$
  \item $|\alpha|_g$ extends continuously to $X$ and vanishes precisely on
  $Z$;
  \item for every $U\Subset X$, $\int_{U\setminus Z}|\nabla\alpha|_g^2\,\vol_g<\infty;$
  \item for every $U\Subset X$, there are constants
  $c_U,r_U,\delta_U>0$ such that $\int_{B_r(p)}|\alpha|_g^2\,\vol_g
      \leq c_Ur^{3+\delta_U}$
  for every $p\in Z\cap U$ and every $0<r<r_U$.
\end{enumerate}
\end{definition}

For an open set $U\subset X$, we say that $\alpha$ is \emph{exact on $U$}
if there exists a smooth section $F\in \mathcal{C}^\infty(U\setminus Z;\cI)$ such that $\alpha=dF$ on $U\setminus Z$.  When $U=X$, we simply say that
$\alpha$ is \emph{exact}. We write $\Zero(\alpha)=Z$.

\begin{definition}\label[definition]{def:monodromy-locus}
The monodromy locus $\Mono(\cI)\subset\Zero(\alpha)$ consists of the points
$p\in\Zero(\alpha)$ for which $\cI|_{U\setminus\Zero(\alpha)}$
is nontrivial for every neighborhood $U$ of $p$.
\end{definition}

Clearly, $\Mono(\cI)\subset\Zero(\alpha)$.  If
$p\in\Zero(\alpha)\setminus\Mono(\cI)$, then $\cI$ is trivial over
$U\setminus\Zero(\alpha)$ for some neighborhood $U$ of $p$.  In this
trivialization, the standard local extension and elliptic regularity argument
shows that $\alpha$ extends smoothly across $p$ as an ordinary harmonic
1-form.

\subsection{Frequency and regularity}
We now summarize Taubes' results on the frequency function and the regularity of the singular set \cite{TaubesZeroLoci}.  

Put $Z=\Zero(\alpha)$.  For $p\in Z$ and $r$ smaller than the injectivity
radius at $p$, let $B_r(p)$ be the geodesic ball and set
\begin{equation}\label{eq:frequency-integrals}
 H_p(r)=\int_{\partial B_r(p)}|\alpha|_g^2\,\mathrm dA_g,
 \qquad
 D_p(r)=\int_{B_r(p)\setminus Z}|\nabla\alpha|_g^2\,\mathrm dV_g.
\end{equation}
By \cite[Lemma~3.1]{TaubesZeroLoci}, $H_p(r)>0$ for all sufficiently
small $r>0$.  Define
\begin{equation}\label{eq:frequency-function}
 \Freq_p(r):=\frac{rD_p(r)}{H_p(r)}.
\end{equation}
This agrees with Taubes' frequency up to a multiplicative factor
$1+O(r^2)$ \cite[equations~(3.2)--(3.3)]{TaubesZeroLoci}.
Thus \cite[Lemma~3.2]{TaubesZeroLoci} gives the finite limit
\[
 \Freq_p(0):=\lim_{r\to0}\Freq_p(r),
\]
which we call the \emph{frequency at $p$}.

A point $p\in Z$ is a \emph{frequency-continuity point} if
$q\mapsto\Freq_q(0)$ is continuous at $p$ as a function on $Z$. Now we summarize some regularity results for $\Ztwo$-harmonic 1-forms
\cite{TaubesZeroLoci,ZhangRectifiability}.
\begin{proposition}
\label[proposition]{prop:taubes-frequency-package}
Let $(Z,\cI,\alpha)$ be a $\Ztwo$-harmonic 1-form on $(X,g)$.
Then the following hold.
\begin{enumerate}
  \item $\Freq_p(0)\geq1/2$ for every $p\in Z$, and
  $p\mapsto\Freq_p(0)$ is upper semicontinuous on $Z$
  \cite[Lemmas~3.4 and~6.4]{TaubesZeroLoci}.
  \item $Z$ is $1$-rectifiable and has Hausdorff dimension at most one
  \cite{ZhangRectifiability}.
  \item The frequency-continuity points form a relatively open dense
  subset of $Z$.  At each such point, $q\mapsto\Freq_q(0)$ is locally
  constant on $Z$ \cite[Lemma~10.3 and the proof of Theorem~1.4]{TaubesZeroLoci}.
  Near such a point, $Z$ is contained in a codimension-two Lipschitz
  graph.  If its frequency is an odd half-integer, then a neighborhood
  of the point in $Z$ is an embedded $\mathcal{C}^1$ curve
  \cite[Proposition~10.1]{TaubesZeroLoci}.
  \item $C_{\mathrm T}$ defined in $\eqref{eq:CT-definition}$ is relatively open in $Z$ and is an embedded
  $\mathcal{C}^1$ curve in $X$, with
  $\Mono(\cI)=\overline{C_{\mathrm T}}$
  \cite[Proposition~10.5]{TaubesZeroLoci}.
  \item For all but at most countably many $p\in Z$,
  $\Freq_p(0)\in\tfrac12\mathbb Z_{>0}$
  \cite[Lemmas~6.1 and~6.6]{TaubesZeroLoci}.
\end{enumerate}
\end{proposition}

\subsection{Critical eigensections and tangent cones}
\label{subsec:tangent-cones-eigensections}

The homogeneous branch models used below arise from critical
$\Ztwo$-eigensections on $S^2$.  We first describe this construction and then
recall Taubes' existence theorem for tangent cones.

Equip $S^2$ with the unit round metric.
Let $P=\{p_1,\dots,p_{2n}\}\subset S^2$ be nonempty, and let
$\cI_P\to S^2\setminus P$ have holonomy $-1$ around every puncture.
Let $\phi$ be a nonzero smooth section of $\cI_P$ satisfying
\begin{equation}\label{eq:spherical-eigensection}
  -\Delta_{S^2}\phi=\lambda\phi
\end{equation}
and
\[
  \int_{S^2\setminus P}\bigl(|\phi|^2+|d\phi|^2\bigr)<\infty.
\]
The local expansion theory of Taubes--Wu
\cite{TaubesWuTopologicalEigenfunctions} gives, in a complex coordinate
$z_j$ centered at $p_j$,
\[
  \phi(z_j)
  =\operatorname{Re}\bigl(a_jz_j^{m_j+1/2}\bigr)
   +O\bigl(|z_j|^{m_j+3/2}\bigr),
  \qquad
  m_j\in\Z_{\geq0},
  \qquad
  a_j\neq0.
\]
Equivalently, the convergent expansion factors as
\begin{equation}\label{eq:critical-puncture-factorization}
  \phi(z_j)
  =\operatorname{Re}\left(
    z_j^{m_j+1/2}A_j(z_j,\bar z_j)
  \right),
  \qquad
  A_j(0)\neq0.
\end{equation}
Here $A_j$ is a smooth single-valued complex function; after shrinking the
coordinate disk, we assume that it is nowhere zero.
We call $\phi$ a \emph{critical eigensection} if $m_j\geq1$ for every $j$.  In this case,
both $|\phi|$ and $|d\phi|$ extend continuously by zero at the punctures.

\begin{lemma}\label[lemma]{lem:homogeneous-exponent}
If $\phi$ is a critical eigensection satisfying
\eqref{eq:spherical-eigensection}, then $\lambda>2$.
\end{lemma}

\begin{proof}
Near each $p_j$, write $\rho=|z_j|$.
The local expansion and criticality give
\[
  |\phi|=O(\rho^{3/2}),\qquad
  |d\phi|=O(\rho^{1/2}),\qquad
  |\nabla^2\phi|=O(\rho^{-1/2}).
\]
Thus $\nabla^2\phi$ is square-integrable.
Integrate the eigenvalue equation against $\phi$ over
$S^2\setminus\bigcup_j B_\eps(p_j)$.
Integrate the Bochner identity over the same domain.
The boundary terms are $O(\eps^3)$ and $O(\eps)$, respectively.
Since $\operatorname{Ric}_{S^2}=g_{S^2}$, letting $\eps\to0$ gives
\[
  \int_{S^2}|d\phi|^2=\lambda\int_{S^2}|\phi|^2,
  \qquad
  \int_{S^2}|\nabla^2\phi|^2
    =(\lambda-1)\int_{S^2}|d\phi|^2.
\]
The trace inequality in dimension two therefore yields
\[
  (\lambda-1)\int_{S^2}|d\phi|^2
  \geq\frac12\int_{S^2}(\Delta_{S^2}\phi)^2
  =\frac\lambda2\int_{S^2}|d\phi|^2.
\]
The integral of $|d\phi|^2$ is positive: otherwise $\phi$ would be
covariantly constant, and its vanishing at $P$ would force $\phi=0$.
Hence $\lambda\geq2$.

If $\lambda=2$, equality holds in the trace inequality, so
\[
  \nabla^2\phi=-\phi g_{S^2}.
\]
It follows that $|d\phi|^2+|\phi|^2$ is constant.
Both terms tend to zero at each puncture, so this constant is zero.
This contradicts $\phi\neq0$, and therefore $\lambda>2$.
\end{proof}

A critical eigensection determines a homogeneous model on $\R^3$.
Let $\mu>1$ be determined by $\mu(\mu+1)=\lambda$, and define
\begin{equation}\label{eq:homogeneous-eigencone}
  F_\infty(r,\omega):=r^\mu\phi(\omega),
  \qquad
  v_\infty:=dF_\infty.
\end{equation}
Indeed,
\[
  \Delta_{\R^3}F_\infty
  =r^{\mu-2}\bigl(\mu(\mu+1)-\lambda\bigr)\phi=0,
\]
and
\[
  |v_\infty|^2
  =r^{2\mu-2}\bigl(\mu^2|\phi|^2+|d\phi|^2\bigr).
\]
Let $Z_\phi=\Zero(v_\infty)$, and let $\cI_\phi$ be the restriction of the
radial pullback of $\cI_P$ to $\R^3\setminus Z_\phi$.  Criticality and
\cref{lem:homogeneous-exponent} imply that
$(Z_\phi,\cI_\phi,v_\infty)$ is a $\Ztwo$-harmonic 1-form on $\R^3$.
For the dilation $D_s(x)=sx$,
\[
  D_s^*v_\infty=s^\mu v_\infty,
  \qquad
  |v_\infty|(sx)=s^{\mu-1}|v_\infty|(x).
\]
Thus $v_\infty$ is homogeneous with frequency $\mu-1$; its branch rays are
the rays through $P$, and it vanishes along these rays and at the vertex.
The critical eigensections above provide the homogeneous vertex models used
in the gluing constructions below.  A critical configuration with $2n$
punctures gives a cone with $2n$ branch rays and hence a local model for a
vertex of valence $2n$.

The following existence result is proved in
\cite{TaubesWuModels,ChenHeCritical}.
\begin{theorem}
\label[theorem]{thm:critical-configurations}
For every $n>1$, there exist infinitely many configurations
$P=\{p_1,\dots,p_{2n}\}\subset S^2$ such that $\cI_P$ admits a critical
eigensection.
\end{theorem}

Taubes and Wu also use eigensections on the complements of graphs in
$S^3$ to construct homogeneous $\Ztwo$-harmonic forms and spinors on
$\R^4$ \cite{TaubesWuPolytopes}.

Let $(Z,\cI,\alpha)$ be a $\Ztwo$-harmonic 1-form on $(X,g)$ and
let $p\in Z$.
A tangent cone at $p$ is a subsequential limit of rescalings centered
at $p$ in geodesic coordinates.
Each rescaled form is normalized to have unit $L^2$ norm on the unit
sphere with respect to the rescaled metric.
Such a limit exists by \cite[Proposition~4.1]{TaubesZeroLoci}.
It is a nonzero homogeneous $\Ztwo$-harmonic 1-form
$(Z_\infty,\cI_\infty,\alpha_\infty)$ on $\R^3$ satisfying
\[
  D_s^*\alpha_\infty=s^{\Freq_p(0)+1}\alpha_\infty.
\]

\subsection{Metric realization}

Let $X$ be an oriented three-manifold.  We now choose a metric for which
a given closed $\cI$-valued 1-form $\alpha$ is harmonic.  The construction
uses a closed $\cI$-valued 2-form $\beta$ with $\alpha\wedge\beta>0$ and
preserves any prescribed metric satisfying $*_{g_0}\alpha=\beta$.

\subsubsection{The metric from a positive dual form}

\begin{theorem}\label{thm:relative-metric-realization}
Let $Z,K\subset X$ be closed, and let $\cI\to X\setminus Z$ be a flat
real line bundle with structure group $\{\pm1\}$.  Suppose that smooth forms
$\alpha\in\Omega^1(X\setminus Z;\cI)$ and
$\beta\in\Omega^2(X\setminus Z;\cI)$ satisfy
\[
  d\alpha=0,\qquad d\beta=0,\qquad \alpha\wedge\beta>0.
\]
Here $\alpha\wedge\beta$ is an ordinary 3-form, and positivity refers to
the orientation of $X$.
Let $U$ be an open neighborhood of $Z\cup K$, and let $g_0$ be a smooth
metric on $U$ satisfying
\[
  *_{g_0}\alpha=\beta\qquad\text{on }U\setminus Z.
\]
Then there are an open neighborhood $V$ of $Z\cup K$, with
$\overline V\subset U$, and a smooth metric $g$ on $X$ such that
\[
  g=g_0\quad\text{on }V,\qquad
  *_g\alpha=\beta\quad\text{on }X\setminus Z.
\]
In particular, $\alpha$ is harmonic with respect to $g$ on $X\setminus Z$.
\end{theorem}

\begin{proof}
Choose $V$ with $Z\cup K\subset V$ and $\overline V\subset U$.
Choose an auxiliary metric $g_{\mathrm{aux}}$ on $X$ that agrees with
$g_0$ on $V$.  In a parallel local frame $e$ of $\cI$, write
$\alpha=a\otimes e$ and $\beta=b\otimes e$.
Since $a\wedge b>0$, the line $\ker b$ is transverse to $\ker a$.
There is therefore a unique vector field $v$ satisfying
\[
  \iota_vb=0,\qquad a(v)=1.
\]
Define $H:=\ker a$ and orient $H$ so that $v$ followed by a positive basis
of $H$ is positively oriented in $X$.
Let $dA_{g_{\mathrm{aux}},H}$ be the area form of the restricted metric
on $H$.  Write $b|_H=\lambda\,dA_{g_{\mathrm{aux}},H}$, where $\lambda>0$,
and define $L:=|a|_{g_{\mathrm{aux}}}^{-1}$.
Define $g$ by
\[
  v\perp_g H,\qquad g(v,v):=L^2,\qquad
  g|_H:=L\lambda\,g_{\mathrm{aux}}|_H.
\]
The $g$-area form on $H$ is $L\,b|_H$, and $|a|_g=L^{-1}$.
Hence $*_ga=b$.

Replacing $e$ by $-e$ changes $(a,b,v)$ to $(-a,-b,-v)$ and reverses the
orientation of $H$.
The quantities $L$, $\lambda$, and $g$ remain unchanged.
These formulas therefore define a smooth metric on $X\setminus Z$.
The metric depends smoothly on $\alpha$, $\beta$, and $g_{\mathrm{aux}}$.

If $b=*_{g_{\mathrm{aux}}}a$, then
\[
  v=\frac{a^{\sharp_{g_{\mathrm{aux}}}}}{|a|_{g_{\mathrm{aux}}}^2},
  \qquad \lambda=|a|_{g_{\mathrm{aux}}}=L^{-1}.
\]
The defining formulas then give $g=g_{\mathrm{aux}}$.
Thus $g=g_0$ on $V\setminus Z$, and $g$ extends smoothly across $Z$ by
defining $g=g_0$ there.

Finally, closedness of $\beta$ gives
\[
  d_g^*\alpha=-*_gd*_g\alpha=-*_gd\beta=0.
\]
Together with $d\alpha=0$, this proves harmonicity.
\end{proof}

When $\alpha=dF$ and $\beta=d\lambda$, closedness is automatic.
We use the following terminology for their positivity condition.

\begin{definition}\label[definition]{def:positive-pair}
Let $F\in\mathcal{C}^\infty(X\setminus Z;\cI)$ and
$\lambda\in\Omega^1(X\setminus Z;\cI)$.
We call $(F,\lambda)$ a \emph{positive pair} if
\[
  dF\wedge d\lambda>0\qquad\text{on }X\setminus Z.
\]
\end{definition}

\subsubsection{Inverse metric densities}

Let $U\subset X$ be an oriented coordinate chart with
coordinates $x=(x^1,x^2,x^3)$, and use the summation convention.
Write $g=g_{ij}\,dx^i\otimes dx^j$.
Define $|g|:=\det(g_{ij})$ and $(g^{ij}):=(g_{ij})^{-1}$.
The \emph{inverse metric density} of $g$ is the symmetric positive
definite matrix $A=(A^{ij})$ defined by
\[
  A^{ij}:=\sqrt{|g|}\,g^{ij}.
\]
In dimension three, $\det A=\sqrt{|g|}$.
Conversely, a symmetric positive definite matrix $A$ determines the
metric
\begin{equation}\label{eq:metric-from-density}
  (g_{ij}):=(\det A)A^{-1}.
\end{equation}

For a 1-form $\xi=\xi_j\,dx^j$, define
\[
  \diver(A\xi):=\partial_i(A^{ij}\xi_j).
\]
Here $A\xi$ is a vector density, and $\diver$ is its divergence in the
coordinates $x$.
On a coordinate ball in $U\setminus Z$, choose a parallel frame of
$\cI$ and write $\alpha=du$.
Then
\[
  d_g^*du=-\frac{1}{\det A}\diver(A\,du).
\]
Thus choosing $A$ with $\diver(A\,du)=0$ is the coordinate form of
requiring the 2-form $\beta=*_gdu$ to be closed.

\section{A Euclidean-ended Cantor accumulation model}
\label{sec:cantor-accumulation}

This section constructs the Euclidean-ended Cantor model in
\cref{thm:intro-cantor}.  We start from an ordinary function $F$ with paired quadratic
zeros. Then replace each pair by two shrinking branch ellipses and obtain a $2$-valued function. Finally, we construct a smooth metric on \(\R^3\) so that the new function is harmonic.

\subsection{The Euclidean model}

We construct the ordinary function $F$ on \(\R^3\) from a function \(h\) of one variable.
Fix the middle-thirds Cantor set $K\subset[-1,1]$ centered at the origin, so that the first removed interval is $I_1=(-1/3,1/3)$.  
The remaining bounded components of $\R\setminus K$ are denoted by $I_j=(a_j,b_j)$ with length $\ell_j=b_j-a_j$ ($j\geq2$). 
The function \(h\) is given by the following lemma.

\begin{lemma}\label[lemma]{lem:straight-cantor-profile}
There is a function $h\in \mathcal{C}^\infty(\R)$ with the following properties.
\begin{enumerate}
  \item Function $h$ and all its derivatives vanish on $K$.
  
  \item Each interval $I_j$ contains exactly two zeros $q_j^-<q_j^+$ of $h$. Moreover, $h$ is
  linear near $q_j^\pm$, and $h'(q_j^\pm)=\pm s_j$ for some $s_j>0$.

  \item $h\equiv 1$ near infinity and $\int_{\R}(h-1)\,dt=0$.
  \item There is smooth function $H$ such that $H'=h$ and $H|_K\equiv 0$. Moreover, for some constant $C_0$,
    $G(t):=H(t)-(t+C_0)$ is compactly supported.
\end{enumerate}
\end{lemma}

\begin{proof}
Choose a smooth function $\varphi$ on $[0,1]$ which is flat at both
endpoints, symmetric about $1/2$, has exactly two simple zeros
$q^-<q^+$, is linear near these zeros with \(\varphi'(q^\pm)=\pm s\), and has integral \(\int_0^1\varphi\,dt=0\).

On $I_j$ put
 $ h(t)=A_j\varphi\left(({t-a_j})/{\ell_j}\right)$, where we choose $A_j>0$ super-polynomially small in $\ell_j$, i.e., fix any $m>0$ we have $A_j\ell_j^{-m}\to 0$ as $j\to\infty$. Thus the derivative estimate
\[
  \|A_j\varphi((\,\cdot-a_j)/\ell_j)\|_{\mathcal{C}^m(I_j)}
  \leq C_mA_j\ell_j^{-m}
\]
shows that the
function $h$ extends smoothly by zero across $K$ with vanishing $m$-jet for each $m$. In each $I_j$ the two zeros are
$q_j^\pm=a_j+\ell_jq^\pm$.  The integral of $h$ over every $I_j$ is zero.
Extend $h$ positively outside $(-1,1)$ so that it equals one near
infinity and $\int_{\R}(h-1)\,dt=0$.

Choose a primitive $H$ that vanishes at one point of $K$.  The Cantor set has
Lebesgue measure zero, and the integral over every $I_j$ vanishes.
Consequently $H|_K=0$.  Since $(H-t)'=h-1$ is compactly supported, $H-t$ is constant at
both ends of $\R$.  The difference of the two constants is $\int(h-1)=0$. Denote this constant by $C_0$ proves the last assertion.
\end{proof}
We call the numbers $A_j$ in the proof the \emph{amplitudes} of $h$ on $I_j$.  They may be decreased independently in
the later diagonal choice without moving the zeros $q_j^\pm$ or
affecting the conclusions of the lemma.

Put $K_0=K\times\{(0,0)\}$ and $p_j^\pm=(q_j^\pm,0,0)$ in \(\R^3\) with coordinates $(t,x,y)$.
We now use $h$ to construct an ordinary harmonic function on $\R^3$
with paired zeros.  Choose
$B_*>1+\frac12\|h'\|_{\mathcal{C}^0}$ and a smooth function $b_1\geq1$ that
equals $B_*$ on a bounded interval containing $(-1,1)\cup\supp h'$
and equals one outside a larger interval.  In each $I_j$, choose a
nonnegative bump $\beta_j$ equal to $s_j/2$ near $q_j^-$ and supported
away from $q_j^+$, and set $b=b_1+\sum_j\beta_j$ and $c=b+\frac12h'$. The choice of the amplitudes \(A_j\) gives
$s_j\ell_j^{-m}=sA_j\ell_j^{-m-1}\to0$ for every $m$. Thus \(\sum_j\beta_j\) converges in every $\mathcal{C}^m$ norm, and $b$ is smooth.
Then $b$ and $c$ are positive, equal $B_*$ on $K$, equal one near
infinity, and are constant near every $q_j^\pm$.  Moreover,
\[
  (b,c)|_{q_j^+}=(B_*,B_*+s_j/2),
  \qquad
  (b,c)|_{q_j^-}=(B_*+s_j/2,B_*).
\]

 Define $F_0(t,x,y)=H(t)+x^2-y^2$ and
$A_0(t)=\operatorname{diag}(1,b(t),c(t))$ on \(\R^3\).  Then
$\diver(A_0\nabla F_0)=h'+2b-2c=0$, so the metric
$g_0=(\det A_0)A_0^{-1}$ makes $dF_0$ harmonic.  Thus the zero set of $dF_0$ is
$K_0\cup\{p_j^\pm:j\geq1\}$.
Near each paired zero $q_j^\pm$, $h$ is linear and $b,c$ are constant. Let $g_j^\pm$ denote the constant metrics, which are values of $g_0$ near $p_j^\pm$. 

\begin{figure}[htbp]
  \centering
  \includegraphics[width=.70\textwidth]{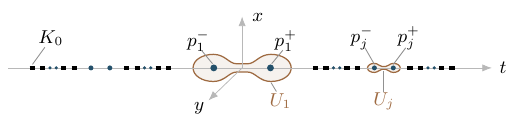}
  \caption{$K_0\cup\{p_j^\pm:j\geq1\}$ on the
  $t$-axis.}
  \label{fig:cantor-background}
\end{figure}

Near \(p_j^\pm\), in the
coordinates $\tau=t-q_j^\pm$, the function $F_0$ is given by quadratic
polynomials that are harmonic with respect to the constant metrics \(g_j^\pm\)
\begin{equation}\label{eq:paired-quadratics}
  P_j^\pm(\tau,x,y)=\pm A_j\ell_j\int_0^{q^+}\varphi\,dt\pm\frac{s_j}{2}\tau^2+x^2-y^2.
\end{equation}
In coordinates \((\tau,x,y)\), the orientation-preserving linear map $Q(\tau,x,y)=(-\tau,y,x)$ makes \(P_j^-=-P_j^+\circ Q\) and \(Q^*g_j^+=g_j^-\).

We now modify \(F_0\) further to obtain the ordinary function \(F\) on \(\R^3\), together with a metric with respect to which \(F\) is harmonic.

\begin{proposition}
  \label[proposition]{prop:euclidean-ended-background}
There are a complete smooth metric $\bar{g}$ on $\R^3$, a smooth
function $F$, and sufficiently large $R$ such that $ \bar{g}=g_{\Euc}$ and $dF=dt$
  on $\R^3\setminus B_R$, and such that 
$dF$ is $\bar{g}$-harmonic. Moreover, the zero set of \(dF_0\) is exactly
$K_0\cup\{p_j^\pm:j\geq1\}$, and near this set we have $F=F_0$ and $\bar{g}=g_0$.
\end{proposition}
\begin{proof}
  Define the following vector fields on \(\R^3\)
  \[V_G=\left((b+c-2)xy,-\tfrac12(h-1)y,\tfrac12(h-1)x\right) \text{ and } V_Q=(2xy,0,0).\]
Using $c=b+h'/2$, direct calculation gives
\[
  \operatorname{curl}V_G=(h-1,2(b-1)x,-2(c-1)y) \text{ and } \operatorname{curl}V_Q=(0,2x,-2y).
\]

Let \(\rho=\sqrt{x^2+y^2}\). Choose $1\ll R_0<R_1$ and let $\chi_G$ be a cut-off function such that \(\chi_G(\rho)\equiv 1\) for
$\rho\leq R_0$ and \(\chi_G(\rho)\equiv 0\) for $\rho\geq R_1$.  We can choose \(\chi_G\) so that
$|\rho\chi_G'|\le C/\log(R_1/R_0)$, taking $R_1/R_0$ large makes its derivative as small as needed.

Next choose $R_3>R_2>R_1$ and choose cut-off function \(\chi_Q\) such that \(\chi_Q(\rho)\equiv 1\) for $\rho\leq R_2$ and \(\chi_Q(\rho)\equiv 0\) for $\rho\geq R_3$.  
With \(R_3\) sufficiently large, we can choose \(\chi_Q\) so that
$2(\rho^2\chi_Q)(\rho^2\chi_Q)'/\rho>-1/4$ for $\rho>0$. 

Finally, choose cut-off function $\zeta(t)$ equal to one near
$\supp G\cup\supp(h-1)\cup\supp(b-1)$, such that
$\|\zeta'\|_{\mathcal{C}^0}\sup_{\rho\geq0}\rho^2\chi_Q(\rho)<1/4$.
Define 
\begin{equation}\label{eq:def-F}
F=t+C_0+\chi_G(\rho)G(t)
+\zeta(t)\chi_Q(\rho)(x^2-y^2)
\end{equation}
and vector field
$X=\partial_t+\operatorname{curl}(\chi_GV_G)
+\operatorname{curl}(\zeta\chi_QV_Q)$.

Set $\alpha=dF$ and $\beta=\iota_X(dt\wedge dx\wedge dy)$.  Both forms are
closed.  On the region where $\rho\leq R_0$ and $\zeta(t)=1$, $\alpha=dF_0$, $X=A_0\nabla F_0=(h,2bx,-2cy)$, and
$\beta=\star_{g_0}\alpha$.  Outside a bounded cylinder,
$\alpha=dt$ and $\beta=dx\wedge dy$.

Next we prove that $\alpha(X)>0$ away from $K_0\cup\{p_j^\pm:j\ge 1\}$.  When \(\rho\le R_0\) and \(\zeta=1\), this follows from
$\alpha(X)=h^2+4bx^2+4cy^2$.  When $\zeta=1$ but
$R_0<\rho<R_1$, we have $\chi_Q=1$ and direct expansion gives
\[
\begin{aligned}
\alpha(X)=&(1+\chi_G(h-1))^2
 +4(1+\chi_G(b-1))x^2+4(1+\chi_G(c-1))y^2\\
&+\rho\chi_G'\big(
 {(h-1)(1+\chi_G(h-1))}/{2}+{4(b+c-2)x^2y^2}/{\rho^2}\big)\\
&+\frac{2\chi_G'G}{\rho}
 \big((1+\chi_G(b-1))x^2-(1+\chi_G(c-1))y^2\big).
\end{aligned}
\]
The first line is at least
$4\min\{1,\inf b,\inf c\}\rho^2$. In particular, it is positive.  Since all coefficients depending on $t$
are bounded, the absolute value of the last two lines is at most
$C|\rho\chi_G'|(1+\rho^2)$.  Here $\rho>R_0>1$, so taking
$|\rho\chi_G'|$ sufficiently small makes them smaller than the first
line.  Hence $\alpha(X)>0$ on this region.

At every remaining point, either $\rho\geq R_1$ or $\zeta\ne1$.  In the first
case $\chi_G=\chi_G'=0$; in the second, the support condition on $\zeta$
implies $G=h-1=b-1=c-1=0$.  Hence all terms of \(dF\) involving $G$ and $V_G$ vanish,
and
\[
  \alpha(X)=1+\zeta'\chi_Q(x^2-y^2)
  +{2\zeta^2(\rho^2\chi_Q)(\rho^2\chi_Q)'}/{\rho}>\frac12.
\]
 The absolute
value of the second term of the right hand side is less than $1/4$ by the construction of \(\zeta\) and the fact that \(|x^2-y^2|\le \rho^2\). 
And the last term is greater than $-1/4$.  

Note that
$\alpha\wedge\beta=\alpha(X)\,dt\wedge dx\wedge dy$, thus \(\alpha\wedge\beta\) is positive awyay from 
$K_0\cup\{p_j^\pm:j\ge 1\}$. In particular, the zero set of $\alpha=dF$ is exactly $K_0\cup\{p_j^\pm:j\ge 1\}$.

Apply \cref{thm:relative-metric-realization} to the pair $(\alpha,\beta)$, taking $Z=K_0\cup\{p_j^\pm:j\ge1\}$, retaining $g_0$ near $Z$, and prescribing the Euclidean metric on the Euclidean end, where $\alpha=dt$ and $\beta=dx\wedge dy$. This yields a metric
$\bar{g}$ on $\R^3$ satisfying $\star_{\bar{g}}\alpha=\beta$.
Since $\alpha$ and $\beta$ are closed, $\alpha=dF$ is harmonic.  The metric
equals $g_0$ near the zero set $K_0\cup\{p_j^\pm:j\ge 1\}$ and is Euclidean
outside a compact set, so it is complete.
\end{proof}

\subsection{Paired branch surgery}

The following theorem assigns to each nondegenerate homogeneous quadratic polynomial a $\mathbb{Z}/2$-harmonic function in $\R^3$ 
whose singular set is an ellipse.
 See \cite{YanShrinkingBranches,donaldson-twistor} for two different constructions 
 and \cite{LMZ2026} for the rigidity of these constructions.  
 
\begin{theorem}\label{thm:compact-branch-model}
Let $P$ be a nondegenerate homogeneous
quadratic polynomial that is harmonic on $(\R^3,g_{\mathrm{Euc}})$.  
There is a nondegenerate, exact $\Z/2$-harmonic $1$-form $(\Sigma,\cI_{\Sigma},\alpha)$ on $(\R^3,g_{\mathrm{Euc}})$ associated with $P$,
 where $\Sigma\subset\R^2\subset \R^3$ is an ellipse. Let $f$ be a $\cI_\Sigma$-valued function satisfying $\alpha=df$ and $\Sigma\subset |f|^{-1}(0)$. In a
trivialization of $\cI_\Sigma$ near infinity, $f$ has the asymptotic expansion
\begin{align}\label{eq:asymptotic-expansion-model}
  f(\mathbf v)=a_0+P(\mathbf v)+O'(|\mathbf v|^{-1}),
\end{align}
where $a_0$ is a constant depends on $P$. The ellipse $\Sigma$ is also determined by the coefficients of $P$. Moreover, $\Zero(\alpha)$ is exactly $\Sigma$.
\end{theorem}

We call $(\Sigma,\cI_{\Sigma},f)$ the model associated with $P$. Fix the trivialization of $\cI_\Sigma$ near infinity.  For sufficiently large
$R$, define the flux of $f$ by
\[
  \mathrm{Flux}(f)
  :=\int_{|\mathbf v|=R}\star_{g_{\Euc}}df.
\]
This is independent of
$R$, since $d\star_{g_{\Euc}}df=0$.  The same definition applies to $\Z/2$-harmonic functions on manifolds with Euclidean end. 
The sign of the flux depends on the choice of trivialization of $\cI_\Sigma$ near infinity. 

Recall that $g_j^\pm$ are the constant values of $\bar g$ near $p_j^\pm$ and $P_j^\pm$ are quadratic polynomials given in \eqref{eq:paired-quadratics}.
Set $\hat P_j^\pm=P_j^\pm-P_j^\pm(p_j^\pm)$.  These homogeneous polynomials are
$g_j^\pm$-harmonic and satisfy $\hat P_j^-=-\hat P_j^+\circ Q$.

Let $\alpha=dF$ be the ordinary $1$-form constructed in
\cref{prop:euclidean-ended-background}.  Fix $j$ and choose disjoint
coordinate balls $B_j^\pm$ centered at $p_j^\pm$.  On $B_j^\pm$, the metric
$\bar g$ is the constant metric $g_j^\pm$ and
$\alpha=dP_j^\pm=d\hat P_j^\pm$.  Let $\gamma_j$ be the segment of the
$t$-axis joining $p_j^-$ to $p_j^+$.  By construction, its interior contains
no zero of $\alpha$.  Let $U_j\Subset\R^3$ be a smoothly bounded regular
neighborhood of $B_j^-\cup\gamma_j\cup B_j^+$ such that
$\Zero(\alpha)\cap\overline U_j=\{p_j^-,p_j^+\}$.
We call $U_j$ a smooth dumbbell domain.
The case $j=1$ is illustrated in \cref{fig:cantor-background}.

Use the coordinate identifications
$(t,x,y)\mapsto(\tau,x,y)=(t-q_j^\pm,x,y)$ to regard $B_j^\pm$ as balls about
the origins in two copies of $\R^3$ equipped with the constant metrics
$g_j^\pm$.  Choose an orientation-preserving linear isometry $L_j^+:(\R^3,g_j^+)\to (\R^3,g_{\mathrm{Euc}})$ 
and define $L_j^-:=L_j^+\circ Q:(\R^3,g_j^-)\to (\R^3,g_{\mathrm{Euc}})$.
Apply \cref{thm:compact-branch-model} to $\hat P_j^\pm$, and denote the
resulting functions and ellipses by $f_j^\pm$ and $\Sigma_j^\pm$.  Since
$\hat P_j^-=-\hat P_j^+\circ Q$ and $Q^*g_j^+=g_j^-$, we have
$f_j^-=-f_j^+\circ Q$.  

For $\eps>0$, set
$f_{j,\eps}^\pm(\mathbf v)=\eps^2f_j^\pm(\mathbf v/\eps)$ and
$\Sigma_{j,\eps}^\pm=\eps\Sigma_j^\pm$, and pull back the corresponding line
bundles under $\mathbf v\mapsto\mathbf v/\eps$. Denote these data by $(\Sigma_{j,\eps}^\pm,\cI_{j,\eps}^\pm,f_{j,\eps}^\pm)$.
For small
$\eps$, the ellipses $\Sigma_{j,\eps}^\pm$ lie in $B_j^\pm$. 
The paired replacement is illustrated in \cref{fig:cantor-paired-surgery}.

\begin{figure}[!htb]
  \centering
  \includegraphics[width=.80\textwidth]{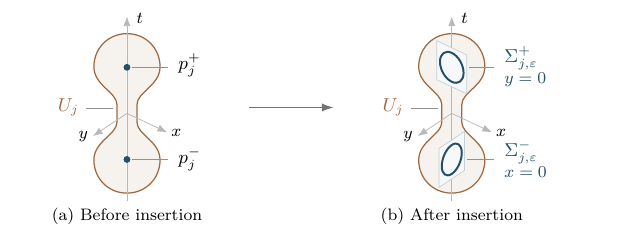}
  \caption{Paired branchsurgery in $U_j$.}
  \label{fig:cantor-paired-surgery}
\end{figure}

On every fixed annulus centered at the origin, we have
$|df_{j,\eps}^\pm-d\hat P_j^\pm|=O'(\eps^3)$. 
Moreover, $\mathrm{Flux}(f_{j,\eps}^\pm)=\eps^3\mathrm{Flux}(f_j^\pm)$.
With compatible trivializations of the flat line bundles $\cI_{j,\eps}^\pm$ near infinity,
we have $\mathrm{Flux}(f_{j,\eps}^+)=-\mathrm{Flux}(f_{j,\eps}^-)$.
We fix these trivializations away from the ellipses afterward. Through the coordinate
identifications above, $(\Sigma_{j,\eps}^\pm,\cI_{j,\eps}^\pm,f_{j,\eps}^\pm)$ can be regarded as defined on $B_j^\pm$. 
Set $U_{j,\eps}=U_j\setminus(\Sigma_{j,\eps}^+\cup\Sigma_{j,\eps}^-)$. 

Choose concentric coordinate balls $B_{j,0}^\pm\Subset B_{j,1}^\pm\Subset B_j^\pm$.
Set \mbox{$W=U_j\setminus\overline{B_{j,0}^+\cup B_{j,0}^-}$}.
Put \mbox{$A^\pm=B_{j,1}^\pm\setminus\overline{B_{j,0}^\pm}$}.
Choose smooth cutoffs $\chi^\pm$ on $A^\pm$.
Each cutoff is one near the inner boundary and zero near the outer boundary.
Choose $L_0\Subset W$ with two components containing $\supp d\chi^+$ and
$\supp d\chi^-$, respectively.
Require each component to meet $\gamma_j$.
Enlarge $L_0$ slightly and connect its components by a thin tube along the
$t$-axis to obtain a smoothly bounded connected open set $L_1\Subset W$
containing $\overline{L_0}$.
The set $L_0$ will contain the support of the gluing error.
The correction will be supported in $L_1$.
All these choices are fixed for each $j$ and do not depend on $\eps$.

\begin{lemma}\label[lemma]{lem:compact-support-right-inverse}
There is a linear operator
\[
  T:\left\{\rho\in\Omega_c^3(L_0):\int_{L_0}\rho=0\right\}
  \longrightarrow\Omega_c^2(L_1)
\]
such that $d(T\rho)=\rho$.
The forms $T\rho$ are supported in a fixed compact subset of $L_1$.
For every integer $m\geq0$, there is $C_m<\infty$ such that
\[
  \|T\rho\|_{C^m(\R^3)}\leq C_m\|\rho\|_{C^m(\R^3)}.
\]
The norms use Euclidean coordinates.
\end{lemma}

\begin{proof}
Equip $L_1$ with the Euclidean metric.
Write $\rho=f\,\vol_{g_{\Euc}}$.
Extend $f$ by zero to $L_1$.
Let $\nu$ be the outward unit normal on $\partial L_1$.
Since $L_1$ is connected and $\int_{L_1}f\,\vol_{g_{\Euc}}=0$, the Neumann problem
\[
  \Delta u=f\quad\text{in }L_1,\qquad
  \partial_\nu u=0\quad\text{on }\partial L_1
\]
has a unique solution $u\in C^\infty(\overline{L_1})$ with $\int_{L_1}u\,\vol_{g_{\Euc}}=0$.
The boundary condition is imposed on every component of $\partial L_1$.
Set $\eta=\star_{g_{\Euc}}du$.
The equation $\Delta u=f$ gives $d\eta=\rho$.
Let $i:\partial L_1\hookrightarrow\overline{L_1}$ be the inclusion.
The Neumann condition gives $i^*\eta=(\partial_\nu u)\,\vol_{\partial L_1}=0$.

Fix a collar of $\partial L_1$ disjoint from $\overline{L_0}$.
Let $\pi$ be its projection onto $\partial L_1$.
Let $\mathcal H$ be the homotopy operator obtained by integration along the collar coordinate, see \cite[Chapter~I, \S4]{BottTu1982}.
The homotopy formula is
$d\mathcal H+\mathcal Hd=\operatorname{id}-\pi^*i^*$.
On the collar, $d\eta=0$ because $\supp\rho\subset L_0$.
Together with $i^*\eta=0$, this gives $d\mathcal H\eta=\eta$.
Choose a fixed smooth cutoff $\chi_\partial$ supported in the collar and equal to one near $\partial L_1$.
Extend $\chi_\partial\mathcal H\eta$ by zero off the collar.
Define $T\rho=\eta-d(\chi_\partial\mathcal H\eta)$.
Then $d(T\rho)=\rho$ because $d\eta=\rho$ and $d^2=0$.
On the collar, the identity $d\mathcal H\eta=\eta$ gives
\[
  T\rho=(1-\chi_\partial)\eta-d\chi_\partial\wedge\mathcal H\eta.
\]
Thus $T\rho$ vanishes on a fixed neighborhood of $\partial L_1$.
It therefore extends smoothly by zero to $\R^3$.
The support $\supp(T\rho)$ lies in a fixed compact subset of $L_1$ because $\chi_\partial$ is fixed.
The operator $T$ is linear because $u$ and $d(\chi_\partial\mathcal H\eta)$ depend linearly on $\rho$.

By standard elliptic theory,
$\|\eta\|_{C^m(\overline{L_1})}\leq A_m\|\rho\|_{C^m(\R^3)}$.
The operator $\mathcal H$ is bounded in every $C^m$ norm because it integrates over a fixed interval.
The displayed formula on the collar and the equality $T\rho=\eta$ elsewhere therefore give
\[
  \|T\rho\|_{C^m(\R^3)}
  \leq B_m\|\eta\|_{C^m(\overline{L_1})}
  \leq C_m\|\rho\|_{C^m(\R^3)}.
\]
The construction of $T$ does not depend on $m$.
\end{proof}

\begin{proposition}
  \label[proposition]{prop:quantitative-paired-surgery}
For every sufficiently small $\eps>0$, there are embedded ellipses
$\Sigma_{j,\eps}^\pm\subset B_j^\pm$, a smooth metric $g_{j,\eps}$ on $U_j$, and a nondegenerate
$\Ztwo$-harmonic $1$-form
$(\Sigma_{j,\eps}^+\cup\Sigma_{j,\eps}^-,\cI_{j,\eps},\alpha_{j,\eps})$ on $(U_j,g_{j,\eps})$ with
the following properties. The line bundle $\cI_{j,\eps}$ is trivial on a collar of $\partial U_j$.
Moreover, in a trivialization, $\alpha_{j,\eps}=\alpha$ and $g_{j,\eps}=\bar g$ on a (smaller) collar.
\end{proposition}
\begin{proof}
For sufficiently small $\eps$, the scaled ellipses satisfy
$\Sigma_{j,\eps}^\pm\subset B_{j,0}^\pm$.
Use the chosen trivializations to glue $\cI_{j,\eps}^\pm$ to the trivial line bundle on $W$.
Denote the resulting line bundle on $U_{j,\eps}$ by $\cI_{j,\eps}$.
This construction fixes a trivialization of $\cI_{j,\eps}$ on $W$.

On $A^\pm$, in the chosen trivialization of $\cI_{j,\eps}$, set
$\theta_{\eps}^\pm=f_{j,\eps}^\pm-\hat P_j^\pm$.
Define $\alpha_{j,\eps}$ to be $df_{j,\eps}^\pm$ on $B_{j,0}^\pm$, to be
$d\hat P_j^\pm+d(\chi^\pm\theta_{\eps}^\pm)$ on $A^\pm$, and to be $\alpha$
outside $B_{j,1}^+\cup B_{j,1}^-$.
The choices of $\chi^\pm$ make these definitions agree near the annular boundaries.
Each local expression is closed, so $\alpha_{j,\eps}$ is closed.

By \cref{thm:compact-branch-model}, $\Zero(df_{j,\eps}^\pm)=\Sigma_{j,\eps}^\pm$ in $B_{j,0}^\pm$.
On the annulus $A^\pm$, the asymptotic expansion \eqref{eq:asymptotic-expansion-model} implies that $\alpha_{j,\eps}$
converges uniformly to $d\hat P_j^\pm$ as $\eps\to 0$. Since $d\hat{P}_j^\pm$ has no zeros on $A^\pm$, the form $\alpha_{j,\eps}$ 
is nonvanishing on $A^\pm$ for sufficiently small $\eps$. 
Outside the outer balls $B_{j,1}^\pm$, $\alpha_{j,\eps}=\alpha$, whose only zeros in $\overline U_j$ lie in the
inner balls.  Hence
$\Zero(\alpha_{j,\eps})=\Sigma_{j,\eps}^+\cup\Sigma_{j,\eps}^-$, and
$\alpha_{j,\eps}=\alpha$ near $\partial U_j$.

Set $\widehat\beta_\eps=\star_{\bar g}\alpha_{j,\eps}$.
Put $\rho_\eps=d\widehat\beta_\eps$.
Away from $\supp d\chi^\pm$ in $A^\pm$, the cutoff $\chi^\pm$ is locally constant.
On $A^\pm\setminus\supp d\chi^\pm$, we have $\alpha_{j,\eps}=(1-\chi^\pm)d\hat P_j^\pm+\chi^\pm df_{j,\eps}^\pm$.
The forms $d\hat P_j^\pm$ and $df_{j,\eps}^\pm$ are harmonic for $g_j^\pm=\bar g$ on $A^\pm$.
On $B_{j,0}^\pm\setminus\Sigma_{j,\eps}^\pm$, we have $\alpha_{j,\eps}=df_{j,\eps}^\pm$.
Outside $B_{j,1}^+\cup B_{j,1}^-$, we have $\alpha_{j,\eps}=\alpha$.
Hence $\supp\rho_\eps\subset\supp d\chi^+\cup\supp d\chi^-\subset L_0$.
The model asymptotics give $\alpha_{j,\eps}\to\alpha$ smoothly on the fixed annuli.
Since $d\star_{\bar g}\alpha=0$, it follows that $\rho_\eps\to0$ smoothly.

Using the fixed trivialization of $\cI_{j,\eps}$ on $W$, regard $\rho_\eps$ as an ordinary
compactly supported $3$-form and extend it by zero to $U_j$.  Since $\rho_\eps$ is supported in
$A^+\cup A^-$, Stokes' theorem gives
\[
  \int_{U_j}\rho_\eps
  =
    \int_{\partial B_{j,1}^+}\widehat\beta_\eps
    -
    \int_{\partial B_{j,0}^+}\widehat\beta_\eps
+
    \int_{\partial B_{j,1}^-}\widehat\beta_\eps
    -
    \int_{\partial B_{j,0}^-}\widehat\beta_\eps.
\]
The integrals over $\partial B_{j,1}^\pm$ vanish because $\hat{P}_j^\pm$ are smooth homogeneous harmonic
polynomials.  The integrals over $\partial B_{j,0}^\pm$ are
$\mathrm{Flux}(f_{j,\eps}^\pm)$ and cancel by the choice of
trivializations.  Hence $\int_{U_j}\rho_\eps=0$.
Since $\supp\rho_\eps\subset L_0\subset L_1$, we also have
$\int_{L_1}\rho_\eps=\int_{U_j}\rho_\eps=0$.

Apply \cref{lem:compact-support-right-inverse} to $\rho_\eps$ with the fixed domains $L_0,L_1$.
The resulting form $T\rho_\eps$ satisfies $d(T\rho_\eps)=\rho_\eps$.
For fixed $j$, the constants in the lemma do not depend on $\eps$.
The estimates therefore give $T\rho_\eps\to0$ in $C^\infty$.
Extend $T\rho_\eps$ by zero to $W$.
Use the fixed trivialization of $\cI_{j,\eps}$ on $W$ to regard $T\rho_\eps$ as a form with values in $\cI_{j,\eps}$.
Extend $T\rho_\eps$ by zero to $U_{j,\eps}$.
Both extensions are smooth because $T\rho_\eps$ has compact support in $L_1\Subset W$.
Define $\beta_\eps=\widehat\beta_\eps-T\rho_\eps$.
Then $d\beta_\eps=0$ because $d\widehat\beta_\eps=\rho_\eps=d(T\rho_\eps)$.

The form $\alpha$ has no zeros on $W$.
Since $\overline{L_1}\subset W$ is compact, $|\alpha|_{\bar g}$ is bounded away from zero on $\overline{L_1}$.
The smooth
convergence of $\alpha_{j,\eps}$ to $\alpha$ and of $T\rho_\eps$ to zero gives
$\alpha_{j,\eps}\wedge\beta_\eps>0$ on $L_1$ for sufficiently small $\eps$.  
Since $T\rho_\eps=0$ outside $L_1$, we have
$\beta_\eps=\star_{\bar g}\alpha_{j,\eps}$ on $U_{j,\eps}\setminus L_1$.
The form $\alpha_{j,\eps}$ has no zeros on $U_{j,\eps}$, so
$\alpha_{j,\eps}\wedge\beta_\eps=|\alpha_{j,\eps}|^2_{\bar g}\vol_{\bar g}>0$ on $U_{j,\eps}\setminus L_1$.

Set $Z=\Sigma_{j,\eps}^+\cup\Sigma_{j,\eps}^-$.
Apply \cref{thm:relative-metric-realization} to $\alpha_{j,\eps}$ and $\beta_\eps$, retaining $\bar g$ near $Z$ and near $\partial U_j$.
Use $\bar g$ as the auxiliary metric $g_{\mathrm{aux}}$ in its proof.
The resulting smooth metric $g_{j,\eps}$ satisfies $\star_{g_{j,\eps}}\alpha_{j,\eps}=\beta_\eps$.
The construction returns $g_{\mathrm{aux}}$ wherever $\beta_\eps=\star_{g_{\mathrm{aux}}}\alpha_{j,\eps}$.
Thus $g_{j,\eps}=\bar g$ outside $L_1$.
In particular, $g_{j,\eps}=g_j^\pm$ on $B_{j,0}^\pm$.
Since $\alpha_{j,\eps}$ and $\beta_\eps$ are closed, $\alpha_{j,\eps}$ is $g_{j,\eps}$-harmonic.
On $B_{j,0}^\pm$, we have $\alpha_{j,\eps}=df_{j,\eps}^\pm$.
Thus $\alpha_{j,\eps}$ is nondegenerate by \cref{thm:compact-branch-model}.
\end{proof}

On $\overline{L_1}$, the pair $(\alpha_{j,\eps},\beta_\eps)$ converges smoothly to
$(\alpha,\star_{\bar g}\alpha)$.
The metric construction in \cref{thm:relative-metric-realization} depends smoothly on a positive pair.
Since $\alpha$ has no zeros on $\overline{L_1}$, we obtain $g_{j,\eps}\to\bar g$ smoothly on $\overline{L_1}$.
On $U_j\setminus L_1$, we have $g_{j,\eps}=\bar g$.
Thus $g_{j,\eps}\to\bar g$ in every $C^m$ norm on $U_j$.
Moreover, since $\Sigma_{j,\eps}^\pm$ is obtained from the fixed ellipse
$\Sigma_j^\pm$ by dilation by $\eps$, its diameter and length converge to zero as $\eps\to 0$, and it's contained in a ball centered at $p_j^\pm$ whose radius also converges to zero.
All these quantities are measured with respect to $\bar g$.

\begin{lemma}\label[lemma]{lem:paired-surgery-estimates}
The data in \cref{prop:quantitative-paired-surgery} can be chosen so that
\[
  \bigl\||\alpha_{j,\eps}|_{g_{j,\eps}}-|\alpha|_{\bar g}\bigr\|_{\mathcal{C}^{0,1/2}(\overline U_j)}
  \longrightarrow0
  \qquad\text{as }\eps\to0.
\]
Here the H\"older norm uses Euclidean metric. Moreover,
\[
  \int_{U_{j,\eps}}|\nabla^{g_{j,\eps}}\alpha_{j,\eps}|_{g_{j,\eps}}^2
    \,\vol_{g_{j,\eps}}
  \longrightarrow
  \int_{U_j}|\nabla^{\bar g}\alpha|_{\bar g}^2\,\vol_{\bar g}\qquad\text{as }\eps\to0.
\]
\end{lemma}

\begin{proof}
By the construction in \cref{prop:quantitative-paired-surgery}, $g_{j,\eps}=\bar g$ outside $L_1$.
In particular, $g_{j,\eps}=g_j^\pm$ on $B_{j,0}^\pm$.
Set $w_\eps:=|\alpha_{j,\eps}|_{g_{j,\eps}}-|\alpha|_{\bar g}$.

Choose $R_C>0$ so that $\Sigma_j^\pm\subset B_{R_C}(0)$ and the asymptotic
expansion \eqref{eq:asymptotic-expansion-model} for $f_j^\pm$ holds for
$|\mathbf v|\geq R_C$.  The norm $|df_j^\pm|$ is $1/2$-H\"older on $B_{R_C}(0)$,
and $df_j^\pm$ has finite energy there.  Write $r_\pm$ for the Euclidean
distance from $p_j^\pm$.  By the definition
$f_{j,\eps}^\pm(\mathbf v)=\eps^2f_j^\pm(\mathbf v/\eps)$, the norm
$|df_{j,\eps}^\pm|$ is $O(\eps)$ on $r_\pm\leq R_C\eps$, and its
$1/2$-H\"older seminorm on this ball is $O(\eps^{1/2})$.  
Since $|\alpha|_{\bar g}=|d\hat P_j^\pm|_{g_j^\pm}$ is Lipschitz on
$B_{j,0}^\pm$ and vanishes at $p_j^\pm$, its supremum on the same core is
$O(\eps)$ and its $1/2$-H\"older seminorm there is $O(\eps^{1/2})$.
Consequently, on $\{r_\pm\le R_C\eps\}$, we have $\|w_\eps\|_\infty=O(\eps)$ and
 $[w_\eps]_{\mathcal{C}^{0,1/2}}=O(\eps^{1/2})$.

On the annulus $\{\mathbf{v}\in B_{j,0}^\pm: R_C\eps\leq r_\pm\}$, differentiating
\eqref{eq:asymptotic-expansion-model} and rescaling gives
\begin{equation}\label{eq:paired-model-derivative-estimates}
  \begin{aligned}
    |df_{j,\eps}^\pm-d\hat P_j^\pm|_{g_j^\pm}
      \leq C\eps^3r_\pm^{-2}\text{ and }
    |\nabla^{g_j^\pm}(df_{j,\eps}^\pm-d\hat P_j^\pm)|_{g_j^\pm}
      \leq C\eps^3r_\pm^{-3}.
  \end{aligned}
\end{equation}

Since $d\hat P_j^\pm$ is linear and $g_j^\pm$ is constant,
$\nabla^{g_j^\pm}d\hat P_j^\pm$ is constant.  The derivative estimate above
and Kato's inequality therefore give
\[
    \bigl|d\bigl(|df_{j,\eps}^\pm|_{g_j^\pm}\bigr)\bigr|_{g_j^\pm}
    \leq |\nabla^{g_j^\pm}df_{j,\eps}^\pm|_{g_j^\pm}
    \leq C(1+\eps^3r_\pm^{-3})\leq C(1+R_C^{-3}),
\]
where $C$ is independent of $\eps$ and the last inequality uses
$r_\pm\geq R_C\eps$.  Any two points in this annulus can be joined within
it by a path whose Euclidean length is at most a fixed multiple of their
Euclidean distance.  Integrating the preceding bound along such paths,
and using the equivalence of $g_j^\pm$ and $g_{\Euc}$, yields $\bigl||df_{j,\eps}^\pm|_{g_j^\pm}(p)
       -|df_{j,\eps}^\pm|_{g_j^\pm}(q)\bigr|
  \leq L|p-q|_{g_{\mathrm{Euc}}}$
for all $p,q$ in this annulus, with $L$ independent of $\eps$.
The same is true of $|\alpha|_{\bar{g}}$, so
$[w_\eps]_{\mathcal{C}^{0,1}}$ is uniformly bounded on this annulus as $\eps\to 0$.
Moreover, the estimate for $df_{j,\eps}^\pm-d\hat P_j^\pm$ gives
$|w_\eps|
    =\bigl||df_{j,\eps}^\pm|_{g_j^\pm}
           -|d\hat P_j^\pm|_{g_j^\pm}\bigr|
    \leq |df_{j,\eps}^\pm-d\hat P_j^\pm|_{g_j^\pm}
    \leq C\eps^3r_\pm^{-2}.$

Since $r_\pm\geq R_C\eps$ on this annulus, it follows that
$\|w_\eps\|_\infty=O(\eps^3(R_C\eps)^{-2})=O(\eps)$.  By the interpolation inequality for H\"older norms, we have
\[
  [w_\eps]_{\mathcal{C}^{0,1/2}}
  \leq\bigl(2\|w_\eps\|_\infty[w_\eps]_{\mathcal{C}^{0,1}}\bigr)^{1/2}
  =O(\eps^{1/2})
\]
on this annulus.

We now combine the estimates on $\{r_\pm\le R_C\eps\}$ and the annulus $\{\mathbf{v}\in B_{j,0}^\pm:r_\pm\ge R_C\eps\}$
 into an estimate on $B_{j,0}^\pm$.
For two points in the ball $B_{j,0}^\pm$, split the line
segment joining them at its intersections with $\{r_\pm=R_C\eps\}$.
There are at most three pieces, and applying the core and annular
estimates to their endpoints gives $\|w_\eps\|_{\mathcal{C}^{0,1/2}(B_{j,0}^\pm)}=O(\eps^{1/2})$.

On the complement of slightly smaller balls contained in $B_{j,0}^\pm$,
$\alpha_{j,\eps}\to\alpha$ and $g_{j,\eps}\to\bar g$ smoothly, and $\alpha$ is nonvanishing.  Thus $w_\eps\to0$ smoothly on this complement.
Combining this with the preceding estimate on $B_{j,0}^\pm$ gives
$\|w_\eps\|_{\mathcal{C}^{0,1/2}(\overline U_j)}\to0$ as $\eps\to0$.

It remains to compare the energies.
Differentiating
$f_{j,\eps}^\pm(\mathbf v)=\eps^2f_j^\pm(\mathbf v/\eps)$
twice and rescaling the integral gives
\[
  \int_{\{r_\pm<R_C\eps\}\setminus\Sigma_{j,\eps}^\pm}
    |\nabla^{g_j^\pm}\alpha_{j,\eps}|_{g_j^\pm}^2
    \,\vol_{g_j^\pm}
  =
  \eps^3
  \int_{B_{R_C}(0)\setminus\Sigma_j^\pm}
    |\nabla^{g_j^\pm}df_j^\pm|_{g_j^\pm}^2
    \,\vol_{g_j^\pm}
  =O(\eps^3),
\]
where the last estimate uses the finite energy of the model.
The energy of $\alpha$ on the same ball is $O(\eps^3)$,
since its covariant derivative is constant.

On the annulus $\{\mathbf v\in B_{j,0}^\pm:r_\pm\ge R_C\eps\}$,
use the fixed trivialization to compare the forms.
Since $\alpha_{j,\eps}=df_{j,\eps}^\pm$ and
$\alpha=d\hat P_j^\pm$ there,
\eqref{eq:paired-model-derivative-estimates} gives
$|\nabla^{g_j^\pm}(\alpha_{j,\eps}-\alpha)|_{g_j^\pm}
\leq C\eps^3r_\pm^{-3}$.
Squaring this bound and integrating in polar coordinates gives
\[
  \|\nabla^{g_j^\pm}(\alpha_{j,\eps}-\alpha)\|_{L^2(\{\mathbf v\in B_{j,0}^\pm:r_\pm\ge R_C\eps\})}^2
  \leq C\eps^6\int_{R_C\eps}^{\infty}r^{-4}\,dr
  =O(\eps^3).
\]
On this annulus, expanding the squared norm and applying
the Cauchy--Schwarz inequality gives
\[
  \begin{aligned}
  &\left|
    \|\nabla^{g_j^\pm}\alpha_{j,\eps}\|_{L^2}^2
    -\|\nabla^{g_j^\pm}\alpha\|_{L^2}^2
  \right|\\
  \leq&
    2\|\nabla^{g_j^\pm}\alpha\|_{L^2}
      \|\nabla^{g_j^\pm}(\alpha_{j,\eps}-\alpha)\|_{L^2}
    +\|\nabla^{g_j^\pm}(\alpha_{j,\eps}-\alpha)\|_{L^2}^2
    =O(\eps^{3/2}).
  \end{aligned}
\]
The form $\alpha$, the metrics $g_j^\pm$, and the balls $B_{j,0}^\pm$ are fixed as $\eps\to0$.
Thus
$\|\nabla^{g_j^\pm}\alpha\|_{L^2}
\leq\|\nabla^{g_j^\pm}\alpha\|_{L^2(B_{j,0}^\pm)}=O(1)$.
Together with $\|\nabla^{g_j^\pm}(\alpha_{j,\eps}-\alpha)\|_{L^2}=O(\eps^{3/2})$, this proves the last estimate.
Since $g_{j,\eps}=\bar g=g_j^\pm$ on $B_{j,0}^\pm$,
adding the estimates on $\{r_\pm<R_C\eps\}$
proves energy convergence on $B_{j,0}^\pm$.

On $\overline W$, the forms $\alpha_{j,\eps}$ and metrics $g_{j,\eps}$
converge smoothly to $\alpha$ and $\bar g$, respectively.
Hence the covariant derivatives, tensor norms, and volume forms
converge uniformly there.
Integration over the region $W$ gives
energy convergence on $W$.
Adding the comparisons on $B_{j,0}^+$, $B_{j,0}^-$, and $W$
proves the assertion.
\end{proof}

\subsection{The diagonal countable insertion}

We apply \cref{prop:quantitative-paired-surgery} in every interval $I_j$,
choosing the parameters so that the resulting metric extends smoothly across $K_0$.
Recall that $\gamma_j$ is the segment joining $p_j^-$ to $p_j^+$.
Set $e_j=\dist_{g_{\Euc}}(\gamma_j,K_0)>0$.
The points $p_j^\pm$ do not depend on the amplitudes $A_j$ in \cref{lem:straight-cantor-profile}.
We may therefore fix numbers $0<\sigma_j\leq 2^{-j}(e_j/2)^{2j}$ before decreasing the amplitudes.

Decrease each $A_j$ so that $A_j\|\varphi\|_{\mathcal{C}^0}<\sigma_j/4$.
These decreases preserve the decay required in
\cref{lem:straight-cantor-profile}.
Construct $F$ and $\bar g$ with these amplitudes as in
\cref{prop:euclidean-ended-background}, retaining $g_0$ on a fixed neighborhood of
$[-1,1]\times\{(0,0)\}$.
Keep $F$ and $\bar g$ fixed from now on, and write $\alpha=dF$.

On $\gamma_j$, $\alpha=h(t)\,dt$ and $b,c\geq1$, so
$\sup_{\gamma_j}|\alpha|_{\bar g}
\leq A_j\|\varphi\|_{\mathcal{C}^0}<\sigma_j/4$.
Choose a smooth dumbbell domain
$U_j\Subset I_j\times\R^2$ around $\gamma_j$, contained in
its $r_j$-neighborhood, with
$0<r_j<\min\{\sigma_j,e_j/2\}$.
Taking $U_j$ sufficiently thin gives
$\sup_{U_j}|\alpha|_{\bar g}<\sigma_j/2$.
Since $|\nabla^{\bar g}\alpha|_{\bar g}^2$ is bounded near $\gamma_j$,
shrinking $U_j$ further gives
\begin{equation}\label{eq:diagonal-background-energy}
  \vol_{\bar g}(U_j)
  +\int_{U_j}|\nabla^{\bar g}\alpha|_{\bar g}^2\,\vol_{\bar g}
  <\frac{\sigma_j}{2}.
\end{equation}

The domains $U_j$ are pairwise disjoint because they are contained in
the pairwise disjoint regions $I_j\times\R^2$.
Set $d_j=\dist_{g_{\Euc}}(U_j,K_0)$.
The inequality $d_j\geq e_j-r_j>e_j/2$ implies
\begin{equation}\label{eq:sigma-flatness}
  d_j\geq\frac{e_j}{2},
  \qquad
  \sigma_j\leq2^{-j}d_j^{2j}.
\end{equation}

Choose the endpoint balls $B_j^\pm$ inside $U_j$ as in
\cref{prop:quantitative-paired-surgery}.
For each $j$, fix $B_{j,0}^\pm,B_{j,1}^\pm,\chi^\pm,L_0,L_1$ and the operator $T$ in
\cref{lem:compact-support-right-inverse} before choosing $\eps_j$.
For $\eps_j>0$ sufficiently small, denote the resulting metric, line bundle,
and form by $g_j$, $\cI_j$, and $\alpha_j$, respectively.
In the prescribed trivialization near $\partial U_j$,
$g_j=\bar g$ and $\alpha_j=\alpha$.
In the following, we use Euclidean coordinates for tensor $\mathcal{C}^m$ norms and Euclidean distance for H\"older norms.

\begin{lemma}\label[lemma]{lem:strong-diagonal-choice}
The parameters $\eps_j$ can be chosen so that, for every $j\geq1$,
\begin{align}
  \|g_j-\bar g\|_{\mathcal{C}^j(U_j)}
    &\leq\sigma_j,
    \label{eq:diagonal-metric}\\
  \bigl\||\alpha_j|_{g_j}-|\alpha|_{\bar g}\bigr\|_{\mathcal{C}^{0,1/2}(\overline U_j)}
    &\leq\frac{\sigma_j}{2},
    \label{eq:diagonal-norm-difference}\\
  \vol_{g_j}(U_j)+
  \int_{U_{j}}
    |\nabla^{g_j}\alpha_j|_{g_j}^2\,\vol_{g_j}
    &\leq\sigma_j.
    \label{eq:diagonal-energy}
\end{align}
\end{lemma}

\begin{proof}
For fixed $j$, the metric convergence following
\cref{prop:quantitative-paired-surgery} and the estimates in
\cref{lem:paired-surgery-estimates} show that the first two left-hand sides
tend to zero, while the last tends to a value less than $\sigma_j/2$
by \eqref{eq:diagonal-background-energy}.
Thus all three bounds hold for sufficiently small $\eps_j$.
\end{proof}

Fix $\vec{\eps}=(\eps_j)_{j\geq1}$ given by \cref{lem:strong-diagonal-choice}.
Since $|\alpha|_{\bar g}$ is Lipschitz on a fixed compact neighborhood
containing all $U_j$ and $\sup_{U_j}|\alpha|_{\bar g}<\sigma_j/2$,
\eqref{eq:diagonal-norm-difference} gives
\begin{equation}\label{eq:diagonal-norm-bounds}
  \sup_{U_j}|\alpha_j|_{g_j}\leq\sigma_j,
  \qquad
  [|\alpha_j|_{g_j}]_{\mathcal{C}^{0,1/2}(\overline U_j)}\leq C,
\end{equation}
where $C$ is independent of $j$, using $\sigma_j\leq1$.

All geometric quantities below use the Euclidean metric.
The choice of $U_j$ gives $\diam U_j\leq\ell_j+2\sigma_j$.
Since $\Sigma_{j,\eps_j}^\pm\subset U_j$ and an ellipse has length at most
$\pi$ times its diameter, we have
\begin{equation}\label{eq:diagonal-geometry}
  \begin{aligned}
    \diam\Sigma_{j,\eps_j}^\pm
      \leq\ell_j+2\sigma_j\text{ and }
    \operatorname{length}(\Sigma_{j,\eps_j}^\pm)
      \leq\pi(\ell_j+2\sigma_j).
  \end{aligned}
\end{equation}
Since $\sum_j\ell_j\leq2$ and $\sum_j\sigma_j<\infty$,
the lengths of ellipses are summable.

Define $g=g_j$ on $U_j$ and $g=\bar g$ outside $\bigcup_jU_j$.
$g_j$ agrees with $\bar{g}$ on a collar of $\partial U_j$, so $g$ is smooth on $\R^3\setminus K_0$.
To prove smoothness at $K_0$, fix a multiindex $\mu$ and an integer $N\geq0$.
Since $\gamma_j\subset U_j$, we have $0<d_j\leq e_j\leq\ell_j\to0$.
For $|\mu|\le j$, \eqref{eq:sigma-flatness} and
\eqref{eq:diagonal-metric} imply
\[
  \sup_{x\in U_j}
  \frac{|\partial^\mu (g-\bar{g})(x)|}{\dist_{g_{\Euc}}(x,K_0)^N}
  \leq 2^{-j}d_j^{2j-N}\longrightarrow0.
\]
On $\R^3\setminus K_0$, these derivatives vanish outside the domains.
Each of the finitely many omitted domains has positive distance from
$K_0$.
Thus $\partial^\mu (g-\bar{g})$ extends continuously by zero to $K_0$.
At each $p\in K_0$, the estimate with $N=1$ gives
$\partial^\mu (g-\bar{g})(p+v)=o(|v|)$ as $v\to0$.
The extended derivative is therefore differentiable at $p$, with
derivative zero.
Thus $g$ is smooth and every derivative of $g-\bar g$ vanishes on $K_0$.

Set $Z=K_0\cup\bigcup_{j\geq1}
  (\Sigma_{j,\eps_j}^+\cup\Sigma_{j,\eps_j}^-)$.
Since $\cI_j$ has a trivialization on a collar of $\partial U_j$, 
we glue these bundles to the trivial real line bundle outside the domains $U_j$, obtaining
a smooth flat real line bundle $\cI\to\R^3\setminus Z$.
The local forms $\alpha_j$ agree with $\alpha$ on the collars and hence
define a smooth $\cI$-valued $g$-harmonic $1$-form $\alpha_{\vec{\eps}}$ on
$\R^3\setminus Z$.

The norms $|\alpha_{j}|_{g_j}$ extend continuously by zero across
$\Sigma_{j,\eps_j}^+\cup\Sigma_{j,\eps_j}^-$.
Since $\sup_{U_j}|\alpha_j|_{g_j}\leq\sigma_j\to0$ by
\eqref{eq:diagonal-norm-bounds} and $|\alpha|_{\bar g}$ vanishes on $K_0$,
$|\alpha_{\vec{\eps}}|_g$ also extends continuously by zero across $K_0$.
The proof of \cref{prop:quantitative-paired-surgery} gives
$\Zero(\alpha_j)=\Sigma_{j,\eps_j}^+\cup\Sigma_{j,\eps_j}^-$.
Since every $p_j^\pm$ lies in $U_j$, the only zeros of $\alpha$
outside the domains lie on $K_0$.
Hence $\Zero(\alpha_{\vec{\eps}})=Z$.

\subsection{Verification of the Cantor model}

We verify the analytic conditions (iv) and (v) in \cref{def:z2-harmonic} and compute
the frequency at every point of $Z$.
Throughout this subsection, $B_r(p)$ denotes the geodesic ball for $g$.
H\"older seminorms continue to use Euclidean distance.

We first record a consequence of \eqref{eq:sigma-flatness}.
For $p\in K_0$, set $J_p(r)=\{j:U_j\cap B_r(p)\ne\varnothing\}$.
The metrics $g$ and $g_{\Euc}$ are uniformly equivalent near $K_0$.
Thus $d_j\leq C_1r$ for $j\in J_p(r)$, with $C_1$ independent of $p$.

Fix an integer $N\geq1$.
Each of the finitely many domains with $2j<N$ has positive distance from $K_0$.
For sufficiently small $r$, every $j\in J_p(r)$ therefore satisfies $2j\geq N$.
Taking $C_1r<1$, we obtain
\begin{equation}\label{eq:cantor-tail-sum}
  \sum_{j\in J_p(r)}\sigma_j
  \leq\sum_{j\in J_p(r)}2^{-j}(C_1r)^{2j}
  \leq(C_1r)^N.
\end{equation}
For each fixed $N$, there is $r_N>0$ such that this estimate holds
for every $p\in K_0$ and $0<r<r_N$.

\begin{lemma}\label[lemma]{lem:global-holder-patching}
The norm $|\alpha_{\vec{\eps}}|_g$ is globally H\"older continuous with exponent $1/2$.
The covariant derivative satisfies $\nabla^g \alpha_{\vec{\eps}}\in L^2(\R^3\setminus Z,g)$.
There are constants $C,r_0>0$ such that for  $p\in Z$ and $0<r<r_0$
\begin{align}\int_{B_r(p)}|\alpha_{\vec{\eps}}|_{g}^2\,\vol_g\leq Cr^4.\label{eq:C4}\end{align}
For $p\in K_0$, there is a stronger bound
\begin{align}
  \int_{B_r(p)}|\alpha_{\vec{\eps}}|_{g}^2\,\vol_g\leq Cr^5.\label{eq:C5}
\end{align}
\end{lemma}

\begin{proof}
First we prove the H\"older continuity. For each $j$, extend the function $w_{\eps_j}$ from the proof of
\cref{lem:paired-surgery-estimates} by zero outside $U_j$.
Thus \eqref{eq:diagonal-norm-difference} gives
\[
  \sum_j\|w_{\eps_j}\|_{\mathcal{C}^{0,1/2}(\R^3)}
  \leq\frac12\sum_j\sigma_j<\infty.
\]
Consequently $|\alpha_{\vec{\eps}}|_g-|\alpha|_{\bar g}=\sum_jw_{\eps_j}$
is globally H\"older continuous with exponent $1/2$.
Since $|\alpha|_{\bar g}$ is bounded and Lipschitz, $|\alpha_{\vec{\eps}}|_g=|\alpha|_{\bar g}+\sum_jw_{\eps_j}$ is also globally H\"older continuous with exponent $1/2$.

Next we prove the $L^2$-integrability of the covariant derivative. \eqref{eq:diagonal-energy} gives
\[
  \int_{\R^3\setminus Z}|\nabla^g\alpha_{\vec{\eps}}|_g^2\,\vol_g
  \leq
  \int_{\R^3}|\nabla^{\bar g}\alpha|_{\bar g}^2\,\vol_{\bar g}
  +\sum_j\sigma_j<\infty.
\]

Finally, we prove the growth estimates.
For $z\in Z$, the H\"older bound and vanishing of $|\alpha_{\vec{\eps}}|_g$ on $Z$ give
$\sup_{B_r(z)}|\alpha_{\vec{\eps}}|_g\leq Cr^{1/2}$.
Together with $\vol_g(B_r(z))\leq Cr^3$, uniform near the compact set $Z$
for small $r$, this proves \eqref{eq:C4}.

For $p\in K_0$, the Lipschitz bound for $|\alpha|_{\bar g}$ gives
$\sup_{B_r(p)}|\alpha|_{\bar g}\leq Cr$.
Since $|\alpha_{\vec{\eps}}|_g-|\alpha|_{\bar g}$ vanishes outside the domains,
\eqref{eq:diagonal-norm-difference} and \eqref{eq:cantor-tail-sum} imply
\begin{equation}\label{eq:cantor-norm-comparison}
  \bigl\||\alpha_{\vec{\eps}}|_g-|\alpha|_{\bar g}\bigr\|_{L^\infty(B_r(p))}
  \leq\frac12\sum_{j\in J_p(r)}\sigma_j
  \leq \frac{1}{2}C_1^Nr^N
\end{equation}
for every integer $N\geq1$ and sufficiently small $r$.
Since $(C_1r)^N\leq C_1r$ for $N\geq1$ and $C_1r<1$, these estimates give
$\sup_{B_r(p)}|\alpha_{\vec{\eps}}|_g\leq Cr$ and hence proves \eqref{eq:C5}.
\end{proof}

We next identify the tangent form at a point in $K_0$.
For $p=(q,0,0)\in K_0$, use the affine coordinates
$(X,Y,T)=(\sqrt{B_*}x,\sqrt{B_*}y,B_*(t-q))$, in which $g(p)$ is Euclidean,
and define $\eta_p=\frac{2}{B_*}(X\,dX-Y\,dY)$.
Let $\delta_{p,r}(X,Y,T)=(rX,rY,rT)$ and set
\begin{equation}\label{eq:cantor-rescaling}
  g_{p,r}=r^{-2}\delta_{p,r}^*g,
  \qquad
  \alpha_{r,p}=r^{-2}\delta_{p,r}^*\alpha_{\vec{\eps}}.
\end{equation}
The line bundle is also pulled back under $\delta_{p,r}$.

\begin{proposition}
  \label[proposition]{prop:cantor-frequency-uniqueness}
At every $p\in K_0$, the frequency of $\alpha_{\vec{\eps}}$ is $\Freq_p(0)=1$. In the trivialization in which
$\alpha_{\vec{\eps}}=dF$ outside $U_j$'s, $\alpha_{r,p}$ 
converge to $\eta_p$ as $r\to0$.
On the other hand, at every point of $\Sigma_{j,\eps_j}^\pm$, the frequency of
$\alpha_{\vec{\eps}}$ is $1/2$.
\end{proposition}
\begin{proof}
Fix $p=(q,0,0)\in K_0$. Define $H_p(r)$ and $D_p(r)$ by
\eqref{eq:frequency-integrals} for $\alpha_{\vec{\eps}}$, and set
\begin{align*}
    \bar H_p(r)=\int_{\partial B_r(p)}|\alpha|_{\bar g}^2\,\mathrm dA_g,\qquad
    \bar D_p(r)=\int_{B_r(p)}
      |\nabla^{\bar g}\alpha|_{\bar g}^2\,\vol_{\bar g}.
\end{align*}
By \cref{lem:global-holder-patching}, both norms on $B_r(p)$ are $O(r)$.
Since $\partial B_r(p)$ has area $O(r^2)$,
\eqref{eq:cantor-norm-comparison} with $N=2$ gives
\begin{equation}\label{eq:height_comp}
  |H_p(r)-\bar H_p(r)|=O(r^5).
\end{equation}
Outside the domains $U_j$, the forms and metrics agree.
By \eqref{eq:diagonal-background-energy}, \eqref{eq:diagonal-energy}, and
\eqref{eq:cantor-tail-sum} with $N=4$,
\begin{equation}\label{eq:energy_comp}
  |D_p(r)-\bar D_p(r)|
  \leq\frac32\sum_{j\in J_p(r)}\sigma_j
  =O(r^4).
\end{equation}

Near $p$, we have $\alpha=h(t)dt+\eta_p$.
Integration by
parts on the Euclidean unit ball $B_1$ gives
\[
  c_*:=\int_{\partial B_1}|\eta_p|_{g_{\Euc}}^2\,\mathrm dA_{\Euc}
  =\int_{B_1}|\nabla^{g_{\Euc}}\eta_p|_{g_{\Euc}}^2
    \,\vol_{g_{\Euc}}>0.
\]

Since $\bar g(p)=g_{\Euc}$ and $\alpha(p)=0$, we obtain
\[
  \begin{aligned}
    |\alpha(r\mathbf w)|_{\bar g(r\mathbf w)}^2
      &=r^2|\eta_p(\mathbf w)|_{g_{\Euc}}^2+O(r^3),\\
    |\nabla^{\bar g}\alpha(r\mathbf w)|_{\bar g(r\mathbf w)}^2
      &=|\nabla^{g_{\Euc}}\eta_p|_{g_{\Euc}}^2+O(r),
  \end{aligned}
\]
uniformly for bounded $\mathbf w=(X,Y,T)$. The metrics $g_{p,r}$ converge smoothly to $g_{\Euc}$ on bounded sets, and
$\delta_{p,r}^{-1}(B_r(p))=B_1^{g_{p,r}}(0)$ converges smoothly to $B_1$. Consequently,
\[
  \begin{aligned}
    \bar H_p(r)
      &=r^4\int_{\partial B_1(0)}
        \bigl(|\eta_p|_{g_{\Euc}}^2+O(r)\bigr)
        =c_*r^4+o(r^4),\\
    \bar D_p(r)
      &=r^3\int_{B_1(0)}
        \bigl(|\nabla^{g_{\Euc}}\eta_p|_{g_{\Euc}}^2+O(r)\bigr)
        =c_*r^3+o(r^3).
  \end{aligned}
\]
Then \eqref{eq:height_comp} and \eqref{eq:energy_comp} give the same leading terms for $H_p$ and $D_p$.
By definition \eqref{eq:frequency-function} we obtain
\[
  \Freq_p(r)=\frac{rD_p(r)}{H_p(r)}
  \longrightarrow1, \text{ as } r\to 0.
\]

We next prove convergence to $\eta_p$.
In the affine coordinates $(X,Y,T)$, every $U_j$ lies within distance
$\sqrt{B_*}\sigma_j$ of the $T$-axis. For each fixed $R>0$, the rescaled domains
$\delta_{p,r}^{-1}(U_j)$ meeting $B_R(0)$ therefore lie within distance
$\frac{\sqrt{B_*}}{r}\sum_{j\in J_p(2Rr)}\sigma_j=o(r)$
of the $T$-axis, by \eqref{eq:cantor-tail-sum}.
Since any compact set away from the $T$-axis has positive distance
from it, such a set avoids all $\delta_{p,r}^{-1}(U_j)$ for sufficiently
small $r$. On this set, in the prescribed trivialization of $\cI$, we have
$\alpha_{r,p}=r^{-2}\delta_{p,r}^*\alpha$.
The Taylor expansion of the background form $\alpha$ then gives
$\alpha_{r,p}\to\eta_p$ smoothly there.

Since $|\alpha_{r,p}|_{g_{p,r}}=r^{-1}|\alpha_{\vec{\eps}}|_g\circ\delta_{p,r}$,
\eqref{eq:cantor-norm-comparison} with $N\ge 2$ and the Taylor
expansion of $\alpha$ give uniform convergence 
$|\alpha_{r,p}|_{g_{p,r}}\to|\eta_p|_{g_{\Euc}}$ on bounded sets.
Fix $R>0$.
The energy of $\alpha_{r,p}$ on the rescaled dumbell
domains inside $B_R(0)$ satisfies
\[
\begin{aligned}
&\int_{B_R(0)\cap\bigcup_j\delta_{p,r}^{-1}(U_j)}
  |\nabla^{g_{p,r}}\alpha_{r,p}|_{g_{p,r}}^2
  \,\vol_{g_{p,r}}\\
  =&\,
  r^{-3}\int_{\delta_{p,r}(B_R(0))\cap\bigcup_jU_j}
  |\nabla^g\alpha_{\vec{\eps}}|_g^2\,\vol_g
\leq 
  r^{-3}\sum_{j\in J_p(2Rr)}\sigma_j
  \longrightarrow0.
\end{aligned}
\]
Thus no energy concentrates in the rescaled dumbell domains.
In the prescribed trivialization, every subsequential limit of
$\alpha_{r,p}$ agrees with $\eta_p$ away from the $T$-axis. Since $\eta_p$ extends smoothly
across the axis, the tangent is unique and is represented by the
ordinary harmonic $1$-form $\eta_p$ on $\R^3$.

Finally, fix $p\in \Sigma_{j,\eps_j}^\pm$.
Near $p$, the construction agrees with the nondegenerate model in
\cref{thm:compact-branch-model}.
In a complex normal coordinate $\zeta$ centered at the point, the form has
leading term $d\operatorname{Re}(c\zeta^{3/2})$ with $c\ne0$.
This term has homogeneous order $1/2$.
Therefore $N_p(0)=1/2$.
\end{proof}

It remains to identify the monodromy locus.
Since $p_j^\pm\in U_j$ accumulate precisely on $K_0$ and
$\diam U_j\leq\ell_j+2\sigma_j\to0$, the ellipses
$\Sigma_{j,\eps_j}^\pm\subset U_j$ also accumulate precisely on $K_0$.
In particular, every neighborhood of a point of $K_0$ contains an entire
ellipse. Thus in any neighborhood of a point of $K_0$, the line bundle $\cI$ is nontrivial.
By \cref{def:monodromy-locus}, $K_0\subset\Mono(\cI)$.  We conclude that $\Mono(\cI)=\Zero(\alpha_{\vec{\eps}})=Z$.

Each ellipse has a neighborhood in $Z$ on which the frequency equals $1/2$.
The frequency at every Cantor point equals $1$.
By \eqref{eq:CT-definition},
we obtain $C_{\mathrm T}=\cup_{j\geq1}(\Sigma_{j,\eps_j}^+\cup\Sigma_{j,\eps_j}^-)$.

The total length of ellipses under the Euclidean metric is finite by \eqref{eq:diagonal-geometry}.
Uniform equivalence of $g$ and $g_{\Euc}$ on the compact region containing
the ellipses then gives
$\sum_j\bigl(
  \operatorname{length}_g(\Sigma_{j,\eps_j}^+)
  +\operatorname{length}_g(\Sigma_{j,\eps_j}^-)
  \bigr)<\infty$.

All dumbell domains $U_j$ lie in a fixed compact set.
Hence $g=g_{\Euc}$ and $\alpha_{\vec{\eps}}=dt$ outside a compact set.
In particular, the smooth metric $g$ is complete.
The preceding construction verifies the bundle and harmonicity conditions
in \cref{def:z2-harmonic}.
The norm, energy, and growth conditions follow from
\cref{lem:global-holder-patching}.
Together with \cref{prop:cantor-frequency-uniqueness}, these conclusions
prove \cref{thm:intro-cantor}.

\section{Positive-even graphs as monodromy loci}
\label{sec:even-graphs}

This section proves \cref{thm:intro-even-graphs}.
For every finite graph $G$ with positive even valence at each vertex,
we construct a smooth metric on $B^3$ and a $\Ztwo$-harmonic 1-form whose
monodromy locus is the image of a tame embedding
$G\hookrightarrow\interior B^3$.
We first construct compatible harmonic models near the vertices and edges,
then extend them over $B^3$ using the relative positive-pair $h$-principle.

\subsection{Local modifications of homogeneous models}

Near each vertex of valence $2n$, we use a homogeneous model
on $\R^3$ with $2n$ branch rays.
These rays correspond to the edges meeting at the vertex.
For $n>1$, the models are obtained from the critical eigensections
in Theorem~\ref{thm:critical-configurations}.
We will modify the functions and metrics away from the vertices
so that the models agree where they are glued along the edges.

\subsubsection{Homogeneous vertex models}

For a vertex of valence $2n\geq4$, choose a configuration
$P=\{p_1,\ldots,p_{2n}\}\subset S^2$ and a critical
$\cI_P$-valued eigensection $f$ such that
\[
  -\Delta_{S^2}f=\mu(\mu+1)f.
\]
Let $\cI$ be the radial pullback of $\cI_P$ to
$\R^3\setminus\bigcup_{p\in P}\R_{\geq0}p$.
In polar coordinates $(R,\omega)$, define
\[
  F(R,\omega):=R^\mu f(\omega),
\]which is a homogeneous harmonic function.

At a two-valent vertex, use
$F(t,z)=\operatorname{Re}(z^{3/2})$ on $\R_t\times\C_z$.
For a four-valent vertex, one may take the Taubes--Wu tetrahedral
configuration, whose critical eigensections vanish to order $3/2$
at every puncture \cite{TaubesWuModels,ChenHeCritical}.
For general $2n$, the vanishing order of $F$ along the branch rays
may be greater than $3/2$.

For the homogeneous model associated with $f$, the ordinary zero rays
are determined by the critical zeros of $f$.
Define
\begin{equation}\label{eq:Qf-definition}
  Q_f:=\{q\in S^2\setminus P:f(q)=0,\ df(q)=0\}.
\end{equation}
Since $dF=\mu R^{\mu-1}f\,dR+R^\mu df,$ the ordinary zeros of $dF$ are precisely
\[
  \R_{>0}Q_f:=\{Rq:R>0,\ q\in Q_f\}.
\]
The points of $Q_f$ are isolated by
\cite[Theorem~2.5]{ChengNodalSets}, and the local expansion of $F$ \eqref{eq:critical-puncture-factorization} along edge excludes accumulation at $P$.
Hence $Q_f$ is finite.

\subsubsection{Normal forms and metric matching}

We will put the homogeneous function $F$ into the form
$\operatorname{Re}(z^\gamma)$ near a compact segment of a zero ray.
Keeping $F$ fixed, we will construct a metric that agrees with the
original metric near one end and with $dt^2+|dz|^2$ near the other.

In the following local constructions, write a compact interval as
$J=[1,3]_t$ and use normal coordinates $z=x+iy$ in
$D_\eps:=\{z\in\C:|z|<\eps\}$.
For $\gamma\in\tfrac12\mathbb Z$ with $\gamma>1$, define
\[
  u_\gamma(z):=\operatorname{Re}(z^\gamma).
\]
On $J\times(D_\eps\setminus\{0\})$, let $\cI$ be the flat real line
bundle with holonomy $(-1)^{2\gamma}$ around each normal circle.
We regard $u_\gamma$ as an $\cI$-valued function.
We use the inverse metric density $A$ from
\eqref{eq:metric-from-density}, so the harmonic equation is
$\diver(A\,du)=0$.
In the chosen neighborhood, $A_{\Euc}$ denotes the density of the original
Euclidean metric(which might not be $\mathrm{I}_3$), and $\mathrm{I}_3$ be the $3$-by-$3$ identty matrix which is the density of $dt^2+|dz|^2$.

We consider coefficients that are smooth in $t$ and admit power series
in $(x,y)$.
After decreasing $\eps$, these series converge absolutely and uniformly
on $J\times D_\eps$.
We require the same convergence after any number of derivatives in $t$,
using the same $\eps$.
We call a section \emph{real analytic in the normal variables}
if its coefficients satisfy these conditions.

\begin{lemma}
  \label[lemma]{lem:fixed-order-collar}
Suppose $A_-$ and $A_+$ are positive inverse metric densities on
$J\times D_\eps$, smooth in $t$ and real analytic in the normal variables,
satisfying
\[
  \diver(A_-\,du_\gamma)=\diver(A_+\,du_\gamma)=0.
\]
After decreasing $\eps$, there is a positive inverse metric density $A$
in the same class satisfying
\[
  \diver(A\,du_\gamma)=0,
  \qquad
  A=A_-\text{ near }t=1,
  \qquad
  A=A_+\text{ near }t=3.
\]
\end{lemma}

\begin{proof}
Choose a cut-off $\chi\in\mathcal{C}^\infty(J;[0,1])$ with $\chi=0$ near $t=1$
and $\chi=1$ near $t=3$.
Define $\widetilde A:=(1-\chi)A_-+\chi A_+$.
Since $u_\gamma$ is independent of $t$, the error in the harmonic
equation is
\[
  \diver(\widetilde A\,du_\gamma)
  =2\gamma\operatorname{Re}\bigl(z^{\gamma-2}G\bigr),
  \qquad
  G:=\frac{\chi'}2z
  \bigl(A_+^{tx}-A_-^{tx}+i(A_+^{ty}-A_-^{ty})\bigr).
\]
We correct this error by adding a symmetric matrix with only normal
components.
For a complex-valued function $C$, define
\[
  H_C:=
  \begin{pmatrix}
    0&0&0\\
    0&\operatorname{Re}C&\operatorname{Im}C\\
    0&\operatorname{Im}C&-\operatorname{Re}C
  \end{pmatrix}.
\]
With $\partial_z=\tfrac12(\partial_x-i\partial_y)$, direct
differentiation gives
\[
  \diver(H_C\,du_\gamma)
  =2\gamma\operatorname{Re}\left[
    z^{\gamma-2}\bigl(z\partial_zC+(\gamma-1)C\bigr)
  \right].
\]
Writing $G=\sum_{p,q\geq0}G_{pq}(t)z^p\bar z^q$, define
\[
  C:=-\sum_{p,q\geq0}
       \frac{G_{pq}(t)}{p+\gamma-1}z^p\bar z^q,
  \qquad A:=\widetilde A+H_C.
\]
The bound $p+\gamma-1\geq\gamma-1>0$ gives convergence in the same
regularity class. A straight forward computation implies
$$z\partial_zC+(\gamma-1)C=-G,$$ so $\diver(A\,du_\gamma)=0$.

As $G|_{z=0}=0$ and $G=0$ near $t=1$ and $t=3$,
the defining series for $C$ gives
\[
  C|_{z=0}=0,
  \qquad
  C=0\text{ near }t=1\text{ and }t=3.
\]
Thus $A=A_-$ near $t=1$ and $A=A_+$ near $t=3$.
Moreover, $H_C=O(|z|)$ uniformly in $t$, and
$\widetilde A\geq c\,\mathrm{I}_3$ on $J\times\{0\}$
for some constant $c>0$.
After decreasing $\eps$, the matrix $A$ is therefore positive definite.
\end{proof}

For $p\in P$, the ray $\R_{>0}p$ is a branch ray.
For $p\in Q_f$, it is an ordinary zero ray of $dF$.

\begin{lemma}
  \label[lemma]{lem:exact-ray-normal-form}
Let $p\in P\cup Q_f$, and let $J\subset\R_{>0}p$ be a compact
segment parametrized by $t\in[1,3]$.
There are tubular coordinates $(t,z)$ along $J$ in which
\[
  F=u_\gamma,
  \qquad
  \gamma\in
  \begin{cases}
    \{3/2,5/2,\ldots\},&p\in P,\\
    \{2,3,\ldots\},&p\in Q_f.
  \end{cases}
\]
On a sufficiently small tube $J\times D_\eps$, there is a smooth
positive inverse metric density $A$ satisfying
\[
  \diver(A\,du_\gamma)=0,
  \qquad
  A=A_{\Euc}\text{ near }t=1,
  \qquad
  A=\mathrm{I}_3\text{ near }t=3.
\]
\end{lemma}

\begin{proof}
Choose a stereographic coordinate $z$ centered at $p$.
Factoring the convergent local expansion of $f$ gives
\[
  F(t,z)=\operatorname{Re}\bigl(z^\gamma a(t,z,\bar z)\bigr),
  \qquad a(t,0)\neq0.
\]
The exponent $\gamma$ is half-integral for $p\in P$ and integral
for $p\in Q_f$; see \cite[Section~4]{TaubesWuModels}
for the branch case.
Criticality of the eigensection gives $\gamma\geq3/2$ on a branch ray.
At an ordinary direction $q\in Q_f$, the equations $f(q)=df(q)=0$
give $\gamma\geq2$.
The function $a$ is real analytic in the normal variables with the
uniform convergence specified above.
Its dependence on $t$ is smooth because the radial factor $R^\mu$
is smooth on $J$.

After shrinking the tube, $a$ is nowhere zero.
Since the tube is contractible, choose a smooth branch of $a^{1/\gamma}$.
Consider the map
\[
  (t,z)\longmapsto\bigl(t,za(t,z,\bar z)^{1/\gamma}\bigr).
\]
At $z=0$, the real differential in the normal variables $(x,y)$ satisfies
\[
  D_{(x,y)}\bigl(za^{1/\gamma}\bigr)(t,0)[v]
  =a(t,0)^{1/\gamma}v,
  \qquad v\in\C.
\]
This real linear map is invertible because $a(t,0)\neq0$.
The inverse function theorem therefore gives coordinates on a smaller tube.
Denote the new normal coordinate again by $z$; then $F=u_\gamma$.

The coordinate change and its inverse are real analytic in the normal
variables, so $A_{\Euc}$ has the regularity required in
\cref{lem:fixed-order-collar}.
The original harmonic equation gives
$\diver(A_{\Euc}\,du_\gamma)=0$.
Also, $u_\gamma$ is independent of $t$ and harmonic in $z$, so
$\diver(\mathrm{I}_3\,du_\gamma)=0$.
Apply that lemma with $A_-=A_{\Euc}$ and $A_+=\mathrm{I}_3$.
\end{proof}

\subsubsection{Changing the vanishing order}

We next connect the models $u_\gamma$ and $u_{\gamma+1}$.
The metric is $dt^2+|dz|^2$ near both ends, and the differential
has no zeros off the axis.

\begin{lemma}\label[lemma]{lem:phase-adjacent-transition}
Let $\gamma\in\tfrac12\mathbb Z$ with $\gamma>1$.
For sufficiently small $\eps>0$, there are a smooth positive inverse
metric density $A$ on $J\times D_\eps$ and a smooth $\cI$-valued
function $u$ on $J\times(D_\eps\setminus\{0\})$ satisfying
\[
  \diver(A\,du)=0,
  \qquad du\neq0\text{ for }z\neq0.
\]
In suitable normal coordinates near the two ends,
\[
  (u,A)=
  \begin{cases}
    (u_\gamma,\mathrm{I}_3),&\text{near }t=1,\\
    (u_{\gamma+1},\mathrm{I}_3),&\text{near }t=3.
  \end{cases}
\]
\end{lemma}

\begin{proof}
Choose $b_0>0$ and $b\in\mathcal{C}^\infty(J;[0,b_0])$ such that
\[
  b=b_0\text{ near }t=1,
  \qquad b=0\text{ near }t=3,
  \qquad b'<0\text{ where }0<b<b_0.
\]
Choose $\vartheta\in\R$ satisfying
\begin{equation}\label{eq:phase-condition}
  \cos\bigl((\gamma+1)\vartheta+\pi\gamma\bigr)\neq0,
\end{equation}
and define
\begin{equation}\label{eq:phase-block-potential}
  u(t,z):=\operatorname{Re}
  \bigl(e^{i\vartheta}b(t)z^\gamma+z^{\gamma+1}\bigr).
\end{equation}
To make $u$ harmonic, define
\[
  H:=
  \begin{pmatrix}
    \operatorname{Re}(e^{i\vartheta}z)
      -2\gamma\operatorname{Re}(e^{i\vartheta}\bar z)
      &\operatorname{Im}(e^{i\vartheta}z)\\
    \operatorname{Im}(e^{i\vartheta}z)
      &-\operatorname{Re}(e^{i\vartheta}z)
       -2\gamma\operatorname{Re}(e^{i\vartheta}\bar z)
  \end{pmatrix},
\]
and define
\[
  A:=\operatorname{diag}
  \left(1,\mathrm{I}_2-\frac{b''(t)}{2(\gamma+1)}H\right).
\]
Direct differentiation gives
\[
  \diver(A\,du)
  =b''\operatorname{Re}(e^{i\vartheta}z^\gamma)
   -b''\operatorname{Re}(e^{i\vartheta}z^\gamma)=0.
\]
Since $H=O(|z|)$ uniformly in $t$, the matrix $A$ is positive
definite on a sufficiently small tube.

We now check that $du$ has no zeros off the axis.
For $z\neq0$, the equations $\partial_xu=\partial_yu=0$ require
\[
  z=z_*(t):=-\frac{\gamma}{\gamma+1}e^{i\vartheta}b(t).
\]
Where $0<b<b_0$, substitute this value into
$\partial_tu=b'\operatorname{Re}(e^{i\vartheta}z^\gamma)$.
Condition \eqref{eq:phase-condition} gives
$\partial_tu(t,z_*(t))\neq0$.
Where $b=b_0$, the bound $\eps<\gamma b_0/(\gamma+1)$ excludes
$z_*(t)$ from the tube.
Where $b=0$, the only possible zero is on the axis.

Near $t=1$, we have $b=b_0$.
On a smaller disk, make the holomorphic coordinate change
\[
  z\longmapsto z(e^{i\vartheta}b_0+z)^{1/\gamma},
\]
and continue to denote the new coordinate by $z$.
Then $u=u_\gamma$, and the transformed density is real analytic
in the normal variables.
Apply \cref{lem:fixed-order-collar} in this region to make $A=\mathrm{I}_3$
near $t=1$, keeping $u$ fixed and retaining the constructed density
at the other end of the region.
Near $t=3$, we already have $b=b''=0$, so
$(u,A)=(u_{\gamma+1},\mathrm{I}_3)$.
\end{proof}

We summarize these constructions in the following result on local
modifications near the rays.

\begin{corollary}\label[corollary]{cor:ray-standard-model}
With $p$ and $J$ as in \cref{lem:exact-ray-normal-form},
there exist, for sufficiently small $\eps>0$, a smooth positive
inverse metric density $A$ on $J\times D_\eps$ and a smooth
$\cI$-valued function $u$ on $J\times(D_\eps\setminus\{0\})$
satisfying
\[
  \diver(A\,du)=0,
  \qquad du\neq0\text{ for }z\neq0.
\]
Near $t=1$, we have
\[
  (u,A)=(F,A_{\Euc}).
\]
In suitable normal coordinates near $t=3$, we have
\[
  (u,A)=
  \begin{cases}
    (u_{3/2},\mathrm{I}_3),&p\in P,\\
    (u_2,\mathrm{I}_3),&p\in Q_f.
  \end{cases}
\]
If $p\in Q_f$, the function $u$ extends smoothly across $z=0$.
\end{corollary}

\begin{proof}
The result follows directly from
\cref{lem:exact-ray-normal-form,lem:phase-adjacent-transition}.
\end{proof}

\subsubsection{Modifying ordinary zero rays}

For an ordinary zero ray, we can further modify the model $u_2$
to remove the zeros of $du$ near one end of the segment.
The original function and metric remain unchanged near the other end.

\begin{lemma}\label[lemma]{lem:ordinary-ray-cap}
Let $q\in Q_f$, and let $J\subset\R_{>0}q$ be a compact segment
parametrized by $t\in[1,3]$.
On a sufficiently small tube $J\times D_\eps$ in tubular coordinates
$(t,z)$, there exist a smooth function $u$ and a smooth positive inverse
metric density $A$ satisfying
\[
  \diver(A\,du)=0.
\]
The pair $(u,A)$ can be chosen to satisfy either
\[
\begin{gathered}
  Z(du)=[2,3]\times\{0\},\\
  (u,A)=(F,A_{\Euc})\text{ near }t=3,
  \qquad A=\mathrm{I}_3\text{ near }t=1,
\end{gathered}
\]
or
\[
\begin{gathered}
  Z(du)=[1,2]\times\{0\},\\
  (u,A)=(F,A_{\Euc})\text{ near }t=1,
  \qquad A=\mathrm{I}_3\text{ near }t=3.
\end{gathered}
\]
\end{lemma}

\begin{proof}
On $t>2$, apply \cref{cor:ray-standard-model} with the direction
reversed to keep $(F,A_{\Euc})$ unchanged near $t=3$ and obtain
$(u_2,\mathrm{I}_3)$ near $t=2$.
The resulting differential vanishes exactly on the axis in this
portion of the tube.
We extend it to $t\leq2$.

Choose $b_0>0$ and $b\in\mathcal{C}^\infty(J;[0,b_0])$ such that
\[
  b=b_0\text{ for }t\leq3/2,
  \qquad b'<0\text{ for }3/2<t<2,
  \qquad b=0\text{ for }t\geq2.
\]
On $t\leq2$, define
\begin{equation}\label{eq:ordinary-cap-model}
\begin{aligned}
  u(t,x,y)&:=b(t)x+x^2-y^2,\\
  A&:=\operatorname{diag}\left(
    1,\mathrm{I}_2+\frac{b''(t)}4
      \begin{pmatrix}x&-y\\-y&3x\end{pmatrix}
  \right).
\end{aligned}
\end{equation}
All derivatives of $b$ vanish at $t=2$, so these expressions join
smoothly to $(u_2,\mathrm{I}_3)$ on $t>2$.
On a sufficiently small tube, $A$ is positive definite and
$\diver(A\,du)=b''x-b''x=0$.

The critical point equations for $u$ are
\[
  y=0,\qquad x=-b(t)/2,\qquad b(t)b'(t)=0.
\]
On $3/2<t<2$, the product $bb'$ is nonzero.
On $t\leq3/2$, choose $\eps<b_0/2$ to exclude $x=-b_0/2$.
Thus $du\neq0$ for $t<2$.
Together with the construction on $t\geq2$, this gives
$Z(du)=[2,3]\times\{0\}$.
On $t\leq3/2$, we also have $A=\mathrm{I}_3$ and
\[
  \partial_xu=b_0+2x\geq b_0-2\eps>0.
\]

For the second choice, start with $(F,A_{\Euc})$ near $t=1$
and carry out the same construction towards $t=3$.
\end{proof}

\subsubsection{Modified vertex models}

Combining \cref{cor:ray-standard-model} with
\cref{lem:ordinary-ray-cap}, we obtain the following modification
of a vertex model.

\begin{corollary}\label[corollary]{cor:modified-vertex-model}
Let $F$ be the homogeneous model \eqref{eq:homogeneous-eigencone}
associated with $f$, with branch directions $P$ and ordinary zero
directions $Q_f$.
Let $B\subset\R^3$ be a closed ball centered at the vertex $0$.
For each $q\in Q_f$, choose $r_q>0$ such that $r_q q\in\interior B$.
There exist an open neighborhood $U$ of
\[
  (B\cap\R_{\geq0}P)\cup\bigcup_{q\in Q_f}[0,r_q]q,
\]
a smooth metric $g$ on $U$, and a smooth $\cI$-valued function $u$
on $U\setminus\R_{\geq0}P$ satisfying
\[
  \Delta_gu=0,
  \qquad
  \Zero(du)
  =(U\cap\R_{\geq0}P)\cup\bigcup_{q\in Q_f}[0,r_q]q.
\]
Near $0$, we have $(u,g)=(F,g_{\Euc})$.
Near each point of $\partial B\cap\R_{>0}P$, there are tubular
coordinates $(t,z)$ in which
\[
  (u,g)=\bigl(u_{3/2},dt^2+|dz|^2\bigr).
\]
\end{corollary}

\subsection{Gluing the vertex models}

Let $G$ be a finite graph with positive even valence at each vertex.
For a tame embedding $\iota:G\hookrightarrow\interior B^3$,
define $Z:=\iota(G)$.
Since every vertex has even valence, there is a flat real line bundle
$\cI\to B^3\setminus Z$ with holonomy $-1$ around every edge meridian
\cite[Section~2.1]{HaydysMazzeoTakahashiIndex}.

At each vertex of valence $2n\geq4$, choose a modified model
from \cref{cor:modified-vertex-model} with $2n$ branch rays.
At a two-valent vertex, use $(u_{3/2},dt^2+|dz|^2)$.
Choose $\iota$ smooth away from the vertices so that, in the
coordinates of each local model, $Z$ is the union of its branch rays.
Define $Z_{\mathrm{ord}}$ to be the union of the ordinary zero
segments in these models.

\begin{proposition}
  \label[proposition]{prop:harmonic-graph-neighborhood}
There exist a compact regular neighborhood $N\subset\interior B^3$ of
$Z\cup Z_{\mathrm{ord}}$ with smooth boundary, a smooth metric $g_0$
near $N$, and a smooth $\cI$-valued function $F_0$ near $N\setminus Z$
satisfying
\[
  \Delta_{g_0}F_0=0,
  \qquad
  \Zero(dF_0)=Z\cup Z_{\mathrm{ord}}.
\]
The triple $(Z\cup Z_{\mathrm{ord}},\cI,dF_0)$ is a
$\Ztwo$-harmonic 1-form on $N$ with $\Mono(\cI)=Z$.
\end{proposition}

\begin{proof}
For each vertex $v$, let $(F_v,g_v)$ be the chosen local model.
By \cref{cor:modified-vertex-model}, on the regions where we
attach the edge tubes, we have
\[
  (F_v,g_v)=\bigl(u_{3/2},dt^2+|dz|^2\bigr).
\]
Use the same product model on each edge tube.
These pairs agree on the overlaps and therefore glue together
to give $(F_0,g_0)$ near a compact regular neighborhood
$N\subset\interior B^3$ of $Z\cup Z_{\mathrm{ord}}$ with smooth boundary.
By construction, $(Z\cup Z_{\mathrm{ord}},\cI,dF_0)$ is a
$\Ztwo$-harmonic 1-form on $N$.
\end{proof}

Figure~\ref{fig:figure-eight-ordinary-zero-set} illustrates a regular
neighborhood of $Z\cup Z_{\mathrm{ord}}$ for $Z=S^1\vee_p S^1$.

\begin{figure}[H]
  \centering
  \includegraphics[width=.80\textwidth]
    {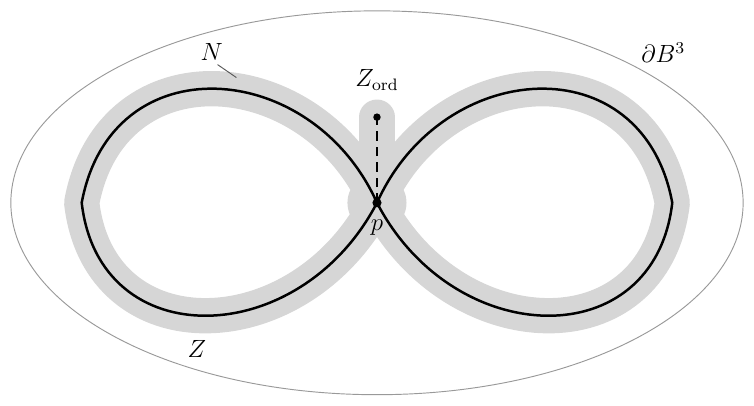}
  \caption{A schematic projection of the regular neighborhood $N$
  (shaded), with $Z$ solid and $Z_{\mathrm{ord}}$ dashed.}
  \label{fig:figure-eight-ordinary-zero-set}
\end{figure}

\subsection{Global extension by the relative
\texorpdfstring{$h$}{h}-principle}

We extend $(dF_0,g_0)$ to $B^3$, keeping it unchanged near $N$
and introducing no zeros outside $N$.
We first construct a positive pair using the relative $h$-principle
\cite[Theorem~7.2.4]{EliashbergMishachev}.
Then \cref{thm:relative-metric-realization} gives the metric.

Fix the data from \cref{prop:harmonic-graph-neighborhood} and define
$W:=B^3\setminus\interior N$ and $S:=\partial N$.
The manifold $W$ is connected, with boundary $\partial W=S\sqcup\partial B^3$.
Choose a collar $S\times(-2\eps,2\eps)$ on which $dF_0$ is nonzero, with $S\times[0,2\eps)\subset W$.

\begin{lemma}
  \label[lemma]{lem:interface-primitive}
There is a $\cI$-valued 1-form $\lambda_0$ on
$S\times(-2\eps,2\eps)$ such that $d\lambda_0
  =\star_{g_0}dF_0.$
\end{lemma}

\begin{proof}
Each connected component of $S$ contains an edge meridian along which
$\cI$ has holonomy $-1$, so $H^2(S;\cI)=0$ by Poincar\'e duality.
As $\star_{g_0}dF_0$ is closed, such $\lambda_0$ exist on the collar.
\end{proof}

On the collar, we have
\[
  dF_0\wedge d\lambda_0
  =|dF_0|_{g_0}^2\,\vol_{g_0}>0.
\]
Thus $(F_0,\lambda_0)$ is a positive pair there.
Define $K_{\mathrm{in}}:=S\times[0,\eps]$.
We will keep $(F_0,\lambda_0)$ unchanged near $K_{\mathrm{in}}$.

We seek a smooth section $F$ and a smooth 1-form $\lambda$ on $W$
such that $F=F_0$ and $\lambda=\lambda_0$ near $K_{\mathrm{in}}$.
We also seek a 1-form $p$ and a two-tensor $Q$ satisfying
$p=dF_0$ and $Q=\nabla\lambda_0$ near $K_{\mathrm{in}}$.
We use a torsion-free connection, coupled to the flat connection on
$\cI$, to define $\nabla\lambda$.
For a $\cI$-valued two-tensor $Q$, define
\[
  (\operatorname{Alt}Q)(X,Y):=Q(X,Y)-Q(Y,X).
\]
The following lemma constructs these extensions with
$p\wedge\operatorname{Alt}Q>0$ on $W$.

\begin{lemma}
  \label[lemma]{lem:graph-extension-data}
There exist $F\in\mathcal{C}^\infty(W;\cI)$,
$\lambda,p\in\Omega^1(W;\cI)$, and
$Q\in\Gamma(T^*W\otimes T^*W\otimes\cI)$ such that
$p\wedge\operatorname{Alt}Q>0$ on $W$ and
\[
  (F,\lambda,p,Q)=(F_0,\lambda_0,dF_0,\nabla\lambda_0)
  \quad\text{near }K_{\mathrm{in}}.
\]
\end{lemma}

\begin{proof}
We first extend $dF_0$ to a nowhere zero section of
$E:=T^*W\otimes\cI$.
The complement of the zero section in $E$ has fiber
$\R^3\setminus\{0\}\simeq S^2$.
Since this fiber is connected and simply connected, $dF_0$ extends
over the relative two-skeleton of $(W,K_{\mathrm{in}})$.
The only remaining obstruction lies in
$H^3(W,K_{\mathrm{in}};\mathfrak{o}_E)$, where $\mathfrak{o}_E$ is the
integral orientation local system of $E$
\cite[Sections~3.H and~4.3]{HatcherAlgebraicTopology}.
Since $W$ is oriented, the monodromy of $\mathfrak{o}_E$ is given by
the holonomy of $\cI$.
As $K_{\mathrm{in}}$ retracts onto $S$, Poincar\'e--Lefschetz duality gives
\begin{equation}\label{eq:graph-relative-obstruction}
  H^3(W,K_{\mathrm{in}};\mathfrak{o}_E)
  \cong H_0(W,\partial B^3;\mathfrak{o}_E)=0.
\end{equation}
The last equality follows because $W$ is connected and
$\partial B^3\neq\varnothing$.
Thus there is a continuous nowhere zero section $p$ of $E$
such that $p=dF_0$ near $K_{\mathrm{in}}$.
We smooth $p$ by a sufficiently small perturbation so that
$p=dF_0$ near $K_{\mathrm{in}}$ and $p\neq0$ on $W$.

Choose a smooth metric $g_{\mathrm{aux}}$ on $W$ equal to $g_0$ near
$K_{\mathrm{in}}$.
Multiplying $F_0,\lambda_0$ by a cutoff equal to one near
$K_{\mathrm{in}}$ and supported in the collar gives smooth extensions
$F,\lambda$ to $W$.
Viewing 2-forms as alternating two-tensors, define
\[
  Q:=\nabla\lambda+
  \frac12\bigl(\star_{g_{\mathrm{aux}}}p-d\lambda\bigr).
\]
Since $\operatorname{Alt}(\nabla\lambda)=d\lambda$, we have
\[
  \operatorname{Alt}Q=\star_{g_{\mathrm{aux}}}p,
  \qquad
  p\wedge\operatorname{Alt}Q
  =|p|_{g_{\mathrm{aux}}}^2\,\vol_{g_{\mathrm{aux}}}>0.
\]
Near $K_{\mathrm{in}}$, we have
$\star_{g_{\mathrm{aux}}}p=\star_{g_0}dF_0=d\lambda_0=d\lambda$.
Hence $Q=\nabla\lambda_0$ near $K_{\mathrm{in}}$.
\end{proof}

For a sufficiently small inner collar, $W\setminus K_{\mathrm{in}}$
is connected and meets $\partial B^3$.  Thus the compression
construction in
\cite[pp.~40--41 and~68]{EliashbergMishachev} gives the following.

\begin{lemma}
  \label[lemma]{lem:relative-compression}
There is a polyhedron $P\subset W$ of positive codimension, that is,
a subcomplex of dimension at most two in a smooth triangulation of $W$,
with the following property.  For every neighborhood $\mathcal U$ of
$K_{\mathrm{in}}\cup P$, there is a smooth family of embeddings
$c_t:W\hookrightarrow W$, $0\leq t\leq1$, satisfying
\[
  c_0=\id_W,
  \qquad c_t=\id_W\quad\text{near }K_{\mathrm{in}},
  \qquad c_1(W)\subset\mathcal U.
\]
\end{lemma}

The polyhedron $P$ in \cref{lem:relative-compression} may be chosen in a
triangulation compatible with $K_{\mathrm{in}}$, with
$P\setminus K_{\mathrm{in}}\subset\interior W$.
Fix a smooth background metric $g_{\mathrm{aux}}$ on $W$; all pointwise norms below use
$g_{\mathrm{aux}}$ and the flat metric on $\cI$.
The holonomic approximation theorem applies to sections of arbitrary smooth
fiber bundles, in particular to $\cI\oplus(T^*W\otimes\cI)$, with
derivatives expressed using the chosen connection.

\begin{lemma}
  \label[lemma]{lem:holonomic-approximation}
\cite[Theorem~1.2.1]{EliashbergMishachevHolonomic}
Let
\[
  F\in \mathcal{C}^\infty(W;\cI),
  \qquad \lambda,p\in\Omega^1(W;\cI),
  \qquad Q\in\Gamma(T^*W\otimes T^*W\otimes\cI)
\]
satisfy $p=dF$ and $Q=\nabla\lambda$ near $K_{\mathrm{in}}$.
For every $\delta>0$, one can choose $P$ as in
\cref{lem:relative-compression} and a neighborhood $\mathcal U$ of
$K_{\mathrm{in}}\cup P$ on which there exist
$f\in \mathcal{C}^\infty(\mathcal U;\cI)$ and
$\eta\in\Omega^1(\mathcal U;\cI)$ such that
\begin{equation}\label{eq:holonomic-approximation}
  |f-F|+|\eta-\lambda|+|df-p|+|\nabla\eta-Q|<\delta
  \quad\text{on }\mathcal U,
\end{equation}
and $(f,\eta)=(F,\lambda)$ near $K_{\mathrm{in}}$.
\end{lemma}

We apply the relative $h$-principle directly on $W$, keeping the data fixed
near $K_{\mathrm{in}}$ and imposing no condition on
$\partial B^3$.

\begin{theorem}
  \label{thm:relative-positive-pair-hprinciple}
\cite{EliashbergMishachev}
Suppose that every connected component of $W\setminus K_{\mathrm{in}}$
meets $\partial B^3$.
Let $F\in\mathcal{C}^\infty(W;\cI)$,
$\lambda,p\in\Omega^1(W;\cI)$, and
$Q\in\Gamma(T^*W\otimes T^*W\otimes\cI)$ satisfy
$p\wedge\operatorname{Alt}Q>0$ on $W$, with $p=dF$ and
$Q=\nabla\lambda$ near $K_{\mathrm{in}}$.
Then there exist $\widetilde F\in\mathcal{C}^\infty(W;\cI)$ and
$\widetilde\lambda\in\Omega^1(W;\cI)$ such that
$d\widetilde F\wedge d\widetilde\lambda>0$ on $W$ and
\[
  (\widetilde F,\widetilde\lambda)=(F,\lambda)
  \quad\text{near }K_{\mathrm{in}}.
\]
\end{theorem}

\begin{proof}
We adapt the argument in
\cite[Section~4.3 and the proof of Theorem~7.2.2]{EliashbergMishachev}
to $\cI$-valued functions and 1-forms.

For $\delta>0$, \cref{lem:holonomic-approximation} gives
$P$, a neighborhood $\mathcal U$ of $K_{\mathrm{in}}\cup P$,
and sections $f,\eta$ on $\mathcal U$ satisfying
\eqref{eq:holonomic-approximation}, with $(f,\eta)=(F,\lambda)$ near
$K_{\mathrm{in}}$.
Since $d\eta=\operatorname{Alt}(\nabla\eta)$, the approximation makes
$df$ and $d\eta$ close to $p$ and $\operatorname{Alt}Q$, respectively.
Compactness of $W$ and $p\wedge\operatorname{Alt}Q>0$ therefore allow us
to choose $\delta$ sufficiently small that
$df\wedge d\eta>0$ on $\mathcal U$.
Choose a family $c_t$ as in \cref{lem:relative-compression} for this
neighborhood $\mathcal U$.
The covering homotopy property lifts $c_t$
to the double cover defining $\cI$, starting at the identity.  The lifts
commute with the deck involution and give flat identifications
$c_t^*\cI\cong\cI$, equal to the identity near $K_{\mathrm{in}}$.
Using these identifications, set $\widetilde F=c_1^*f$ and
$\widetilde\lambda=c_1^*\eta$.
Flatness makes exterior differentiation commute with these pullbacks.
Since $c_1$ preserves orientation,
\[
  d\widetilde F\wedge d\widetilde\lambda
  =c_1^*(df\wedge d\eta)>0.
\]
The pair equals $(F,\lambda)$ near $K_{\mathrm{in}}$, since the isotopy
and the bundle identifications are the identity there.
\end{proof}

\begin{proof}[Proof of \cref{thm:intro-even-graphs}]
For a sufficiently small inner collar, $W\setminus K_{\mathrm{in}}$
is connected and meets $\partial B^3$.  Applying
\cref{thm:relative-positive-pair-hprinciple} to the data supplied by
\cref{lem:graph-extension-data} gives
$F\in\mathcal{C}^\infty(W;\cI)$ and
$\lambda\in\Omega^1(W;\cI)$ such that
\[
  dF\wedge d\lambda>0\quad\text{on }W,
  \qquad (F,\lambda)=(F_0,\lambda_0)\quad\text{near }K_{\mathrm{in}}.
\]
By \cref{thm:relative-metric-realization}, there is a smooth metric
$g$ on $W$ for which $dF$ is harmonic and $g=g_0$ near
$K_{\mathrm{in}}$.
Since $F=F_0$ near $K_{\mathrm{in}}$, we extend $F$ by $F_0$ and $g$ by $g_0$ from $N$ to $B^3$. The inequality $dF\wedge d\lambda>0$ gives $dF\neq0$ on $W$.
Hence, by \cref{prop:harmonic-graph-neighborhood}, $dF$ is a
$\Ztwo$-harmonic 1-form on $B^3$ with $\Mono(\cI)=Z$.
\end{proof}

\section{Desingularization of critical eigencones}
\label{sec:desingularization}

In this section, we prove \cref{thm:intro-desingularization}.
For the homogeneous model associated with a critical eigensection,
we replace the branch rays by smooth curves joining the paired ends.
The function and the Euclidean metric remain unchanged outside
a compact set.

\subsection{The local harmonic function}

Let $P=\{p_1,\dots,p_{2n}\}\subset S^2$, let $\phi$ be a critical
$\cI_P$-valued eigensection, and fix a pairing
\[
  \Pi=\bigl\{\{p_{i_a},p_{j_a}\}:1\le a\le n\bigr\}
\]
of $P$.
Write $F_{\mathrm{hom}}(r,\omega)=r^\mu\phi(\omega)$ for the
homogeneous model \eqref{eq:homogeneous-eigencone}.
Let $Q_\phi$ be the finite set of ordinary zero directions defined
in \eqref{eq:Qf-definition}.

For each $a=1,\dots,n$, choose a smooth embedded curve
$\gamma_a$ whose two ends agree with $\R_{>0}p_{i_a}$ and
$\R_{>0}p_{j_a}$ outside the unit ball.
Choose these curves pairwise disjoint, as illustrated in
\cref{fig:tetrahedral-pairing}.
Define $\Gamma:=\gamma_1\sqcup\cdots\sqcup\gamma_n$.
Let $\cI\to\R^3\setminus\Gamma$ be the flat real line bundle
with holonomy $-1$ around each meridian of $\Gamma$.
We use the same notation for its restrictions to open subsets.
On the radial end, identify $\cI$ with the radial pullback of $\cI_P$.

Fix a curve $\gamma_a$ and write $i=i_a$, $j=j_a$.
On a tubular neighborhood of $[2,3]p_i:=\{r p_i:2\le r\le 3\}$,
\cref{cor:ray-standard-model} gives a harmonic function and a
smooth metric.
Near $3p_i$, they agree with $(F_{\mathrm{hom}},g_{\Euc})$.
Near $2p_i$, the harmonic function and the metric have the form $u_{3/2}=\operatorname{Re}(z^{3/2})$ and  $dt^2+|dz|^2$
in tubular coordinates $(t,z)$ with $z=0$ on $\gamma_a$.
We make the same construction on $[2,3]p_j$.

Both models therefore have the same product form near
$2p_i$ and $2p_j$.
We connect them along $\gamma_a\cap\overline{B_2}$ using
$u_{3/2}$ and the product metric.
For $r\geq3$, we use $(F_{\mathrm{hom}},g_{\Euc})$ on both radial ends.
This gives a harmonic function for a smooth metric near $\gamma_a$.
Its differential vanishes precisely along $\gamma_a$.

Choose $R_0$ so that all modifications lie in $B_{R_0}$.
For each $q\in Q_\phi$, choose a small tubular neighborhood of a
segment of $\R_{>0}q$ containing $R_0q$.
Applying \cref{lem:ordinary-ray-cap}, we obtain a harmonic function
for a smooth metric on a neighborhood of $[R_0,\infty)q$.
Its differential vanishes precisely on $[R_0,\infty)q$.
The function and metric agree with $(F_{\mathrm{hom}},g_{\Euc})$
for sufficiently large $r$.
The resulting ordinary zero set is
\[
  [R_0,\infty)Q_\phi:=\{rq:r\geq R_0,\ q\in Q_\phi\}.
\]

Choose $R_1>R_0$ so that all modifications lie in $B_{R_1}$.
Choose small open tubular neighborhoods $N_\Gamma$ of $\Gamma$
and $N_Q$ of $[R_0,\infty)Q_\phi$ on which the local models are defined.
Define
\[
  U:=\{r>R_1\}\cup N_\Gamma\cup N_Q.
\]

The local models agree on their overlaps and define
$F_{\mathrm{loc}}\in\mathcal{C}^\infty(U\setminus\Gamma;\cI)$
and a smooth metric $g_{\mathrm{loc}}$ on $U$ satisfying
\[
  \Delta_{g_{\mathrm{loc}}}F_{\mathrm{loc}}=0
  \quad\text{on }U\setminus\Gamma,
  \qquad
  \Zero(dF_{\mathrm{loc}})=\Gamma\sqcup[R_0,\infty)Q_\phi.
\]
Moreover,
\[
  (F_{\mathrm{loc}},g_{\mathrm{loc}})
  =(F_{\mathrm{hom}},g_{\Euc})
  \qquad\text{on }\{r>R_1\}.
\]

\begin{figure}[H]
  \centering
  \includegraphics[width=.60\textwidth]
    {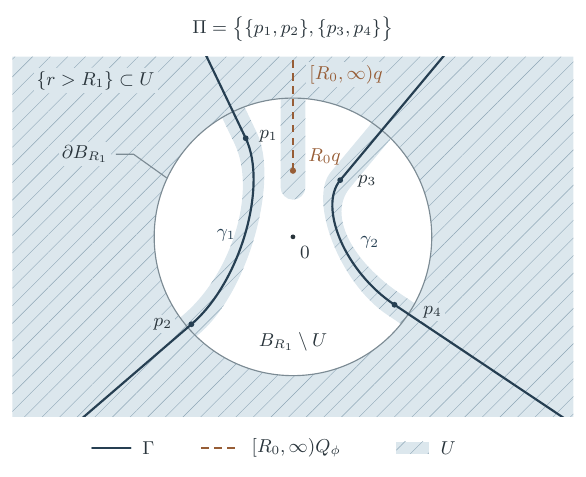}
  \caption{The local-model domain $U$ for the tetrahedral pairing,
with a possible ordinary zero ray $[R_0,\infty)q$, $q\in Q_\phi$.}
  \label{fig:tetrahedral-local-model-domain}
\end{figure}

\subsection{Extension of the odd function}

Let $\pi:\widetilde M\to\R^3$ be the double branched cover branched along
$\Gamma$, with involution $\tau$.
Give $\widetilde M$ the lifted orientation, so that $\tau$ preserves
orientation.
A tilde denotes a lift to this cover; in particular,
$\widetilde F_{\mathrm{hom}}=r^\mu\widetilde\phi$ on the radial end.
The lift $\widetilde F_{\mathrm{loc}}$ is smooth and odd on $\pi^{-1}(U)$.

We extend $\widetilde F_{\mathrm{loc}}$ to a smooth odd function on
$\widetilde M$ that equals $\widetilde F_{\mathrm{hom}}$ for sufficiently
large $r$.
The extension has only finitely many additional Morse critical points
and provides the increasing arcs used to construct the positive dual form.

\begin{proposition}\label[proposition]{prop:odd-indefinite-filling}
After increasing $R_1$ if necessary, there exists a smooth odd function
$\widetilde F$ on $\widetilde M$ with the following properties:
\begin{enumerate}
  \item $\widetilde F=\widetilde F_{\mathrm{loc}}$ near
  $\pi^{-1}(\Gamma\cup[R_0,\infty)Q_\phi)$ and
  $\widetilde F=\widetilde F_{\mathrm{hom}}$ on $\{r\geq R_1\}$;
  \item The critical points of $\widetilde F$ outside
  $\pi^{-1}(\Gamma\cup[R_0,\infty)Q_\phi)$ form $k\geq0$
  disjoint pairs $\{x_j,\tau x_j\}$, $1\leq j\leq k$.
  The points $x_j$ and $\tau x_j$ are Morse critical points of
  indices one and two, respectively;
  \item For every $x\in\pi^{-1}(\overline{B_{R_1}})$ with
  $d\widetilde F_x\neq0$, there is a smooth oriented embedded
  curve $\gamma$ through $x$ satisfying
  $d\widetilde F(\dot\gamma)>0$.
  Outside a compact set, $\gamma$ consists of two radial rays
  extending to infinity.
\end{enumerate}
\end{proposition}
We call a smooth oriented embedded curve $\gamma$
a \emph{positive arc} if $d\widetilde F(\dot\gamma)>0$.
\subsubsection{A separator with connected sides}

We begin the proof of \cref{prop:odd-indefinite-filling} by
constructing an auxiliary odd function on the compact region $Y$
defined below. The function will agree with the local
model near $\partial Y$, have zero as a regular value, and have
connected positive and negative regions. This connectedness will be
used in the next subsection.

To define $Y$, choose $R>R_1$ and put
$\widetilde B_R=\pi^{-1}(\overline{B_R})$.
Let $N$ be a compact, $\tau$-invariant union of sufficiently thin tubes around
\[
  \pi^{-1}\bigl(\Gamma\cup[R_0,\infty)Q_\phi\bigr)
  \cap\widetilde B_R,
\]
contained in $\pi^{-1}(N_\Gamma\cup N_Q)\cap\widetilde B_R$.
We choose $N$ so that $\pi^{-1}(R_0Q_\phi)$ lies in its interior,
and smooth the corners of $N$ in $\operatorname{int}\widetilde B_R$.
The lateral boundary of $N$ meets $\partial\widetilde B_R$ transversely.

Set $Y:=\overline{\widetilde B_R\setminus N}$.
Then $Y$ is a connected manifold with corners. By construction, $Y\cap\pi^{-1}\bigl(\Gamma\cup[R_0,\infty)Q_\phi\bigr)
  =\varnothing$. In particular $\tau$ acts freely on $Y$. We call the connected components of
$\overline{\partial Y\setminus\partial\widetilde B_R}$ and
$\partial Y\cap\partial\widetilde B_R$ the \emph{boundary faces} of $Y$,
and their boundary components the \emph{boundary circles} of $Y$.

\begin{lemma}\label[lemma]{lem:boundary-regularity}
The neighborhood $N$ can be chosen so that
$d\widetilde F_{\mathrm{loc}}\neq0$ near $\partial Y$
and $0$ is a regular value of the restriction of
$\widetilde F_{\mathrm{loc}}$ to each boundary face
and each boundary circle of $Y$.
\end{lemma}
\begin{proof}
Since $\pi$ is unbranched near $\partial Y$, we can check the
required regularity in local coordinates on $\R^3$.
On $N_\Gamma\cup N_Q$, use the local tubular coordinates
$(t,z)$, with $z=x+iy$, chosen in the construction of
$F_{\mathrm{loc}}$.
By \cref{lem:exact-ray-normal-form,lem:phase-adjacent-transition,lem:ordinary-ray-cap},
a local representative $u$ of $F_{\mathrm{loc}}$ is $u_\gamma$,
the function in \eqref{eq:phase-block-potential}, or
$b(t)x+x^2-y^2$ from \eqref{eq:ordinary-cap-model}.
Choose $\pi(N)$ locally by $|z|\leq\rho$ for sufficiently small
$\rho>0$, and take its inverse image under $\pi$.

Write $z=\rho e^{i\theta}$.
For $u=u_\gamma$, at each zero on $|z|=\rho$ we have
$\partial_\theta u=-\gamma\rho^\gamma\sin(\gamma\theta)\neq0$.
For $u$ in \eqref{eq:phase-block-potential} where $b'\neq0$,
the equations $u=\partial_tu=0$ imply $\partial_\rho u=0$.
Since $du\neq0$ for $z\neq0$ by \cref{lem:phase-adjacent-transition},
we obtain $\partial_\theta u\neq0$.
Where $b'=0$, either $b=0$ and $u=u_{\gamma+1}$, or $b=b_0$.
For $b=b_0$, the expression
$\rho^{-\gamma}u=b_0\cos(\vartheta+\gamma\theta)+\rho\cos((\gamma+1)\theta)$
has simple zeros for sufficiently small $\rho$.

For $u=b(t)x+x^2-y^2$ with $b'\neq0$, the equations
$u=\partial_tu=0$ force $x=y=0$, contradicting $|z|=\rho>0$.
Where $b=0$, we have $u=u_2$.
Where $b=b_0$, choose $\rho<b_0/2$.
At a zero with $0<|z|\leq\rho$, the identity $x(b_0+x)=y^2$
gives $x\geq0$ and $y\neq0$, so
$\partial_\theta u=-y(b_0+4x)\neq0$.
Close $\pi(N)$ by disks at a fixed $t$ where $b=b_0$.
On these disks, $\partial_xu=b_0+2x>0$.
Smooth the intersections of the disks with $|z|=\rho$ within
$\{b=b_0,\ 0<|z|\leq\rho\}$, keeping $\partial_\theta$
tangent to $\partial\pi(N)$.

Join the local choices of $\pi(N)$ where $u=u_\gamma$,
using a smooth positive radius $\rho=\rho(t,\theta)$.
At each zero of the restriction to $\partial\pi(N)$, we have
\[
  \partial_\theta\bigl[u(t,\rho(t,\theta)e^{i\theta})\bigr]
  =-\gamma\rho(t,\theta)^\gamma\sin(\gamma\theta)\neq0.
\]

On $Y\cap\partial\widetilde B_R$,
$\widetilde F_{\mathrm{loc}}=R^\mu\widetilde\phi$.
Since $Y\cap\pi^{-1}(R(P\cup Q_\phi))=\varnothing$,
zero is a regular value of
$\widetilde F_{\mathrm{loc}}|_{Y\cap\partial\widetilde B_R}$.
Choose the boundary circles of $Y$ so that their images under $\pi$
are $|z|=\rho$ in the coordinates of \cref{lem:exact-ray-normal-form}.
The calculation for $u_\gamma$ gives regularity on these circles.
Extend $\pi(N)\cap\partial B_R$ by radial dilations near $r=R$.
The identity
$F_{\mathrm{hom}}(r,\omega)=(r/R)^\mu F_{\mathrm{hom}}(R,\omega)$
preserves the nonzero angular derivative at zeros on
$\partial Y\setminus\partial\widetilde B_R$ near $r=R$.
Finally, $\partial Y$ is disjoint from the critical set of
$\widetilde F_{\mathrm{loc}}$, so continuity gives
$d\widetilde F_{\mathrm{loc}}\neq0$ near $\partial Y$.
\end{proof}

\begin{lemma}\label[lemma]{lem:connected-odd-separator}
Fix $N$ as in \cref{lem:boundary-regularity}.
There is a smooth odd Morse function $h:Y\to\R$, equal to
$\widetilde F_{\mathrm{loc}}$ near $\partial Y$, such that zero
is a regular value and both
$\{h>0\}$ and $\{h<0\}$ are connected.
\end{lemma}

\begin{proof}
Choose a smooth $\tau$-invariant cutoff supported where
$\widetilde F_{\mathrm{loc}}$ is defined and equal to one near
$\partial Y$.
Multiplying $\widetilde F_{\mathrm{loc}}$ by this cutoff and extending
by zero gives a smooth odd function $h:Y\to\R$.
It agrees with $\widetilde F_{\mathrm{loc}}$ near $\partial Y$,
so $dh\neq0$ there by \cref{lem:boundary-regularity}.
Since $\tau$ acts freely on $Y$, relative equivariant transversality
allows us to perturb $h$ through odd functions, keeping it fixed
on a boundary neighborhood, so that $h$ has only Morse critical
points and zero is a regular value.
We continue to denote the perturbed function by $h$.

Define $\Omega_\pm:=\{x\in Y:\ \pm h(x)>0\}$.
Note that $\tau(\Omega_\pm)=\Omega_{\mp}$.
Since $h=\widetilde F_{\mathrm{loc}}$ near $\partial Y$,
\cref{lem:boundary-regularity} implies that zero is also a regular
value of the restriction of $h$ to each boundary face and circle.
Thus $\overline{\Omega_\pm}=\{x\in Y:\ \pm h(x)\geq0\}$
are compact smooth three-manifolds with corners.

We now modify $h$, keeping it fixed near $\partial Y$, so that both
$\Omega_+$ and $\Omega_-$ are connected.
If $\Omega_+$ is disconnected, connectedness of $Y$ gives two
distinct components $A,B$ of $\Omega_+$ and an embedded path
$\gamma:[0,1]\to\operatorname{int}Y$ satisfying
$\gamma(0)\in\partial A$, $\gamma(1)\in\partial B$, and
$\gamma((0,1))\subset\Omega_-$.
Choose $\gamma$ transverse to $h^{-1}(0)$ at its endpoints.
Since $h$ has finitely many critical points and $\tau$ acts freely,
general position allows us to choose $\gamma$ disjoint from
the critical set of $h$ and from $\tau\gamma$.

Extend $\gamma$ slightly into $A$ and $B$ at its two ends.
Choose a thin tubular neighborhood
$T_\gamma\Subset\operatorname{int}Y$ disjoint from
$\tau T_\gamma$ and from the critical set of $h$.
Inside $T_\gamma$, attach a thin $1$-handle to
$\overline{\Omega_+}$ joining $A$ and $B$, and smooth the corners
at the attaching disks. Choose a standard Morse function on the handle that vanishes
on its lateral boundary, is positive on the interiors of the
attaching disks, and has a single critical point of index $2$
with positive critical value. See \cref{fig:connected-separator-handle}.

Then extend this function smoothly to $T_\gamma$ so that zero
remains a regular value and it agrees with $h$ near
$\partial T_\gamma$.
We can perturb the resulting function sufficiently slightly
in $T_\gamma$ to make it Morse, keeping it fixed near
$\partial T_\gamma$, the $1$-handle. This defines the modification of $h$ inside $T_\gamma$.
Moreover, we can choose the modification so that it preserve the sign of the
original $h$ away from the handle.

Extend the modification to $\tau T_\gamma$ by setting
$h(\tau x)=-h(x)$ for $x\in T_\gamma$, and leave $h$
unchanged outside $T_\gamma\cup\tau T_\gamma$.
Continue to write $\Omega_\pm={\pm h>0}$.
Each paired modification reduces the number of connected
components of each of $\Omega_+$ and $\Omega_-$ by one.
These numbers are finite because $\overline{\Omega_\pm}$
are compact manifolds with corners.
After finitely many repetitions, we obtain the required
odd Morse function.
\end{proof}

\begin{figure}[H]
  \centering
  \includegraphics[width=.80\textwidth]{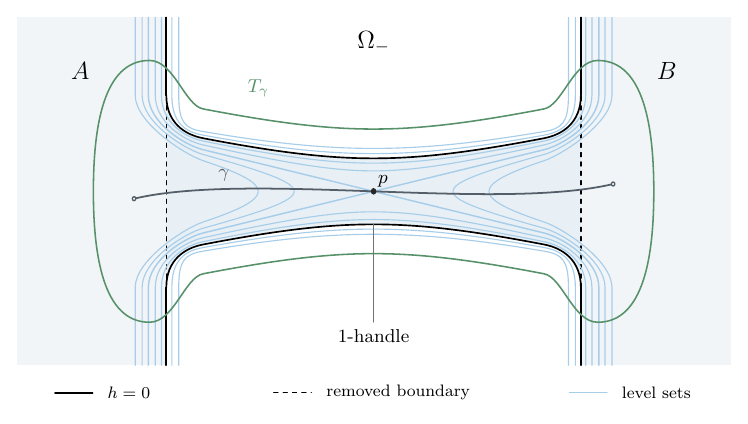}
  \caption{The modification of $h$ to reduce number of components of $\Omega_+$, with $p$ the critical point of Morse index $2$.}
  \label{fig:connected-separator-handle}
\end{figure}

\subsubsection{Completion of the odd extension}

We now complete the proof of \cref{prop:odd-indefinite-filling}
using the connected separator from \cref{lem:connected-odd-separator}.
We use the relative cancellation lemma below to construct an odd extension
whose additional critical points have index one or two.
We then verify the positive-arc property.

\begin{lemma}\label[lemma]{lem:indefinite-relative-filling}
Let $a<b$, and let $P_0$ and $P_1$ be nonempty compact smooth
surfaces, possibly with boundary.
Let $(C,P_0,P_1)$ be a connected compact smooth cobordism
of dimension three, viewed as a manifold with corners.
Write $\partial C=P_0\cup S\cup P_1$ and fix an identification
$S\cong\partial P_0\times[a,b]$ under which
$S\cap P_0=\partial P_0$ corresponds to $t=a$ and
$S\cap P_1=\partial P_1$ corresponds to $t=b$.

Let $f:C\to[a,b]$ be a Morse function with $df\neq0$
near $\partial C$.
Suppose $f|_{P_0}=a$, $f|_{P_1}=b$,
and $f(y,t)=t$ on $S$, where $(y,t)$ is the coordinate of $\partial P_0\times[a,b]$.
Then there is a Morse function $f_1:C\to[a,b]$, equal to $f$
near $\partial C$, whose critical points have index $1$ or $2$.
\end{lemma}
\begin{proof}
Choose a gradient-like vector field for $f$ tangent to $S$.
Since $C$ is connected and $P_0$ and $P_1$ are nonempty, the cancellation
argument in \cite[Theorems~5.4 and~8.1]{MilnorHCobordism},
applied to $f$ and then to $-f$, removes all critical points of
index $0$ and $3$.
The cancellations can be performed with compact support in
$\operatorname{int}C$, giving a Morse function $f_1$ equal to $f$
near $\partial C$.
\end{proof}

\begin{proof}[Proof of \cref{prop:odd-indefinite-filling}]
\textbf{Proof of (i) and (ii).}
Let $h$ be given by \cref{lem:connected-odd-separator}, and define
\[
H(x):=
\begin{cases}
	h(x),&x\in Y,\\
	\widetilde F_{\mathrm{loc}}(x),&x\in N,\\
	\widetilde F_{\mathrm{hom}}(x),
	&x\in\widetilde M\setminus\widetilde B_R.
\end{cases}
\]
These functions agree near their common boundaries, so $H$ is
smooth and odd on $\widetilde M$.
Define $\Omega^\pm:=\{x\in\widetilde B_R:\pm H(x)>0\}$.

\textbf{Step 1:} We show that $\Omega^\pm$ are connected.\par\nopagebreak

By the choice of $h$, the sets $\Omega^\pm\cap Y$ are connected.
It remains to connect every point of $\Omega^\pm\cap N$
to $\Omega^\pm\cap Y$ by a path in $\Omega^\pm$.

Fix either sign and let $x\in\Omega^\pm\cap N$.
Let $D\subset N$ be the disk through $x$ obtained by fixing
$t=t(x)$ in the tubular coordinates $(t,z)$ used in the
proof of \cref{lem:boundary-regularity}.
The local expressions of $H$ in
\cref{lem:exact-ray-normal-form,lem:phase-adjacent-transition,lem:ordinary-ray-cap}
are, respectively,
\[
\operatorname{Re}(z^\gamma),\qquad
\operatorname{Re}\bigl(e^{i\vartheta}b(t)z^\gamma
+z^{\gamma+1}\bigr),\qquad
\operatorname{Re}\bigl(b(t)z+z^2\bigr).
\]
If $D\cap\pi^{-1}(\Gamma)=\varnothing$, use $z$ as a
coordinate on $D$.
If $D\cap\pi^{-1}(\Gamma)\neq\varnothing$, use the
coordinate $w$ on $D$ with $z=w^2$.
In both cases, $H|_D$ is the real part of a nonconstant
holomorphic function and hence is harmonic.

The maximum principle applied to $\pm H$ implies that
the connected component of $\Omega^\pm\cap D$ containing
$x$ meets $\partial D$.
Hence $x$ can be joined within $\Omega^\pm\cap D$ to a
point of $(\Omega^\pm\cap\partial D)\subset\Omega^\pm\cap Y$.
As $\Omega^\pm\cap Y$ is connected, $\Omega^\pm$ is connected.

\textbf{Step 2:} We construct a compact connected region $C^-$ where $H$ satisfies the boundary conditions of \cref{lem:indefinite-relative-filling}.\par\nopagebreak

We first choose regular values $\pm\eps$ of $H$.
Since $H=\widetilde F_{\mathrm{loc}}$ on $N$, the local
models in
\cref{lem:phase-adjacent-transition,lem:ordinary-ray-cap}
give $H(x)=0$ whenever $x\in N$ and $dH_x=0$.
On $Y$, $0$ is a regular value of $h$, so the critical values
of $h$ are bounded away from zero.
Therefore, we can choose $\eps>0$ such that $dH\neq0$ on $\{x\in\widetilde B_R:0<|H(x)|\leq2\eps\}$.

On $\widetilde M\setminus\widetilde B_R$, we have
$H(r,\omega)=r^\mu\widetilde\phi(\omega)$.
The local expansions used in \cref{lem:exact-ray-normal-form}
show that each lift of a point in $P\cup Q_\phi$ is an
isolated critical point of $\widetilde\phi$.
Thus the zero level set of $\widetilde{\phi}$ is regular away from these points.
Therefore, after decreasing $\eps$, we have
$d(H|_{\partial\widetilde B_R})=R^\mu d\widetilde{\phi}\neq0$
 on $0<|H|\leq2\eps$.

We will modify $H$ only in compact subsets of
$\{H\geq\eps\}$ and $\{H\leq-\eps\}$.
The resulting function $\widetilde F$ will satisfy
$\widetilde F=H$ on $\{|H|<\eps\}$.
Thus the local models near $\pi^{-1}(\Gamma\cup [R_0,\infty)Q_\phi)$
will remain unchanged. Note that the sets $\{H\geq\eps\}$ and $\{H\leq-\eps\}$ in
$\widetilde B_R$ are retracts of $\Omega^+$ and $\Omega^-$,
respectively.
They are therefore connected.

Choose $L>\max_{\widetilde B_R}|H|$.
From each point $(R,\omega)$ with $H(R,\omega)\leq-\eps$,
we extend radially outward until $H=-L$.
Since $H(R,\omega)=R^\mu\widetilde\phi(\omega)$, the
condition on the starting point is equivalent to
$\widetilde\phi(\omega)\leq-\eps/R^\mu$.
Along each of these rays, $\partial_rH=\mu H/r<0$.
Thus the inequalities $r\geq R$ and $H(r,\omega)\geq-L$
specify the segment from $(R,\omega)$ to the point where
$H=-L$.
We define $C^-$ to be the union of
$\widetilde B_R\cap\{H\leq-\eps\}$ with these radial segments,
so that
\[
  C^-:=
  \bigl(\widetilde B_R\cap\{H\leq-\eps\}\bigr)
  \cup\left\{
  (r,\omega):
  r\geq R,\quad
  \widetilde\phi(\omega)\leq-\frac{\eps}{R^\mu},
  \quad H(r,\omega)\geq-L
  \right\}.
\]
Then $C^-$ is compact and connected.
The defining inequalities give
$C^-\subset\widetilde B_{R(L/\eps)^{1/\mu}}$.
Moreover, $C^-$ satisfies
\[
\widetilde B_R\cap\{H\leq-\eps\}
\subset C^-
\subset\{-L\leq H\leq-\eps\}.
\]
See \cref{fig:c-minus-radial-extension} for a two-dimensional illustration.

\begin{figure}[htbp]
  \centering
  \includegraphics[width=.80\textwidth]{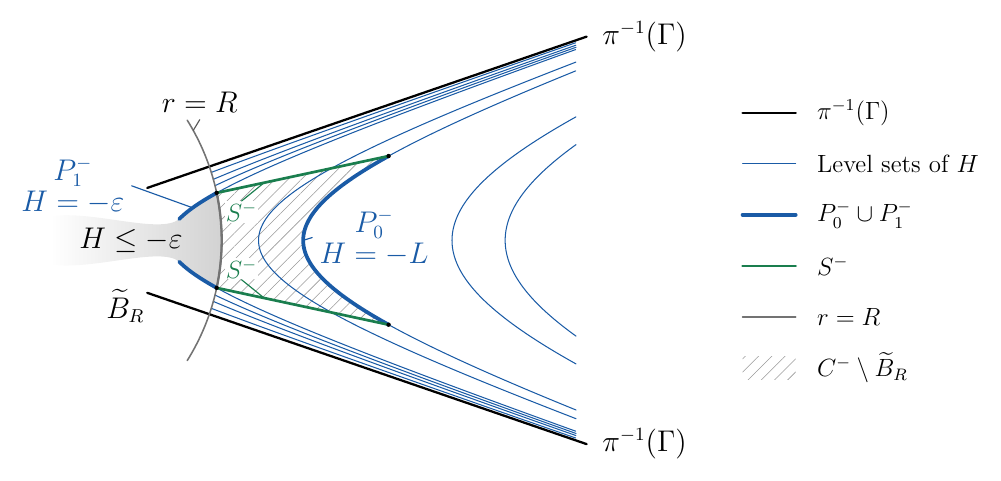}
  \caption{The radial extension defining $C^-$.}
  \label{fig:c-minus-radial-extension}
\end{figure}

Set $P_0^-:=C^-\cap\{H=-L\}$ and
$P_1^-:=\widetilde B_R\cap\{H=-\eps\}$.
Set $C^+:=\tau(C^-)$, $P_0^+:=\tau(P_1^-)$,
and $P_1^+:=\tau(P_0^-)$.
The surfaces $P_0^-$ and $P_1^-$ are nonempty by the choices
of $\eps$ and $L$.
The side boundary $S^-$ consists of the radial segments starting
at $\partial P_1^-$.
Thus $\partial C^-=P_0^-\cup S^-\cup P_1^-$.
Using $t:=H$ along these segments gives
\[
  S^-\cong\partial P_1^-\times[-L,-\eps],
\]
with $H=t$.
Set $S^+:=\tau(S^-)$.
The regularity of $H|_{\partial\widetilde B_R}$ at $-\eps$
and the nonzero radial derivative ensure that the boundary
faces meet transversely.
Thus $C^-$ is a smooth manifold with corners.

Near $P_1^-$, we have $dH\neq0$ by the choice of $\eps$.
On $P_0^-\cup S^-$,
$\partial_rH=\mu H/r<0$.
Therefore $dH\neq0$ near $\partial C^-$.

\textbf{Step 3:} Modify $H$ to obtain $\widetilde F$ satisfying (i) and (ii).\par\nopagebreak

All critical points of $H|_{C^-}$ lie in $Y$, where $H=h$ is Morse.
Apply \cref{lem:indefinite-relative-filling} to
$(C^-,P_0^-,P_1^-)$ with $f=H$, $a=-L$, $b=-\eps$,
and $S=S^-$.
This gives a Morse function $\widetilde F:C^-\to[-L,-\eps]$,
equal to $H$ near $\partial C^-$, whose critical points
have index $1$ or $2$.
Extend $\widetilde F$ to $C^+$ by $\widetilde F(x)=-\widetilde F(\tau x)$
and set $\widetilde F=H$ on $\widetilde M\setminus(C^-\cup C^+)$.
Since $\widetilde F=H$ near $\partial C^-$, the extension is smooth and odd.

We have $\widetilde F=H$ on $\{|H|<\eps\}$, so the
local models remain unchanged near the prescribed curves.
Every additional critical point lies in the interior of $C^\pm$
and is Morse of index one or two.
These points are isolated and hence finite in number, since
$C^\pm$ are compact and have no critical points near their boundaries.
Oddness gives $\operatorname{ind}_{\tau x}(\widetilde F)
=3-\operatorname{ind}_x(\widetilde F)$.

Increase $R_1$, keeping $R_0$ fixed, so that $\widetilde B_R\cup C^-\cup C^+
  \subset\pi^{-1}(B_{R_1}).$ Use this larger radius in the definition of $U$ and shrink its
tubes if necessary.
Then $\widetilde F=\widetilde F_{\mathrm{hom}}$ on $\{r\geq R_1\}$.
This proves (i) and (ii).

\textbf{Proof of (iii).}
\textbf{Step 1.} We first construct positive arcs in $C^\pm$.\par\nopagebreak
Fix one of the domains $C^\pm$. We choose an arbitrary regular point $x\in C^\pm$
and construct an arc $\gamma$ from $P_0^\pm$ to $P_1^\pm$
through $x$ with $d\widetilde F(\dot\gamma)>0$.
Choose a gradient-like vector field for $\widetilde F$
tangent to $S^\pm$.
Follow its flowline from $x$ towards increasing values
of $\widetilde F$.
If the flowline tends to a Morse critical point $p$,
modify it near $p$ as in \cref{fig:morse-gradient-modification}.
The modified arc crosses
$\widetilde F^{-1}(\widetilde F(p))$ through regular points
and satisfies $d\widetilde F(\dot\gamma)>0$.

Continue along the vector field after crossing
$\widetilde F^{-1}(\widetilde F(p))$.
Compactness of $C^\pm$ and finiteness of the critical values
imply that finitely many such modifications give an arc
from $x$ to $P_1^\pm$.
Applying the same argument to $-\widetilde F$ gives an arc
from $P_0^\pm$ to $x$ along which $\widetilde F$ increases.
Joining the two arcs gives the required $\gamma$.

\begin{figure}[htbp]
  \centering
  \includegraphics[width=.50\textwidth]
    {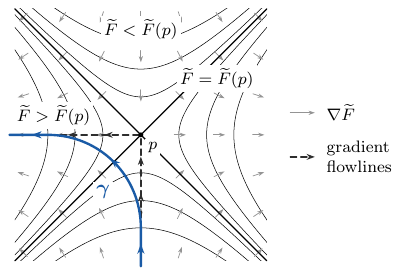}
  \caption{Modification of a gradient flowline near a Morse
  critical point $p$.}
  \label{fig:morse-gradient-modification}
\end{figure}

\textbf{Step 2.} We construct positive arcs from $P_1^-$ to $P_0^+$.\par\nopagebreak
On $\widetilde B_R\cap\{|H|\leq\eps\}$, we have
$\widetilde F=H$.
By \cref{lem:boundary-regularity,lem:connected-odd-separator},
zero is a regular value of $H|_Y$ and of the restriction
of $H$ to each boundary face and boundary circle of $Y$.
Compactness of $Y$ then gives, for sufficiently small $\eps$,
a product structure
\[
  Y\cap\{|H|\leq\eps\}
  \cong \bigl(Y\cap H^{-1}(0)\bigr)\times[-\eps,\eps],
\]
where $H(y,s)=s$ for $(y,s)\in Y\cap H^{-1}(0)\times[-\eps,\eps]$.
Thus every point of $Y\cap\{|H|\leq\eps\}$ lies on an arc
from $Y\cap P_1^-$ to $Y\cap P_0^+$ along which
$dH(\dot\gamma)>0$.

Let $x\in N$ satisfy $|H(x)|\leq\eps$ and $dH_x\neq0$.
Choose the disk $D\subset N$ through $x$ by fixing
$t=t(x)$ as in Step~1 of the proof of (i) and (ii).
Then $H|_D$ is nonconstant and harmonic.
The centre of $D$ is the only possible critical point
of $H$ in $D$.
By the maximum principle, no level set of $H|_D$
contains a simple closed curve or an arc with both ends
at the centre of $D$.
Thus $x$ can be joined to $\partial D$ within
$D\cap H^{-1}(H(x))$ by a path avoiding $\{dH=0\}$.
Since $H^{-1}(H(x))$ meets the boundary faces and
boundary circles of $Y$ transversely, we can extend
the path within $H^{-1}(H(x))\cap\{dH\neq0\}$
to a point $y\in\operatorname{int}Y$.

Take the arc $\gamma_y$ through $y$ given by the product structure on
$Y\cap\{|H|\leq\eps\}$.
If $|H(x)|<\eps$, the local modification in the proof of
\cite[Lemma~3.2]{ChenHeYanCalabiSurgery}, applied along
the path from $y$ to $x$, gives an arc $\gamma_x$ through $x$
with the same endpoints and with $dH(\dot\gamma)>0$.
See \cref{fig:positive-arc-level-transport}.
For $x\in P_1^-\cup P_0^+$, the same modification
in a boundary collar gives an arc with endpoint $x$.
Therefore for every regular point of
$\widetilde B_R\cap\{|H|\leq\eps\}$ we obtain an arc from $P_1^-$ to $P_0^+$
through it.

\begin{figure}[htbp]
  \centering
  \includegraphics[width=.80\textwidth]
    {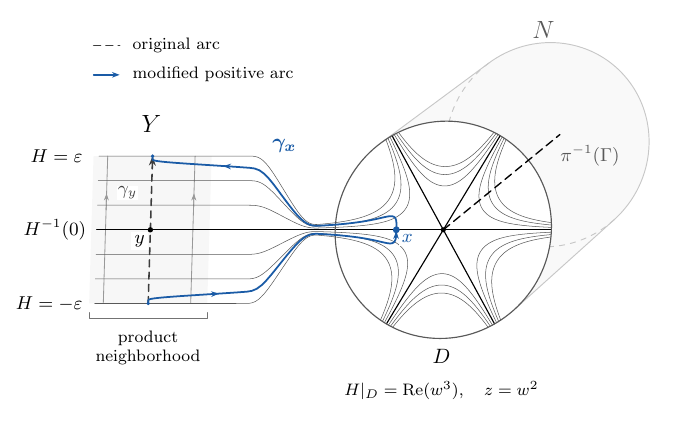}
  \caption{Modification of $\gamma_y$ to a positive arc $\gamma_x$
  through $x$ with the same endpoints, illustrated for $H(x)=H(y)=0$.
  The shaded strip is a local product neighborhood in $Y$, and $D$
  shows the lifted standard model $\operatorname{Re}(w^3)$.}
  \label{fig:positive-arc-level-transport}
\end{figure}

\textbf{Step 3.} Join the arcs from $P_0^-$ to $P_1^+$.\par\nopagebreak

Join the arcs constructed in Steps~1 and~2 across
$P_1^-$ and $P_0^+$ and we can modify the resulting piecewise smooth arc to be smooth.
Every regular point of
$\widetilde B_R\cup C^-\cup C^+$
then lies on a piecewise smooth arc from $P_0^-$ to $P_1^+$
such that $d\widetilde F(\dot\gamma)>0$

\textbf{Step 4.} Extend the arcs to infinity.\par\nopagebreak

Let $\gamma$ be an arc from Step~3, oriented from
$P_0^-$ to $P_1^+$.
Extend $\gamma$ at both endpoints to infinity
by keeping $\omega$ fixed, and smooth the corners.
By the definition of $C^\pm$, the added rays lie in
$\{r>R\}\setminus\operatorname{int}(C^-\cup C^+)$.
Hence $\widetilde F=H=\widetilde F_{\mathrm{hom}}$
on the added rays.

On the ray attached at $P_0^-$, we have $H\leq-L$ and
$d\widetilde F(-\partial_r)=-\mu H/r>0$.
Orient the ray attached at $P_0^-$ towards decreasing $r$.
On the ray attached at $P_1^+$, we have $H\geq L$ and
$d\widetilde F(\partial_r)=\mu H/r>0$.
Orient the ray attached at $P_1^+$ towards increasing $r$.
Therefore, for every regular point
$x\in\widetilde B_R\cup C^-\cup C^+$,
we have constructed a smooth arc $\gamma$ through $x$
with both ends at infinity and $d\widetilde F(\dot\gamma)>0$.

Now let $x=(r,\omega)$ be a regular point outside
$\widetilde B_R\cup C^-\cup C^+$.
Suppose first that $\widetilde\phi(\omega)\neq0$.
Let $y$ be the first point of
$\widetilde B_R\cup C^-\cup C^+$
reached from $x$ by decreasing $r$ while keeping $\omega$ fixed.
Since $\partial_r\widetilde F_{\mathrm{hom}}(y)\neq0$,
the point $y$ is regular.

Choose an arc $\gamma$ from Step~3 through $y$.
If $\widetilde\phi(\omega)>0$, restrict $\gamma$ to its
subarc from $P_0^-$ to $y$.
If $\widetilde\phi(\omega)<0$, restrict $\gamma$ to its
subarc from $y$ to $P_1^+$.
Extend $\gamma$ at both endpoints to infinity
by keeping $\omega$ fixed, and smooth the corners.
Then $\gamma$ passes through $x$ and satisfies
$d\widetilde F(\dot\gamma)>0$.

Suppose now that $\widetilde\phi(\omega)=0$.
Set $y=(R,\omega)$.
Regularity of $x$ gives $d\widetilde\phi(\omega)\neq0$.
Hence the segment from $y$ to $x$ with $\omega$ fixed
is contained in
$\widetilde F^{-1}(0)\cap\{d\widetilde F\neq0\}$.
We have shown that we can choose a smooth and positive arc $\gamma_y$ through $y$ with both ends
at infinity.
Apply the local modification in
\cite[Lemma~3.2]{ChenHeYanCalabiSurgery}
to $\gamma_y$ along the segment from $y$ to $x$.
The resulting smooth and positive arc $\gamma$ passes through $x$
and agrees with $\gamma_y$ outside a compact set.
This proves (iii).
\end{proof}

\subsection{Construction of the harmonic metric}

Fix the function $\widetilde F$ from
\cref{prop:odd-indefinite-filling}.
Near the critical set of $\widetilde F$ and sufficiently far out,
the form $\widetilde b$ will agree with the closed odd 2-forms
described below using the local harmonic models.
After descent to $\R^3$, \cref{thm:relative-metric-realization}
will give the required metric.

\subsubsection{Prescribed forms and a closed extension}

Near $\pi^{-1}(\Gamma\cup[R_0,\infty)Q_\phi)$, set
$b_{\mathrm{loc}}:=\pi^*(*_{g_{\mathrm{loc}}}dF_{\mathrm{loc}})$.
At each Morse critical point $x_j$, applying the Morse lemma, we can choose a metric $g_j$ with
$\Delta_{g_j}\widetilde F=0$, as in
\cite[Section~3.1]{ChenHeYanCalabiSurgery}.
Set $b_{\mathrm{loc}}=*_{g_j}d\widetilde F$ near $x_j$ and
define $b_{\mathrm{loc}}$ near $\tau x_j$ by
$\tau^*b_{\mathrm{loc}}=-b_{\mathrm{loc}}$.
The formulas in
\cref{lem:exact-ray-normal-form,lem:phase-adjacent-transition}
show that $b_{\mathrm{loc}}$ extends smoothly across
$\pi^{-1}(\Gamma)$.
Harmonicity gives $db_{\mathrm{loc}}=0$.

On $\{r>R_1\}$, define
\[
  \lambda_\infty:=-\frac{r^{\mu+1}}{\mu+1}
    *_{S^2}d\widetilde\phi,
  \qquad b_\infty:=d\lambda_\infty,
\]
where $*_{S^2}$ is the Hodge operator on the round sphere $S^2$,
pulled back from $S^2\setminus P$.
The equation $-\Delta_{S^2}\phi=\mu(\mu+1)\phi$ gives
$b_\infty=\pi^*(*_{g_{\Euc}}dF_{\mathrm{hom}})$.
By \cref{lem:exact-ray-normal-form}, $\lambda_\infty$
extends smoothly and oddly across
$\pi^{-1}(\Gamma)\cap\{r>R_1\}$.

Choose a sufficiently small $\tau$-invariant neighborhood
$V$ of $\Zero(d\widetilde F)$ such that each component of
$V$ and of $V\cap\{r>R_1\}$ is contractible.
Since $b_{\mathrm{loc}}=b_\infty$ on $V\cap\{r>R_1\}$,
a primitive of $b_{\mathrm{loc}}$ on $V$ differs from
$\lambda_\infty$ by an exact form on $V\cap\{r>R_1\}$.
After increasing $R_1$, we can therefore choose an odd
primitive of $b_{\mathrm{loc}}$ agreeing with $\lambda_\infty$
on $V\cap\{r>R_1\}$.

Let $\eta$ denote the odd primitive of $b_{\mathrm{loc}}$
chosen above, and extend $\eta$ by $\lambda_\infty$
to $V\cup\{r>R_1\}$.
Choose a smooth $\tau$-invariant cutoff $\chi$ supported in
$V\cup\{r>R_1\}$, with $\chi=1$ near $\Zero(d\widetilde F)$
and on $\{r\geq R_1+1\}$.
Extend $\chi\eta$ by zero to $\widetilde M$ and set
$\bar b:=d(\chi\eta)$.
Then $\bar b$ is smooth, closed, and odd.
After increasing $R_1$, we have
$\bar b=b_{\mathrm{loc}}$ near $\Zero(d\widetilde F)$
and $\bar b=b_\infty$ on $\{r\geq R_1\}$.

Consequently, the set
\[
  K_{\mathrm{bad}}
  :=\{x\in\widetilde M:d\widetilde F_x\neq0,\,
                  (d\widetilde F\wedge\bar b)_x\leq0\}
\]
is a compact subset of
$\pi^{-1}(B_{R_1})\setminus\Zero(d\widetilde F)$.

\subsubsection{Completion of the proof}

\begin{proposition}\label[proposition]{prop:remote-thom-correction}
There is a smooth closed odd 2-form $\widetilde b$ on
$\widetilde M$ such that $d\widetilde F\wedge\widetilde b>0$
whenever $d\widetilde F\neq0$,
with $\widetilde b=b_{\mathrm{loc}}$ near $\Zero(d\widetilde F)$
and $\widetilde b=b_\infty$ outside a compact set.
\end{proposition}

\begin{proof}
If $K_{\mathrm{bad}}=\varnothing$, take $\widetilde b=\bar b$.
By \cref{prop:odd-indefinite-filling}(iii) and compactness of
$K_{\mathrm{bad}}$, choose finitely many positive arcs
$\gamma_1,\ldots,\gamma_m$ with radial ends and tubular
neighborhoods covering $K_{\mathrm{bad}}$.
For $R_3\gg R_1$, join the radial ends of $\gamma_j$ in
$\{R_3<r<2R_3\}$ to form an embedded loop $L_{j,R_3}$.
Parametrize the connecting curve by
$t\mapsto(R_3a_j(t),\omega_j(t))$, $0\leq t\leq1$,
where $a_j$ and $\omega_j$ are independent of $R_3$
and $1<a_j(t)<2$.
Choose the curve to agree with the radial ends of $\gamma_j$
near its endpoints.
Require the projection of $\omega_j([0,1])$ to $S^2$
to be disjoint from $P\cup Q_\phi$.

As in \cite[proof of Theorem~3.10]{ChenHeYanCalabiSurgery},
choose Thom forms $\psi_{j,R_3}$ compactly supported in
tubes about $L_{j,R_3}$, with
$\psi_{j,R_3}|_{\{r\leq R_1\}}$ independent of $R_3$.
Require $\operatorname{supp}\psi_{j,R_3}
\subset\{r<2R_3\}\setminus\Zero(d\widetilde F)$.
On $\{R_3<r<2R_3\}$, use Euclidean normal coordinates
$(y_1,y_2)$ on disks of radius $\delta>0$ and take
$\psi_{j,R_3}=\rho_j(y_1,y_2)\,dy_1\wedge dy_2$,
where $\rho_j\geq0$ and both $\delta$ and $\rho_j$
are independent of $R_3$.
Choose $\psi_{j,R_3}$ so that
$d\widetilde F\wedge\psi_{j,R_3}\geq0$ on $\{r\leq R_3\}$
and $d\widetilde F\wedge\sum_j\psi_{j,R_3}>0$
on $K_{\mathrm{bad}}$.

Compactness of $K_{\mathrm{bad}}$ gives a constant $C>0$,
independent of $R_3$, such that
$B:=\bar b+C\sum_j\psi_{j,R_3}$ satisfies
$d\widetilde F\wedge B>0$ wherever $d\widetilde F\neq0$
on $\{r\le R_3\}$.

By construction, we have
$|\sum_j\psi_{j,R_3}|=O(1)$ on $\{R_3<r<2R_3\}$
in the lifted Euclidean metric.
Since $\omega_j$ and $\delta$ are independent of $R_3$
and the projections of $\omega_j([0,1])$ are disjoint
from $P\cup Q_\phi$, homogeneity gives
\[
\begin{aligned}
  |d\widetilde F_{\mathrm{hom}}|&\geq cR_3^{\mu-1},\\
  \frac{C|\sum_j\psi_{j,R_3}|}
       {|d\widetilde F_{\mathrm{hom}}|}
  &=O(CR_3^{1-\mu})\longrightarrow0
\end{aligned}
\]
on $\{R_3<r<2R_3\}\cap
\bigcup_j\operatorname{supp}\psi_{j,R_3}$,
where $c>0$ is independent of $R_3$ and
$\mu>1$ by \cref{lem:homogeneous-exponent}.
Outside $\bigcup_j\operatorname{supp}\psi_{j,R_3}$,
we have $B=\bar b$.
Since $\bar b=b_\infty$ on $\{r>R_3\}$, for sufficiently
large $R_3$ we have
\[
  d\widetilde F\wedge B
  \geq
  \left(1-\frac{C|\sum_j\psi_{j,R_3}|}
                   {|d\widetilde F_{\mathrm{hom}}|}\right)
  d\widetilde F_{\mathrm{hom}}\wedge b_\infty
  \geq \frac12
  d\widetilde F_{\mathrm{hom}}\wedge b_\infty>0
\]
on $\{R_3<r<2R_3\}\setminus\Zero(d\widetilde F)$.
Together with positivity on $\{r\leq R_3\}$ and
$B=b_\infty$ on $\{r\geq2R_3\}$, this gives
$d\widetilde F\wedge B>0$ wherever $d\widetilde F\neq0$.

Set $\widetilde b=(B-\tau^*B)/2$.
Then $\widetilde b$ is smooth, closed, and odd.
Since $\tau$ preserves orientation and
$\tau^*d\widetilde F=-d\widetilde F$,
\[
  d\widetilde F\wedge\widetilde b
  =\frac12\bigl(d\widetilde F\wedge B
       +\tau^*(d\widetilde F\wedge B)\bigr)>0
  \quad\text{where }d\widetilde F\neq0.
\]
The support of $\widetilde b-\bar b$ is compact and disjoint
from $\Zero(d\widetilde F)$.
Hence $\widetilde b=b_{\mathrm{loc}}$ near
$\Zero(d\widetilde F)$ and $\widetilde b=b_\infty$
outside a compact set.
\end{proof}

\begin{proof}[Proof of \cref{thm:intro-desingularization}]
Let $F$ and $b$ be the descents of
$\widetilde F$ and $\widetilde b$ to $\R^3\setminus\Gamma$. They are $\cI$-valued forms.
Set $Z=\Zero(dF)$.
By \cref{prop:remote-thom-correction}, the form $b$ is closed
and $dF\wedge b>0$ on $\R^3\setminus Z$.
Apply \cref{thm:relative-metric-realization} using
$\widetilde b=b_{\mathrm{loc}}$ near $\Zero(d\widetilde F)$
and $\widetilde b=b_\infty$ for large $r$.
We obtain a smooth metric $g$ on $\R^3$ for which $dF$ is
harmonic, with $g=g_{\mathrm{loc}}$ near $\Gamma$ and
$g=g_{\Euc}$ for large $r$.
In particular, $g$ is complete.
By \cref{prop:odd-indefinite-filling}(i), we also have
$F=F_{\mathrm{hom}}$ for large $r$.

Near $\Gamma$,
\cref{lem:exact-ray-normal-form,lem:phase-adjacent-transition}
give $|dF|_g=O(|z|^{1/2})$ and
$|\nabla dF|_g=O(|z|^{-1/2})$, uniformly on compact sets.
Together with smoothness of $F$ and $g$ across
$Z\setminus\Gamma$, these estimates verify
Definition~\ref{def:z2-harmonic}.
The definition of $\cI$ gives $\Mono(\cI)=\Gamma$.
The curves $\gamma_a$ realize the pairing $\Pi$ by construction.
\end{proof}

\section{\texorpdfstring{Nondegenerate $\Ztwo$-harmonic 1-forms on closed three-manifolds}{Nondegenerate Z/2-harmonic 1-forms on closed three-manifolds}}\label{sec:sphere}

We now prove \cref{thm:intro-sphere}.
We first construct two-valued functions on $3$-balls whose singular sets are prescribed knots.
We insert two such functions near the extrema of a Morse function on $M$
and obtain a transitive two-valued closed 1-form.
Finally we apply \cite[Theorem~3.11]{ChenHeYanCalabiSurgery} to construct a metric under which the $1$-form is harmonic.
We conclude by discussing the topology of these examples on $S^3$
and the failure of the transverse distance condition.

\subsection{A local model on a ball}

Let $K$ be a smooth knot contained in a closed $3$-ball $B$. Let
$\pi:\hat{B}\to B$ be the double branched covering that is branched along $K$, and
write $\tau$ for its involution.
For simplicity we still denote the preimage $\pi^{-1}(K)$ by $K$.
Since $\partial B$ is simply connected, $\pi^{-1}(\partial B)$ consists
of two connected components $\partial_+\hat{B}$ and $\partial_-\hat{B}$, which are exchanged by $\tau$.
Choose a normal framing of $K$ in $B$ that lifts to $K\subset\hat{B}$ and hence define coordinates $(\theta,z)$ in a tubular neighborhood of $K$ in $B$.
We can lift this framing to $K\subset\hat{B}$, and define corresponding product coordinates
$(\theta,w)\in S^1\times D^2$ in which $\tau(\theta,w)=(\theta,-w)$
and $\pi(\theta, w)=(\theta,z=w^2)$. We now construct a smooth odd function on $\hat{B}$, 
which gives a two-valued function on $B$. By odd function, we mean a function $u$ on $\hat{B}$ satisfying $u\circ\tau=-u$.

\begin{lemma}\label[lemma]{lem:sphere-ball-function}
There is a smooth odd function $u:\hat{B}\to[-1,1]$ such that $u=\operatorname{Re}(w^3)$ near $K$, and $u(y,s)=\pm(1-s)$ near $\partial_\pm\hat{B}$,
where $(y,s)\in S^2\times [0,\epsilon)$ are collar coordinates near $\partial_\pm \hat{B}$ that are preserved by $\tau$, with $s$ increasing inward.
The critical set of $u$ consists of $K$ and finitely many Morse critical
points of index $1$ or $2$.

Moreover, for every regular point $x\in \mathrm{Int}\hat{B}$ of $u$ and every
$y_\pm\in\partial_\pm\hat{B}$, there is a smooth embedded arc $\gamma$
from $y_-$ to $y_+$ through $x$ satisfying $du(\dot \gamma)>0$.
The arc $\gamma$ can be chosen so that in the collar coordinates $(y,s)$, its spherical coordinate $y$ is independent of $s$ sufficiently close to its endpoints.
\end{lemma}
\begin{proof}
Choose a tubular neighborhood $N=\{|w|\leq\rho\}\subset\hat{B}$ of $K$
with $0<\rho<1$
and set $Y=\hat{B}\setminus\interior N$.
$Y$ is connected, and $\tau$ acts freely on it. We first define $u$ near $\partial Y$.
Let $u=\operatorname{Re}(w^3)$ near $\partial N$
and $u=\pm(1-s)$ near $\partial_\pm\hat{B}$.
Note that $du\ne 0$ near $\partial Y$, and $0$ is a regular value of $u|_{\partial Y}$.
Apply the proof of \cref{lem:connected-odd-separator},
choosing $|u|<1$ on $\operatorname{int}Y$ when extending $u$.
All perturbations have compact support in $\operatorname{int}Y$
and can be chosen sufficiently small to preserve $|u|<1$.
We obtain an odd Morse extension $u:Y\to[-1,1]$ with zero
as a regular value and with $\{u>0\}$ and $\{u<0\}$ connected.

We extend $u$ to the entire $\hat{B}$ by setting $u=\operatorname{Re}(w^3)$ on $N$.
Each component of $\{\operatorname{Re}(w^3)>0\}\cap N$ meets
$\{x\in Y:u(x)>0\}$ along $\partial N$.
The same holds for the negative components.
Consequently, on the entire $\hat{B}$, the regions $\{u>0\}$ and $\{u<0\}$ are still connected.
The critical set of $u$ intersects $u^{-1}(0)$ precisely in $K$.

On $N\setminus K$, $du=d\operatorname{Re}(w^3)\neq0$.
Since $Y$ is compact and $du\neq0$ on $Y\cap u^{-1}(0)$,
we can choose $\eps>0$ sufficiently small that $du\neq0$ when $0<|u|\leq2\eps$.
After decreasing $\eps$ if necessary, $Y\cap\{|u|\leq\eps\}$
has a product structure with $u$ as one coordinate taking values in
$[-\eps,\eps]$.

Set $C^\pm:=\{\pm u\geq\eps\}$.
Set $P_0^-:=\partial_-\hat B$ and $P_1^-:=u^{-1}(-\eps)$.
Set $P_0^+:=u^{-1}(\eps)$ and $P_1^+:=\partial_+\hat B$.
Flowing along the gradient field of $u$ gives a retraction of $\{u>0\}$ onto $C^+$.
Since $\{u>0\}$ is connected, so is $C^+$.
Oddness of $u$ gives $C^- =\tau(C^+)$, so $C^-$ is connected as well.

We have $\partial C^+=P_0^+\sqcup P_1^+$.
Since $u|_Y$ is Morse and $du\neq0$ on $N\setminus K$,
the restriction $u|_{C^+}:C^+\to[\eps,1]$ is Morse.
The choice of $\eps$ and the identity $u=1-s$ near
$\partial_+\hat{B}$ imply $du\neq0$ near $\partial C^+$.
Apply \cref{lem:indefinite-relative-filling} to
$(C^+,P_0^+,P_1^+)$ with $f=u$, $a=\eps$, $b=1$,
and $S=\varnothing$ to modify $u|_{C^+}$ while keeping $u$
fixed near $\partial C^+$.
After the modification, all critical points of $u|_{C^+}$ have index $1$ or $2$.
Extend the modification to $C^-$ by oddness and keep $u$
fixed on $\hat{B}\setminus(C^+\cup C^-)$.
Continue to denote the resulting smooth odd function by $u$.
The function $u$ is unchanged on $\{|u|\leq\eps\}$.
In particular, $u=\operatorname{Re}(w^3)$ on
$N\cap\{|u|\leq\eps\}$.

We now prove the second assertion about positive arcs. Consider the double cover of $S^3$ branched along $K\subset B\subset S^3$. 
It is a classical fact that this double cover is a rational homology sphere. Consequently, $b_1(\hat{B})=0$. 
Since $b_1(\hat{B})=0$, the Reeb graph of $u$ is a tree. Here, the Reeb graph is defined to be the space of 
connected components of level sets of $u$ with the quotient topology. 
Since $u$ has no interior local extrema, the only vertices of this tree correspond to $\partial_\pm\hat{B}$, so the Reeb graph of $u$ 
is identified with $[-1,1]$.

Let $x$ be a regular point of $u$ with $u(x)=0$. Consider the gradient flowline of $u$ passing through $x$.
This flowline intersects $u^{-1}(0)$ once exactly at $x$. So if it tends to a critical point of $u$, that point must be a Morse critical point of
index $1$ or $2$. Thus we can modify the flowline near the critical points to 
cross the level sets containing critical points, as in the proof of \cref{prop:odd-indefinite-filling}(iii).
This defines a positive arc $\gamma$ satisfying the assertions of the lemma.

Now since the Reeb graph of $u$ is a segment, $\gamma$ passes through every level set of $u$ exactly once. 
For any regular point $x$ of $u$ in the interior, applying \cite[Lemma~3.2]{ChenHeYanCalabiSurgery}, 
we can modify $\gamma$ to pass through $x$ while preserving the strict increase of $u$ along $\gamma$. 
The same argument in the boundary collars allows us to prescribe the endpoints $y_\pm$ 
and make the spherical coordinate constant near each endpoint.
\end{proof}

\subsection{Construction of the harmonic 1-form}\label{subsec:sphere-closed-form}
Let $M$ be a closed connected oriented three-manifold.
By classical Morse theory, there is a Morse function $f\to\mathbb R$
with a unique local minimum $x_0$ and a unique local maximum $x_1$.

Choose sufficiently small disjoint closed $3$-balls $B_0,B_1\subset M$
centered at $x_0,x_1$ in Morse coordinates, so that their boundaries
are regular level sets of $f$, and set
$W=M\setminus\operatorname{Int}(B_0\sqcup B_1)$.
After an affine rescaling of $f$, we may assume that
$f|_{\partial B_0}=0$ and $f|_{\partial B_1}=1$.
Near each $\partial B_i$, choose coordinates $(y,s)$, with
$y\in\partial B_i$ and $s>0$ inside $B_i$, such that
$f=-s$ near $\partial B_0$ and $f=1+s$ near $\partial B_1$.

Since all critical points of $f$ in $W$ have index $1$ or $2$, every regular 
point $x\in\mathrm{Int}W$ lies on a smooth arc $\gamma$ from any prescribed $y_0\in\partial B_0$ 
to $y_1\in\partial B_1$ with $df(\dot\gamma)>0$.
The arc can be chosen with constant sphere coordinate near each endpoint.
If $x\in\partial B_i$, take $y_i=x$.

Choose any smooth knot $K_i\subset\interior B_i$ for $i=0,1$,
and take the covers $\hat{B}_i$ and functions $u_i$ from
\cref{lem:sphere-ball-function}.
Use the lift of $s$ as the collar coordinate on each $\hat{B}_i$.
Over $W$, take the trivial double cover $W_+\sqcup W_-$.
Along $\partial B_0$, glue $\partial_+\hat{B}_0$ and
$\partial_-\hat{B}_0$ to the corresponding boundary components of
$W_+$ and $W_-$, respectively.
Along $\partial B_1$, glue $\partial_-\hat{B}_1$
to $W_+$ and $\partial_+\hat{B}_1$ to $W_-$.
This defines a double branched cover $\pi:\widehat M\to M$ branched along
$L=K_0\sqcup K_1$. See \cref{fig:morse-knot-gluing}.
Its involution $\tau$ exchanges $W_+$ and $W_-$ and agrees with involutions on $\hat{B}_0$ and $\hat{B}_1$.
Since $u_i=\pm(1-s)$ near $\partial_\pm\hat{B}_i$, while
$f=-s$ near $\partial B_0$ and $f=1+s$ near $\partial B_1$,
the forms $du_0$ and $du_1$ glue to $df$ on $W_+$ and to $-df$ on $W_-$.
This gives a smooth closed 1-form $\tilde{\alpha}$ on $\widehat M$ and satisfies
$\tau^*\tilde{\alpha}=-\tilde{\alpha}$.

\begin{figure}[htbp]
  \centering
  \includegraphics[width=.50\textwidth]{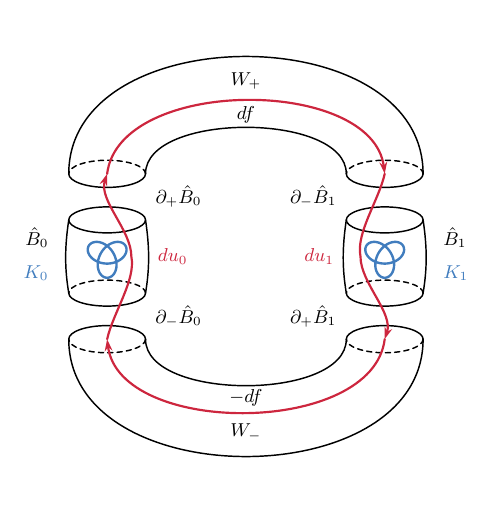}
  \caption{Gluing construction of $\widehat M$ and a positive loop.}
  \label{fig:morse-knot-gluing}
\end{figure}

Let $\cI\to M\setminus L$ be the flat real line bundle 
associated to the double cover $\pi:\widehat M\setminus\pi^{-1}(L)\to M\setminus L$, 
which is considered as a principal bundle with fiber $\mathbb{Z}/2$. The holonomy of $\cI$ around each meridian of $L$ is $-1$.
The form $\tilde{\alpha}$ descends to a closed
$\cI$-valued 1-form $\alpha$.
By construction of $u_0$ and $u_1$, we have $\alpha=d\operatorname{Re}(z^{3/2})$ near $L$, 
where $z$ is the normal complex coordinate near $L$ in $B_0\cup B_1$.
The zero set of $\tilde{\alpha}$ outside $L$ is a finite set $Q$ of Morse
zeros of index $1$ or $2$.

Now we verify the \emph{transitvity} of $\alpha$. Let $x\in M$ be a nonzero point of $\alpha$.
Lift this point to a point $\hat{x}\in\widehat{M}$. In any of the four pieces $W_+,W_-,\hat{B}_0,\hat{B}_1$ containing $\hat{x}$,
 we can find a smooth embedded arc through the lift of $x$ with $\tilde{\alpha}(\dot{\gamma})>0$,
with arbitrary prescribed endpoints of the boundary of that piece. On the remaining three pieces, 
we can find smooth embedded arcs with $\tilde{\alpha}(\dot{\gamma})>0$ and with endpoints matching the previous arc. 
Thus this forms a smooth embedded loop $\hat{\gamma}$ through $\hat{x}$ with $\tilde{\alpha}(\dot{\gamma})>0$. 
It desecends to a loop $\gamma$ through $x$ with $|\alpha(\dot{\gamma})|>0$. A small perturbation makes this loop be a
embedded one while fixing a short segment through $x$. This shows that $\alpha$ is transitive. See \cref{fig:morse-knot-gluing}

\begin{proof}[Proof of \cref{thm:intro-sphere}]
Near $\pi(Q)$, choose a metric $g_0$ making $\alpha$ harmonic as at
the ordinary Morse zeros in \cref{sec:desingularization}.
On small balls around these points, $\star_{g_0}\alpha$ is exact
by the Poincar\'e lemma.
Near $L$, choose the product metric $g_0=d\theta^2+|dz|^2$,
with orientation
$d(\operatorname{Re}z)\wedge d(\operatorname{Im}z)\wedge d\theta$.
Then we have $\star_{g_0}\alpha
  =d\bigl(\operatorname{Im}(z^{3/2})\,d\theta\bigr)$ exact.
The primitive $\bigl(\operatorname{Im}(z^{3/2})\,d\theta\bigr)$ is a $\cI$-valued 1-form on the 
tubular neighborhood.
These metrics near $\pi(Q)$ and $L$ define $g_0$ near $Z$ for which $\alpha$
is harmonic and $\star_{g_0}\alpha$ is exact. Thus
$\alpha$ is locally intrinsically harmonic and locally
$\star$-exact.
We have shown that $\alpha$ is transitive, so we can apply
\cite[Theorem~3.11]{ChenHeYanCalabiSurgery} to $\alpha$. 
Therefore, there exists a smooth metric $g$ on $M$ such that $\alpha$ is harmonic with respect to $g$.
Moreover, by construction $\alpha$ is nondegenerate.
\end{proof}

The above construction can be generalized to produce a metric together with a nondegenerate
$\Ztwo$-harmonic 1-form on every closed connected oriented
$n$-manifold $(n\geq3)$. The singular set is two embedded $(n-2)$-spheres that are contained in two disjoint $n$-balls
and the $1$-form has ordinary Morse zeros of indices
$1,\ldots,n-1$.

\subsection{\texorpdfstring{The case of $S^3$}{The case of S3}}\label{subsec:sphere-s3}

We now discuss two aspects of the case $M=S^3$.
We first describe the corresponding double branched covers and explain why the existence
theorem does not extend to arbitrary links with two components.
We then show that these examples fail the transverse distance condition in
He--Wentworth--Zhang's nonexistence theorem
\cite[Theorem~1.7(2)]{HeWentworthZhangTrees}, and serves as counterexamples of \cite[Conjecture~1.6]{HeWentworthZhangTrees}.

Since $K_0$ and $K_1$ lie in disjoint balls,
we have $\widehat{S^3}\cong
\Sigma_2(K_0)\#\Sigma_2(K_1)\#(S^1\times S^2)$,
where $\Sigma_2(K_i)$ denotes the double cover of $S^3$ branched
along $K_i$.
Each $\Sigma_2(K_i)$ is a rational homology sphere, so
$H^1(\widehat{S^3};\R)\cong\R$.
The positive loop $\hat\gamma\subset\widehat{S^3}$ constructed above satisfies $\int_{\hat\gamma}\tilde\alpha>0$, so $[\tilde{\alpha}]\ne 0$.

The existence statement does not extend to arbitrary links
with two components.
It was proved in \cite{HaydysPSL2R} that
a nonzero $\Ztwo$-harmonic 1-form with a smooth monodromy set
requires the corresponding double branched cover to have
positive first Betti number.
For the Hopf link, the corresponding branched cover is $\mathbb{RP}^3$, whose first
Betti number is zero.
Thus the Hopf link cannot occur as the monodromy set of
a nonzero $\Ztwo$-harmonic 1-form for any smooth metric on $S^3$.

When $M=S^3$, \cref{thm:intro-sphere} gives counterexamples to
\cite[Conjecture~1.6]{HeWentworthZhangTrees}.
These examples fail the transverse distance condition in
\cite[Theorem~1.7(2)]{HeWentworthZhangTrees}.

Following \cite{HeWentworthZhangTrees},
for a piecewise $C^1$ path $\eta:[0,1]\to S^3$, define its
transverse measure by
$\mu_\alpha(\eta)
  =\int_0^1|\alpha_{\eta(t)}(\dot\eta(t))|\,dt$.
For $x,y\in S^3$, the transverse distance is
$d_{S^3,\alpha}(x,y)=\inf_{\eta:x\to y}\mu_\alpha(\eta)$,
where the infimum is over all piecewise $C^1$ paths from $x$ to $y$.
The transverse distance condition requires every arc $\eta$ in $S^3\setminus Z$
transverse to $\ker\alpha$, with endpoints $x,y$, to satisfy
$\mu_\alpha(\eta)=d_{S^3,\alpha}(x,y)$.

Recall that the Reeb graph of $u$ on $\hat B$ is the segment
$[-1,1]$.
Since $u\circ\tau=-u$, the involution induces the reflection
$t\mapsto-t$ on this segment.
The quotient is $[0,1]$, parametrized by $|u|$.
In the construction on $S^3$, we may choose $f$ to be a height
function with exactly two critical points, its minimum and maximum.
The two quotient segments from $B_0$ and $B_1$ then glue to the Reeb interval of
$f|_W$.
Thus the leaf space of $\alpha$, defined in
\cite[Definition~4.3]{HeWentworthZhangTrees}, is a segment $I$.
Its endpoints correspond to the singular leaves containing
$K_0$ and $K_1$, respectively.

Choose $x\in\operatorname{Int}W$ and take the positive embedded loop
through $x$.
The level set of $f$ containing $x$, which is a $2$-sphere, meets this loop at one further
point $y$. Thus this level set separates the loop into two arcs $\eta_1$ and $\eta_2$, that are transverse to $\ker\alpha$.
Write $[x]\in I$ for the image of $x$ in the leaf space.
The projection of $\eta_1$ runs from $[x]$ to one endpoint of $I$
and back to $[x]$.
Thus $\mu_{\alpha}(\eta_1)$ is twice the distance from
$[x]$ to that endpoint and is therefore positive.
However, $x$ and $y$ can be joined within the level set by a
path of zero transverse measure. Thus $d_{S^3,\alpha}(x,y)=0<\mu_{\alpha}(\eta_1)$, which shows that the transverse distance condition fails.

\appendix
\section{Irregular degeneracy along a smooth branch circle}
\label{sec:prescribed-degeneracy}
Even when the monodromy locus is a smooth circle, the leading branch
coefficient of a $\Ztwo$-harmonic 1-form can have an arbitrarily
complicated zero set. More precisely, for any closed subset
$K\subset S^1$, we construct a local exact $\Ztwo$-harmonic 1-form
near the branch circle whose leading branch coefficient vanishes
precisely on $K$.

\begin{theorem}\label{thm:intro-prescribed-degeneracy}
Let $K\subset S^1$ be closed.
There exist a neighborhood $U$ of
$\Sigma_0=S^1\times\{0\}$ in $S^1_t\times\C_z$,
a smooth metric $g$ on $U$, and a $\Ztwo$-harmonic 1-form
\[
  \alpha\in\Omega^1(U\setminus Z;\cI)
\]
whose monodromy locus is $\Sigma_0$.
The form can be chosen as
\[
  \alpha=d\operatorname{Re}
  \bigl(B(t)z^{3/2}+z^{5/2}\bigr),
  \qquad B^{-1}(0)=K,
\]
where $B\in \mathcal{C}^\infty(S^1;[0,\infty))$.
\end{theorem}

\subsection{Motivation}

The construction below is motivated by the obstruction problem studied in
\cite{HeMazzeoTakahashiHaydys}.  Let $(Z,\cI,\alpha)$ be a
$\Ztwo$-harmonic 1-form on $X$, and suppose that $\Mono(\cI)$ is
a smooth curve.  Near each $p\in\Mono(\cI)$, choose adapted coordinates
$(z,t)$, where $z$ is a complex coordinate in the normal plane and $t$
is a coordinate along $\Mono(\cI)$.
Donaldson's local expansion \cite{Donaldson} gives a local primitive $f$ of
$\alpha$ satisfying
\[
  \alpha=df,
  \qquad
  f(z,t)=\operatorname{Re}\bigl(B(t)z^{3/2}\bigr)+O(r^{5/2}),
  \qquad r=|z|.
\]
The coefficient $B$ is the leading branch coefficient;
although its local representative depends on the adapted coordinates, its
zero set is well defined.  The expansion implies
\[
  B(p)\neq0
  \quad\Longleftrightarrow\quad
  \Freq_p(0)=\frac12,
  \qquad
  B(p)=0
  \quad\Longleftrightarrow\quad
  \Freq_p(0)\geq\frac32.
\]

A Kuranishi description of the moduli space is obtained in
\cite{HeMazzeoTakahashiHaydys} under the assumption that $B$ has only
finitely many simple zeros.  This raises the question whether the hypothesis
follows from smoothness of the monodromy locus.  The construction below
gives a negative answer even locally: when $\Mono(\cI)\simeq S^1$,
the set $B^{-1}(0)$ can be any closed subset $K\subset S^1$.

\subsection{Construction of a local \texorpdfstring{$\Ztwo$}{Z/2}-harmonic
1-form}

We construct a local exact $\Ztwo$-harmonic 1-form by prescribing its
leading branch coefficient and choosing a metric for which a function with two angular modes is harmonic.  

\begin{lemma}\label[lemma]{lem:flat-defining-function}
Let $K\subset S^1$ be a proper closed subset.  There is a smooth
function $B:S^1\to[0,\infty)$ such that
\[
  B^{-1}(0)=K,
  \qquad
  \partial_t^mB|_K=0
  \quad\text{for every }m\geq0.
\]
\end{lemma}

\begin{proof}
When $K=S^1$, we can choose $B=0$ and when $K=\emptyset$, we can choose $B=1$. Now we assume $K\neq \emptyset$ or $S^1$. 	
	
Write the complement as a disjoint union of open intervals
\[
  S^1\setminus K=\bigsqcup_{j\geq1}J_j,
  \qquad
  J_j=(a_j,b_j),
  \qquad
  \ell_j=b_j-a_j.
\]
If there are only finitely many intervals, the same construction is used with
a finite index set.  Let
\[
  \psi(s)=
  \begin{cases}
    \exp\bigl(-1/(s(1-s))\bigr),&0<s<1,\\
    0,&s\notin(0,1),
  \end{cases}
\]
and put $\psi_j(t)=\psi((t-a_j)/\ell_j)$ on $J_j$, extended by zero.
Choose positive numbers $\eps_j$ so small that
\begin{equation}\label{eq:flat-amplitude-choice}
  \|\eps_j\psi_j\|_{\mathcal{C}^m(S^1)}
  \leq 2^{-j}\ell_j^j,
  \qquad 0\leq m\leq j.
\end{equation}
Define $B=\eps_j\psi_j$ on $J_j$ and $B=0$ on $K$.

By construction, $B>0$ on $S^1\setminus K$, and hence $B^{-1}(0)=K$.
For each fixed $j$, every derivative of the zero extension of $\psi_j$
vanishes at $a_j$ and $b_j$.

Let $t_n\in S^1\setminus K$ converge to $t\in K$.  After passing to a
subsequence, either all $t_n$ lie in one fixed interval $J_j$, or
$t_n\in J_{j_n}$ for pairwise distinct intervals $J_{j_n}$.  In the first
case, $t$ is either $a_j$ or $b_j$.  Since every derivative of the zero
extension of $\psi_j$ vanishes at both endpoints,
$\partial_t^mB(t_n)\to0$ for every $m$.  In the second case,
$j_n\to\infty$ and, since the intervals are pairwise disjoint,
$\ell_{j_n}\to0$.  Thus, for each fixed $m$ and all sufficiently large $n$,
\eqref{eq:flat-amplitude-choice} gives
\[
  |\partial_t^mB(t_n)|
  \leq 2^{-j_n}\ell_{j_n}^{j_n}
  \longrightarrow0.
\]
Therefore, $B$ extends smoothly across $K$ and that
$\partial_t^mB|_K=0$ for every $m\geq0$.
\end{proof}

We next record the identity that produces the required metric correction.
Write $z=x+iy$.  For $s=a+ib\in\C$, define the real symmetric matrix
\begin{equation}\label{eq:Hs-definition}
  H_s(x,y)=
  \begin{pmatrix}
    -ax-2by & \frac12(bx+ay)\\[1mm]
    \frac12(bx+ay) & -2ax-by
  \end{pmatrix}.
\end{equation}
It vanishes to first order at $z=0$.

\begin{lemma}\label[lemma]{lem:mode-lowering}
For $\gamma\in\frac12\Z$, we have
\begin{equation}\label{eq:mode-lowering}
  \diver_z\bigl(H_s\nabla_z z^\gamma\bigr)
  =\gamma\left(\gamma-\frac32\right)s z^{\gamma-1}.
\end{equation}
In particular,
\begin{equation}\label{eq:two-special-modes}
  \diver_z(H_s\nabla_z z^{3/2})=0,
  \qquad
  \diver_z(H_s\nabla_z z^{5/2})=\frac52s z^{3/2}.
\end{equation}
\end{lemma}

\begin{proof}
Write $H_s=(h_{ij})$.
Since $\partial_xz^\gamma=\gamma z^{\gamma-1}$ and
$\partial_yz^\gamma=i\gamma z^{\gamma-1}$, direct differentiation gives
\begin{align*}
  \diver_z(H_s\nabla_z z^\gamma)
  &=\gamma(\gamma-1)(h_{11}-h_{22}+2ih_{12})z^{\gamma-2}\\
  &\quad
    +\gamma\bigl(
      \partial_xh_{11}+\partial_yh_{12}
      +i\partial_xh_{12}+i\partial_yh_{22}
    \bigr)z^{\gamma-1}\\
  &=\gamma(\gamma-1)\bigl(ax-by+i(bx+ay)\bigr)z^{\gamma-2}
    -\frac \gamma2(a+ib)z^{\gamma-1}\\
  &=\gamma(\gamma-1)sz^{\gamma-1}
    -\frac \gamma2sz^{\gamma-1}\\
  &=\gamma\left(\gamma-\frac32\right)sz^{\gamma-1}.
\end{align*}
\end{proof}

We now use the mode-lowering identity to construct the exact local metric.
Let
\[
  U=S^1_t\times D_\rho,
  \qquad
  Z_0=S^1\times\{0\},
\]
and let $\cI\to U\setminus Z_0$ have holonomy $-1$ around a normal
meridian and trivial holonomy in the $S^1$-direction.  The half-integral
powers below are sections of $\cI$.  Their real parts are anti-invariant under
the deck involution and therefore define $\cI$-valued functions.

\begin{proposition}
  \label[proposition]{thm:local-two-mode-model}
Let $B\in \mathcal{C}^\infty(S^1;\C)$ and $C\in\C^*$, and set
\[
  u=\operatorname{Re}\bigl(B(t)z^{3/2}+Cz^{5/2}\bigr),
  \qquad
  \alpha=du,
\]
After decreasing $\rho$, there is a smooth metric $g_B$ on $U$ such that
\[
  d\alpha=0,
  \qquad
  d_{g_B}^*\alpha=0
  \quad\text{on }U\setminus Z_0,
\]
and
\[
  \left.g_B\right|_{Z_0}=dt^2+dx^2+dy^2.
\]
\end{proposition}

\begin{proof}
Set
\begin{equation}\label{eq:sB-definition}
  s(t)=-\frac{2B''(t)}{5C}
\end{equation}
and, relative to the coordinate volume form $dt\wedge dx\wedge dy$, define
the contravariant metric density
\begin{equation}\label{eq:AB-definition}
  A_B=
  \begin{pmatrix}
    1&0\\
    0&\mathrm{I}_2+H_{s(t)}(x,y)
  \end{pmatrix}.
\end{equation}
Because $s$ is bounded and
$|H_s(x,y)|\leq c|s||z|$, a uniform choice of sufficiently small $\rho$
makes $A_B$ positive definite.  Define
\begin{equation}\label{eq:gB-definition}
  g_B=(\det A_B)A_B^{-1}.
\end{equation}
Since $A_B|_{Z_0}=\mathrm{I}_3$, we have
\[
  \left.g_B\right|_{Z_0}=dt^2+dx^2+dy^2.
\]

The definition of $g_B$ and the block form of $A_B$ give
\[
  \begin{aligned}
  \Delta_{g_B}u
  &=\frac{1}{\det A_B}\diver(A_B\nabla u)\\
  &=\frac{1}{\det A_B}\left(
      \partial_t^2u+\Delta_z u
      +\diver_z\bigl(H_{s(t)}\nabla_z u\bigr)
    \right).
  \end{aligned}
\]
By \cref{lem:mode-lowering} and \eqref{eq:sB-definition},
\begin{align*}
  \diver(A_B\nabla u)
  &=\operatorname{Re}\left[
      \left(B''(t)+\frac52C s(t)\right)z^{3/2}
    \right]\\
  &=0.
\end{align*}
Consequently, $\Delta_{g_B}u=0$.
\end{proof}

\subsection{Zero set, frequency, and analytic estimates}

We now determine the zero set and frequency of the local model and verify the
remaining growth estimate.  Let $B$ be the function constructed in
\cref{lem:flat-defining-function}.  Thus
\[
  B\geq0,
  \qquad
  B^{-1}(0)=K,
  \qquad
  \partial_t^mB|_K=0
  \quad\text{for every }m\geq0,
\]
and in particular $B(t)>0$ precisely when $t\in S^1\setminus K$.

\begin{proposition}
  \label[proposition]{prop:closed-set-frequency}
Let $C>0$, and let $\alpha=du$ be the model constructed in
\cref{thm:local-two-mode-model}.  Then
\begin{equation}\label{eq:two-mode-zero-set}
  \Zero(\alpha)=Z_0\cup\Gamma_B,
  \qquad
  \Gamma_B:=
  \left\{(t,z):z=-\frac{3B(t)}{5C}\right\}.
\end{equation}
Moreover,
\[
  \Gamma_B\cap Z_0=K\times\{0\},
  \qquad
  \Mono(\cI)=Z_0,
\]
and
\begin{equation}\label{eq:prescribed-frequency-values}
  \Freq_{(t,0)}(0)=
  \begin{cases}
    \frac12,&t\notin K,\\[2mm]
    \frac32,&t\in K.
  \end{cases}
\end{equation}
For every $p\in\Gamma_B\setminus Z_0$, we have $\Freq_p(0)=1$.
\end{proposition}

\begin{proof}
A direct calculation gives
\[
  du=0
  \quad\Longleftrightarrow\quad
  z=0
  \quad\text{or}\quad
  z=-\frac{3B(t)}{5C}.
\]
This proves \eqref{eq:two-mode-zero-set}.  Since $B^{-1}(0)=K$, it also
gives $\Gamma_B\cap Z_0=K\times\{0\}$.

By construction, $\cI$ has holonomy $-1$ around every meridian of $Z_0$ and
is trivial near every point of $\Gamma_B\setminus Z_0$.  Hence
$\Mono(\cI)=Z_0$.  The values in \eqref{eq:prescribed-frequency-values}
follow directly from the leading terms
$B(t_0)\operatorname{Re}(z^{3/2})$ for $t_0\notin K$ and
$C\operatorname{Re}(z^{5/2})$ for $t_0\in K$.  

Finally, let $p=(t_0,z_c)\in\Gamma_B\setminus Z_0$, where
\[
  z_c=-\frac{3B(t_0)}{5C}\neq0.
\]
A direct calculation gives
\[
\left.
\partial_z^2\bigl(B(t_0)z^{3/2}+Cz^{5/2}\bigr)
\right|_{z=z_c}
=-\frac32 B(t_0)z_c^{-1/2}\neq0.
\]
Thus $\alpha=du$ vanishes to first order at $p$. Since the frequency of a
smooth harmonic 1-form agrees with its order of vanishing, we obtain $\Freq_p(0)=1.$
\end{proof}

A direct computation shows that there is a constant $c>0$ such that, for
every $p\in Z_0$ and all sufficiently small $r>0$,
\[
  \int_{B_r(p)}|\alpha|_{g_B}^2\,\vol_{g_B}
  \leq cr^4.
\]
The corresponding conditions at the ordinary zeros in
$\Gamma_B\setminus Z_0$ are automatic.  Consequently,
\[
\bigl(\Zero(\alpha),\cI|_{U\setminus\Zero(\alpha)},\alpha\bigr)
\]
is a $\Ztwo$-harmonic 1-form on $(U,g_B)$.

\begin{proof}[Proof of \cref{thm:intro-prescribed-degeneracy}]
The result follows by combining \cref{lem:flat-defining-function},
\cref{thm:local-two-mode-model}, and \cref{prop:closed-set-frequency}.
\end{proof}

\bibliographystyle{plain} 
\bibliography{paper}

\end{document}